\documentclass[10pt,a4paper]{amsart}
\usepackage[utf8]{inputenc}
\usepackage{amsmath}
\usepackage{amsfonts}
\usepackage{amssymb}
\usepackage[left=2cm,right=2cm,top=2cm,bottom=2cm]{geometry}
\usepackage{bbm}
\usepackage{mathrsfs}
\usepackage[all]{xy}
\usepackage{bm}
\usepackage{manfnt}
\usepackage{tikz}
\usepackage[backref=page]{hyperref}
\author{Geoffrey Powell}
\title{Operads, modules over walled Brauer categories, and Koszul complexes}
\address{Univ Angers, CNRS, LAREMA, SFR MATHSTIC, F-49000 Angers, France}
\email{Geoffrey.Powell@math.cnrs.fr}
\urladdr{https://math.univ-angers.fr/~powell/}
\date{}

\subjclass[2000]{}
\newtheorem{THM}{Theorem}
\newtheorem{PROP}[THM]{Proposition}

\newtheorem{thm}{Theorem}[section]
\newtheorem{prop}[thm]{Proposition}
\newtheorem{cor}[thm]{Corollary}
\newtheorem{lem}[thm]{Lemma}
\theoremstyle{definition}
\newtheorem{defn}[thm]{Definition}
\newtheorem{exam}[thm]{Example}
\theoremstyle{remark}
\newtheorem{rem}[thm]{Remark}
\newtheorem*{rem*}{Remark}
\newtheorem{nota}[thm]{Notation}

\renewcommand{\epsilon}{\varepsilon}
\newcommand{\f}{\mathcal{F}}
\renewcommand{\phi}{\varphi}
\renewcommand{\hom}{\mathrm{Hom}}
\newcommand{\sym}{\mathfrak{S}}
\newcommand{\fs}{{\mathsf{FS}}}
\newcommand{\kmod}{\mathtt{Mod}_\kk}
\newcommand{\calc}{\mathcal{C}}
\newcommand{\cald}{\mathcal{D}}
\newcommand{\nat}{\mathbb{N}}
\newcommand{\zed}{\mathbb{Z}}

\newcommand{\ext}{\mathrm{Ext}}
\newcommand{\op}{^\mathrm{op}}
\newcommand{\ob}{\mathrm{Ob}\hspace{2pt}}
\newcommand{\fb}{\mathsf{FB}}
\newcommand{\finj}{{\mathsf{FI}}}
\newcommand{\id}{\mathrm{Id}}
\newcommand{\triv}{\mathsf{triv}}
\newcommand{\sgn}{\mathsf{sgn}}
\newcommand{\aut}{\mathrm{Aut}}
\newcommand{\fin}{\mathsf{FA}}
\newcommand{\n}{\mathbf{n}}
\newcommand{\m}{\mathbf{m}}
\newcommand{\kk}{\mathbbm{k}}
\newcommand{\dash}{\text{-}}
\newcommand{\modules}{\mathsf{mod}}
\newcommand{\opd}{\mathscr{O}}
\newcommand{\uwb}{\mathsf{uwb}}

\newcommand{\dwb}{\mathsf{dwb}}
\newcommand{\db}{\mathsf{db}}
\newcommand{\dwbord}{\dwb^{\mathrm{ord}}}
\newcommand{\dwbtw}{(\kk \dwb)_-}
\newcommand{\tfb}{\fb \times \fb}
\newcommand{\kzcx}{\mathscr{K}}
\newcommand{\uwbord}{\uwb^{\mathrm{ord}}}
\newcommand{\fiord}{\finj^{\mathrm{ord}}}
\newcommand{\fitw}{(\kk \finj)_-}
\newcommand{\uwbtw}{(\kk \uwb)_-}

\newcommand{\cx}{\mathbb{C}}
\newcommand{\spmon}{\mathcal{S}(\kk)}
\newcommand{\cala}{\mathscr{A}}
\newcommand{\calb}{\mathscr{B}}

\newcommand{\nuopds}{\mathsf{Opd}^{\mathrm{nu}}}
\newcommand{\stfb}{S_\circledcirc}
\newcommand{\sfb}{S_\odot}
\newcommand{\ltfb}{\Lambda_\circledcirc}

\newcommand{\wpair}{\mathsf{pair}}
\newcommand{\pc}{\widetilde{\mu}}
\newcommand{\pct}{\overline{\mu}}
\newcommand{\dbd}{\delta_{1,1}}
\newcommand{\shift}[2]{\delta_{#1,#2}}
\newcommand{\ouwb}{\mathsf{Ouwb}}
\newcommand{\odwb}{\mathsf{Odwb}}
\newcommand{\tor}{\mathrm{Tor}}
\newcommand{\traceless}[1]{T^{\{ #1\}}}
\newcommand{\tracelessop}[1]{T^{\{#1\}\tau}}
\newcommand{\uwbup}{\f^\uparrow(\uwb)} 
\newcommand{\uwbdown}{\f^\downarrow (\uwb)}
\newcommand{\GL}{\mathrm{GL}}
\newcommand{\der}{\mathrm{Der}}
\newcommand{\tmalg}{\widetilde{\amalg}}
\newcommand{\torskuwb}{\kk \uwb\dash\mathbf{Tors}}
\newcommand{\elrel}{\sim^\mathsf{elem}}
\newcommand{\gl}{\mathbf{GL}}
\newcommand{\stabgl}{\mathrm{Stab}_{\mathrm{GL}}}
\newcommand{\rep}{\mathrm{Rep}}
\newcommand{\tors}{\mathsf{tors}}
\newcommand{\falg}{\f^\mathrm{alg}}
\newcommand{\fc}{\mathfrak{C}}
\newcommand{\puwb}{P^\uwb}

\numberwithin{equation}{section}
\begin{document}

\begin{abstract}
In this paper, we investigate certain complexes associated with an operad $\mathscr{O}$ in $\mathbbm{k}$-vector spaces, where $\mathbbm{k}$ is a field of characteristic $0$. This exploits the study of modules over the $\mathbbm{k}$-linearization of the  upward walled Brauer category, $\mathbbm{k}\mathsf{uwb}$ (respectively of the  downward walled Brauer category, $\mathbbm{k}\mathsf{dwb}$). These are Koszul $\mathbbm{k}$-linear categories over $\mathbbm{k} (\mathbf{FB \times FB})$, where $\mathbf{FB}$ is the category of finite sets and bijections. The Koszul property yields associated Koszul complexes, with an  elegant interpretation of their (co)homology.

The application to operads starts from the observation that the Chevalley-Eilenberg complex for the Lie algebra of derivations $\mathrm{Der} (\mathscr{O} (V))$ of the free $\mathscr{O}$-algebra on a finite-dimensional vector space $V$ has a precursor given by the Koszul complex on an explicit module over $(\mathbbm{k}\mathsf{dwb})_-$ (a twisted $\mathbbm{k}$-linearization of $\mathsf{dwb}$);  this module  is constructed naturally from the operad $\mathscr{O}$. 
Following Dotsenko, we also consider the more general case where a {\em wheeled} term is included.

This identification exploits functoriality with respect to the category of finite-dimensional $\mathbbm{k}$-vector spaces with morphisms taken to be split monomorphisms, together with the relationship with functors on the upward (and downward) walled Brauer category. We also exploit methods developed by Sam and Snowden for  investigating the stabilization of the families of representations of the general linear groups associated to a functor on the category of split monomorphisms between $\mathbbm{k}$-vector spaces.

Using these methods, we give a new perspective on the results of Dotsenko, who investigated the stable homology of Lie algebras of derivations and established a link with the wheeled bar construction for wheeled operads. In particular, we explain why one of the Koszul complexes that we consider should be considered as the appropriate form of hairy graph complex for operads, by analogy with the case of cyclic operads. 
\end{abstract}

\maketitle

\section{Introduction}
\label{sect:intro}

In this paper, we consider algebraic structures in $\kk$-vector spaces (where $\kk$ is a field of characteristic zero for this introduction), such as commutative algebras or Lie algebras. In particular, we may consider free algebras and then derivations of free algebras. More precisely, we consider algebraic structures that are encoded by an operad $\opd$ in $\kk$-vector spaces; the free algebra generated by a finite-dimensional $\kk$-vector space $V$ is denoted $\opd (V)$ and the vector space of derivations (of $\opd$-algebras) of $\opd (V)$ is denoted $\der (\opd (V))$. Generalizing the classical fact that the derivations of a commutative, associative algebra form a Lie algebra, $\der (\opd (V))$ has a  Lie algebra structure (naturally with respect to $\opd$). This structure is independent of the unit of the operad and  we work  without units.

An important observation is that  $\der (\opd (V))$ is also natural with respect to $V$ when this is  considered as an object of the category $\spmon$, which has morphisms given by split monomorphisms between finite-dimensional $\kk$-vector spaces. This naturality was exploited in this context by the author in \cite{MR4881588}.

One can then form the Chevalley-Eilenberg complex 
$
C_*^\mathsf{CE} (\der \opd (V))
$
for this  Lie algebra structure, and the associated Lie algebra homology, $
H_*^\mathsf{CE} (\der \opd (V))
$. These are natural with respect to $V$ in $\spmon$.

This paper is motivated in part by the following questions: how to encode the structure giving rise to $V \mapsto C_*^\mathsf{CE} (\der \opd (V))$, based directly upon the structure of the operad $\opd$? how to under the {\em stabilization} of the Lie algebra homology $H_*^\mathsf{CE} (\der \opd (V))$ with respect to $V$ in $\spmon$? These are  counterparts of questions addressed in \cite{P_cyclic} in the framework of cyclic operads.

The second question has already been addressed by Dotsenko in \cite{MR4945404} from a different viewpoint (not using  functors on $\spmon$). Moreover, he addresses a more general question by also considering $\opd^\circlearrowright$, the {\em wheeled completion} of the operad. This leads to the consideration of a differential graded Lie algebra that Dotsenko denotes by 
$$
\der (\opd (V)) \ltimes_{\mathrm{div}} |\partial \opd (V)|.
$$
(This will be denoted here by $
\der (\opd (V)) \oplus s^{-1}|\partial \opd (V)|
$
 to keep track of the grading.) 
 Here $|\partial \opd (V)|$ corresponds to the wheeled component $|\dbd \opd|$ in the framework of the current paper (see Section \ref{sect:spmon_operads} for this), and the differential depends on the divergence, $\mathrm{div}$.

This paper gives a different approach to understanding these questions, focusing upon the {\em underlying} structures, before passing to functors on $\spmon$. The relevant structures are encoded by  $\kk\dwb$-modules (where $\dwb$ is the downward walled Brauer category) or $\dwbtw$-modules (where $\dwbtw$ is a  twisted version of $\kk\dwb$). (These categories and their `upward' opposites are reviewed in Section \ref{sect:walled_brauer}.) These categories play an important rôle, for example, in Sam and Snowden's approach to stability of representations of the general linear groups \cite[Section 3]{MR3376738}.

The link with functors on $\spmon$ is established by using the mixed tensor functors: for $m,n \in \nat$, $T^{m,n}$ is the functor $ V \mapsto V^{\otimes m} \otimes (V^\sharp) ^{\otimes n}$, where $V^\sharp$ is the dual of $V$. The product of symmetric groups $\sym_m \times \sym_n$ acts on $T^{m,n}$ by place permutation of the tensor factors of $V$ (respectively $V^\sharp$). The evaluation map 
$
V \otimes V^\sharp \rightarrow \kk
$
provides more structure,  giving a natural transformation $T^{1,1} \rightarrow T^{0,0}$ of functors on $\spmon$. This  extends so that
 $
(m,n) \rightarrow T^{m,n}
$
 defines a fully-faithful functor from $\kk \dwb$ to $\f (\spmon)$, the category of functors from $\spmon$ to $\kk$-vector spaces. 

Using this, it follows that we have a form of `generalized Schur functor' (this terminology is inspired by \cite{MR3876732}):
$$
T^{\bullet, \bullet} \otimes_{\kk \uwb} - 
: 
\kk\uwb \dash\modules 
\rightarrow 
\f (\spmon)
$$
from $\kk \uwb$-modules to $\f (\spmon)$ (see Section \ref{subsect:gen_schur_fn} for a review of this). This allows us to encode structures by using $\kk \uwb$-modules (or related categories). 

The other key ingredient is the fact that $\kk \uwb$ is a Koszul $\kk$-linear category over $\kk (\tfb)$, where $\fb$ denotes the category of finite sets and bijections. This has two  consequences: it allows us to define certain Koszul complexes (this only depends on the fact that $\kk \uwb$ is a homogeneous, quadratic $\kk$-linear category over $\kk (\tfb)$); then, using the Koszul property, we have elegant (co)homological interpretations of the (co)homology of these Koszul complexes. 

To be more explicit, Proposition \ref{prop:quadratic_duals} recalls that the right quadratic dual of $\kk\uwb$ is $\dwbtw$, a twisted form of the $\kk$-linearization of $\dwb$;
the Koszul property is then recalled in Theorem \ref{thm:dwb_koszul}. The Koszul complexes are based on the dualizing complex $\kzcx$ that is introduced in Section \ref{subsect:quadratic_dual_dualizing_complex}; this has underlying object 
$$
\kzcx := \kk \uwb \otimes_{\kk(\tfb)}  \dwbtw,
$$ 
equipped with an explicit differential. 

Then, if $M$ is a $\dwbtw$-module, we can form the following two complexes
\begin{eqnarray}
\label{EQN:1}
&&\kzcx \otimes_{\dwbtw} M 
\\
\label{EQN:2}
&&\kk^\uwb \otimes_{\kk \uwb} \kzcx \otimes _{\dwbtw} M
\end{eqnarray}
in the category of $\kk \uwb$-modules. Corollary \ref{cor:cohomology_koszul_complexes} and Proposition \ref{prop:homology_is_Tor} give the following interpretation of their respective (co)homology:

\begin{THM}
For $M$ a $\dwbtw$-module, 
\begin{enumerate}
\item 
the cohomology of (\ref{EQN:1}) is naturally isomorphic to $\ext^*_{\dwbtw} (\kk (\tfb), M)$; 
\item 
the homology of (\ref{EQN:2}) is naturally isomorphic to $\tor^{\dwbtw}_* (\kk (\tfb), M)$.
\end{enumerate}
\end{THM}

We can apply the generalized Schur functor $T^{\bullet, \bullet} \otimes_{\kk \uwb} - $ to the complex (\ref{EQN:1}). This gives a complex in $\f (\spmon)$ that has underlying object that identifies as 
\begin{eqnarray}
\label{EQN:3}
T^{\bullet, \bullet} \otimes_{\kk (\tfb)} M, 
\end{eqnarray}
equipped with a Koszul-type differential induced by that of $\kzcx$. 

To recover information from this (not expressed in terms of functors on $\spmon$), we appeal to the methods developed by Sam and Snowden in \cite[Section 3]{MR3376738}. 
 A key fact is that $T^{m,n}$ is an {\em algebraic functor},  as defined by Sam and Snowden, hence the above is a  complex in $\falg (\spmon)$, the full subcategory of algebraic functors. As  reviewed in Section \ref{sect:mixed_tensors}, the functor  
\begin{eqnarray}
\label{EQN:4}
\hom_{\falg (\spmon)} (-, T^{\ast, \ast}) 
: \falg (\spmon) \op 
\rightarrow 
\kk \dwb \dash \modules
\end{eqnarray}
is exact. Moreover, it factors across the localization away from torsion modules
$$
\falg (\spmon) 
\rightarrow 
\falg (\spmon)/
\falg_\tors (\spmon).
$$
Hence we consider (\ref{EQN:4}) as a {\em weak stabilization} functor (this is justified further in Section \ref{sect:mixed_tensors} by comparing it with Sam and Snowden's stabilization).
 
This leads to the following result (a combination of Lemma \ref{lem:hom_Tastast_complexes} and Proposition \ref{prop:weak_stabilization_homology}):

\begin{PROP}
For $M$ a $\dwbtw$-module, there is a natural isomorphism of complexes of $\kk \dwb$-modules:
\begin{eqnarray*}
\hom_{\f(\spmon)} ( T^{\bullet, \bullet} \otimes _{\kk \uwb} \kzcx \otimes_{\dwbtw} M, T^{*,*})
\cong 
\hom_\kk (\kk^\uwb \otimes_{\kk \uwb} \kzcx \otimes_{\dwbtw} M, \kk).
\end{eqnarray*}

The weak stabilization of the homology $H_* (T^{\bullet, \bullet} \otimes _{\kk \uwb} \kzcx \otimes_{\dwbtw} M)$, which is a graded object in $\falg (\spmon)$, identifies as graded objects in $\kk \dwb$-modules as follows:
$$
\hom_{\spmon} (H_* (T^{\bullet, \bullet} \otimes _{\kk \uwb} \kzcx \otimes_{\dwbtw} M), T^{\ast, \ast})
\cong 
\hom_\kk (H_* (\kk ^\uwb \otimes_{\kk \uwb} \kzcx \otimes_{\dwbtw} M ), \kk).
$$
Moreover, the right hand side is naturally isomorphic to  the $\kk$-linear dual of 
$
\tor^{\dwbtw}_* ( \kk (\tfb) , M).
$
\end{PROP}
 
This may be resumed as follows: starting from the Koszul complex (\ref{EQN:1}), applying the weak stabilization functor to the complex (\ref{EQN:3}) of algebraic functors on $\spmon$ (obtained using the generalized Schur functor) yields the $\kk$-linear dual of the complex (\ref{EQN:2}). Hence the weak stabilization of the homology of (\ref{EQN:3}) is isomorphic to the $\kk$-linear dual of the homology of (\ref{EQN:2}).

Using the above theory, we can attack the questions posed at the beginning of this Introduction. For this, in Section \ref{sect:operads}, we first formulate the definition of an operad (without unit) using partial compositions. The approach here is non-standard, since we  work in the category of $\kk (\tfb)$-modules,  also introducing the wheeled component $|\dbd \opd|$. Working in $\kk (\tfb)$-modules, we have that 
\begin{enumerate}
\item 
$\opd$ is supported on objects of the form $(\n, \mathbf{1})$ (where $\n = \{1, \ldots , n\}$); 
\item 
$|\dbd \opd|$ is supported on objects of the form $(\n, \mathbf{0})$ (where $\mathbf{0}$ identifies with $\emptyset$). 
\end{enumerate}

Then the key input is the  following restatement of Theorem \ref{thm:ltfb}, in which $\stfb^*$ denotes the `symmetric algebra' formed in the category of $\kk (\tfb)$-modules using the Day convolution product $\circledcirc$, and $\ltfb^*$ denotes the corresponding `exterior algebra' (see Section \ref{sect:day_convolution} for more details on these).

\begin{THM}
\label{THM:opd}
For $\opd$ a non-unital operad, there is a natural surjection of $\dwbtw$-modules:
\begin{eqnarray}
\label{EQN:5}
\ltfb^* \opd \circledcirc \stfb^* |\dbd \opd |
\twoheadrightarrow 
\ltfb^* \opd.
\end{eqnarray}
\end{THM}

Then (\ref{EQN:5})  induces morphisms of Koszul complexes in $\kk \uwb$-modules:
\begin{eqnarray}
\label{EQN:6}
&&
\kzcx \otimes_{\dwbtw} \big( \ltfb^*\opd \circledcirc \stfb^* |\shift{1}{1}\opd|\big)
\twoheadrightarrow 
\kzcx \otimes_{\dwbtw} \ltfb^*\opd 
\\
\label{EQN:7}
&&\kk^\uwb \otimes_{\kk \uwb} \kzcx \otimes_{\dwbtw} \big( \ltfb^*\opd \circledcirc \stfb^* |\shift{1}{1}\opd|\big)
\twoheadrightarrow
\kk^\uwb \otimes_{\kk \uwb} \kzcx \otimes_{\dwbtw} \ltfb^*\opd,
\end{eqnarray}
as described in Section \ref{sect:opd_koszul}. The (co)homological interpretation of these is given in Theorem \ref{thm:opd_ext_tor}. 

The underlying map of (\ref{EQN:7}) can be rewritten:
\begin{eqnarray}
\label{EQN:8}
\kk^\uwb \otimes_{\kk (\tfb)}  \big( \ltfb^*\opd \circledcirc \stfb^* |\shift{1}{1}\opd|\big)
\twoheadrightarrow
\kk^\uwb \otimes_{\kk (\tfb)}  \ltfb^*\opd,
\end{eqnarray}
where each complex is equipped with a Koszul-type differential induced by that of $\kzcx$.

In Definition \ref{defn:wheeled_hairy_flow_graph}, we {\em define} the domain of (\ref{EQN:8}) to be the {\em wheeled hairy flow-graph complex} associated to $\opd$ and the codomain to be the {\em hairy flow-graph complex} associated to $\opd$; this definition is justified in Section \ref{subsect:hairy_flow-graph}. (The terminology {\em flow-graph} refers to the fact that, in this operadic framework, we restrict to `graphs with a flow', as defined in Section \ref{subsect:flow-graph}). This uses our description of (directed) graphs using the category $\uwb$ to encode edges, as explained in Section \ref{sect:graphs}. This explains heuristically {\em why} the complexes that we are considering are an appropriate form of hairy graph complex.

We have not yet explained the link with the Chevalley-Eilenberg complex of the DG Lie algebra $
\der (\opd (V)) \ltimes_{\mathrm{div}} |\partial \opd (V)|$ (in Dotsenko's notation) and the Lie algebra $\der (\opd (V))$. This is achieved in Theorem \ref{thm:ltfb_identify_with_CE}:

\begin{THM}
\label{THM:CE}
For a non-unital operad $\opd$, the morphism of complexes in $\falg (\spmon)$ 
$$
T^{\bullet , \bullet} \otimes _{\kk (\tfb)} \big( \ltfb^*\opd \circledcirc \stfb^* |\shift{1}{1}\opd|\big)
\twoheadrightarrow
T^{\bullet , \bullet} \otimes _{\kk (\tfb)} \ltfb^*\opd 
$$
given by applying $T^{\bullet, \bullet} \otimes_{\kk \uwb} -$ to (\ref{EQN:6})   is naturally isomorphic (after evaluating on $V$ in $\spmon$) to the morphism between Chevalley-Eilenberg complexes associated to  the morphism of DG Lie algebras $\der (\opd (V)) \oplus s^{-1}|\partial \opd (V)|
\twoheadrightarrow \der (\opd (V)) $. This identifies as 
\[
\Lambda^* ( \der (\opd (V))) \otimes S^* ( |\partial \opd (V)|)
\twoheadrightarrow 
\Lambda^* ( \der (\opd (V)))),
\]
and is natural with respect to $V$ in $\spmon$.

On passing to homology, there is a commutative diagram of graded algebraic functors on $\spmon$ 
$$
\xymatrix{
H_* \Big(T^{\bullet , \bullet} \otimes _{\kk (\tfb)} \big( \ltfb^*\opd \circledcirc \stfb^* |\shift{1}{1}\opd|\big)\Big)
\ar[r]
\ar[d]_\cong 
&
H_* \big(T^{\bullet , \bullet} \otimes _{\kk (\tfb)}  \ltfb^*\opd \big)
\ar[d]^\cong 
\\
H^\mathsf{CE}_* ( \der (\opd (-)) \oplus s^{-1} |\partial \opd (-)|)
\ar[r]
&
H^\mathsf{CE}_* ( \der (\opd (-))).
}
$$
\end{THM}

We can then apply the weak stabilization to get the following (see Theorem \ref{thm:weak_stabilization_CE}):

\begin{THM}
\label{THM:weak_stabilization}
For a non-unital operad $\opd$, 
the weak stabilization of the natural transformation of graded algebraic $\kk\spmon$-modules
$$
H^\mathsf{CE}_* ( \der (\opd (-)) \oplus s^{-1} |\partial \opd (-)|)
\rightarrow 
H^\mathsf{CE}_* ( \der (\opd (-)))
$$
identifies as the $\kk$-linear dual of the morphism in homology 
\begin{eqnarray*}
H_* \big(\kk^\uwb \otimes_{\kk (\tfb)} ( \ltfb^*\opd \circledcirc \stfb^* |\shift{1}{1}\opd|)\big) 
\rightarrow 
H_* (\kk^\uwb \otimes_{\kk (\tfb)} \ltfb^*\opd )
\end{eqnarray*}
given by the morphism of complexes (\ref{EQN:8}). 

The latter identifies with the morphism
$$
\tor_*^{\dwbtw} (\kk (\tfb), \ltfb^*\opd \circledcirc \stfb^* |\shift{1}{1}\opd|) 
\longrightarrow     
  \tor_*^{\dwbtw} (\kk (\tfb), \ltfb^*\opd )
$$
induced by the morphism of Theorem \ref{THM:opd}. This can be viewed as the natural morphism from the wheeled hairy flow-graph homology of $\opd$ to the hairy flow-graph homology of $\opd$.
\end{THM}

This result is a counterpart of Dotsenko's \cite[Theorems 4.15 and  4.1]{MR4945404}. The translation between these different approaches requires a little bit of work, since Dotsenko uses stable  $\mathfrak{gl}(V)$-invariants rather than the weak stabilization that is used here. This is explained in Section \ref{subsect:compare_Dotsenko}, using the material of Section \ref{sect:gl_inv} to relate these different approaches to stabilization. 

In particular, Corollary \ref{cor:stable_GL_inv_CE} establishes that the wheeled hairy flow-graph complex (as defined here) is equivalent to a (reduced form) of the coPROP completion of the wheeled bar construction $B^\circlearrowright (\opd^\circlearrowright)$ that appears in \cite[Theorems 4.15]{MR4945404}. This gives further justification for our notion of {\em wheeled hairy flow-graph complex}.

Theorem \ref{THM:weak_stabilization} gives an interpretation of 
$
\tor_*^{\dwbtw} (\kk (\tfb), \ltfb^*\opd \circledcirc \stfb^* |\shift{1}{1}\opd|) $ and $    
  \tor_*^{\dwbtw} (\kk (\tfb), \ltfb^*\opd )$ 
as  (wheeled) hairy flow-graph homology. For $\ext^*$ groups, as far as the author is aware, the terms in  
$$
\ext^*_{\dwbtw} (\kk (\tfb), \ltfb^*\opd \circledcirc \stfb^* |\shift{1}{1}\opd|) 
\rightarrow 
\ext^*_{\dwbtw} (\kk (\tfb), \ltfb^*\opd )
$$
given by taking the cohomology of (\ref{EQN:6}) do not appear in the literature in this context.

However, these do have an important rôle in the theory developed here, since there is a universal coefficients spectral sequence that relates these to the Lie algebra homology appearing in Theorem \ref{THM:CE}. This relies crucially upon the fact that these $\ext^*$-groups have natural (graded) $\kk \uwb$-module structures. 
One has the following (Theorem \ref{thm:Kunneth_ss_operadic_case} in the body of the text):

\begin{THM}
\label{THM:ss}
For $\opd$ a non-unital operad, there are natural  spectral sequences that are functorial with respect to $\spmon$:
\begin{eqnarray*}
\tor_*^{\kk \uwb} (T^{\bullet, \bullet} , \ext^* _{\dwbtw} (\kk (\tfb), \ltfb^*\opd \circledcirc \stfb^* |\shift{1}{1}\opd| )  ) 
&\Rightarrow &
H_*^\mathsf{CE} (\der (\opd (-)) \oplus s^{-1} |\partial \opd (-)| ) 
\\
\tor_*^{\kk \uwb} (T^{\bullet, \bullet} , \ext^* _{\dwbtw} (\kk (\tfb), \ltfb^*\opd   ) 
&\Rightarrow &
H_*^\mathsf{CE} (\der (\opd (-)) ) ,
\end{eqnarray*}
together with a morphism of spectral sequences between these induced by (\ref{EQN:5}).

In particular, evaluating on $V$ (considered as an object of $\spmon$), there is a commutative diagram corresponding to the edge homomorphisms:
$$
\xymatrix{
T^{\bullet, \bullet} (V) \otimes_{\kk \uwb} \ext^* _{\dwbtw} (\kk (\tfb), \ltfb^*\opd \circledcirc \stfb^* |\shift{1}{1}\opd| )
\ar[r]
\ar[d]
&
H^\mathsf{CE}_* ( \der (\opd (V)) \oplus s^{-1} |\partial \opd (V)|)
\ar[d]
\\
T^{\bullet, \bullet} (V) \otimes_{\kk \uwb} \ext^* _{\dwbtw} (\kk (\tfb), \ltfb^*\opd )
\ar[r]
&
H^\mathsf{CE}_* ( \der (\opd (V))).
}
$$
\end{THM}

This theorem gives a precise sense in which the $\ext^*$ groups are a precursor to the {\em unstable} values of the DG Lie algebra homology groups figuring in the following;
$$
H^\mathsf{CE}_* ( \der (\opd (V)) \oplus s^{-1} |\partial \opd (V)|)
\rightarrow 
H^\mathsf{CE}_* ( \der (\opd (V))).
$$

The spectral sequences of Theorem \ref{THM:ss} should be compared with the analogous universal coefficients spectral sequences that relate $\ext^*$ and $\tor_*$ (see Remark \ref{rem:ext_to_tor}).

\begin{rem*}
The story told here is parallel to that told in \cite{P_cyclic}, which addresses the corresponding theory for cyclic operads. When working with cyclic operads, one uses (twisted versions) of the $\kk$-linearization of the upward Brauer category (and their opposites). The starting point there is the analogue of Theorem \ref{THM:opd} for the cyclic case, which is related to Stoll's characterization of modular operads as algebras over the Brauer {\em properad} \cite{MR4541945}. The cyclic story is slightly more complicated, in that one has to deal with both {\em odd} and {\em even} hairy  graph complexes; this is related to the twisting of the Feynman transformation by hyperoperads, as in the work of Getzler and Kapranov \cite{MR1601666}. 

The relationship between the respective theories for cyclic operads and for operads (and, more generally, dioperads) is considered in \cite{2026arXiv260413750P}.
\end{rem*}

\bigskip
\paragraph{\bf Acknowledgement:}
This project was inspired in part by the work of Vladimir Dotsenko and the author is also grateful for his interest. The author would also like to thank Kazuo Habiro and Mai Katada for their interest and for informing him of their related work in progress (for the cyclic operad case).

The author is grateful to José S\~ao Jo\~ao for his interest and for informing him of his work-in-progress investigating  the behaviour in the presence of a unit for the operad, extending Dotsenko's theory. 

Part of this work was presented by the author at the workshop \href{https://sites.google.com/g.ecc.u-tokyo.ac.jp/alg-approach-mcg/}{{\em Algebraic approaches to mapping class groups of surfaces}} at the University of Tokyo in May 2025. The author is very grateful to have been invited to speak there and wishes to acknowledge the financial support for his visit  provided by  Professor Nariya Kawazumi's grant.

\bigskip
\paragraph{\bf Financial Support:}
This research was supported by the ANR project KAsH  \href{https://anr.fr/Projet-ANR-25-CE40-2861}{ANR-25-CE40-2861}.

\tableofcontents

\section{Background}
\label{sect:background}

This section serves to introduce some notation and also to review basics on functor categories. 

\subsection{Notation}

The following  notation will be used for the various flavours of categories of finite sets.

\begin{nota}
\label{nota:fin_sets}
The category of finite sets and all maps is denoted $\fin$, with the  wide subcategories $\fb$, $\finj$, $\fs$ corresponding respectively to the bijective maps, injective maps, surjective maps. 

For $n\in \nat$, the set $\{1, \ldots, n\}$ is denoted by $\n$; for $n=0$, this is understood to be $\emptyset$. The group of automorphisms $\aut(\n)$ is denoted $\sym_n$; this is the symmetric group on $n$ letters.
\end{nota}

\begin{rem}
\ 
\begin{enumerate}
\item 
Each of the categories in Notation \ref{nota:fin_sets}  has skeleton with objects $\{ \n \mid n \in \nat\}$. 
\item  
When using the support of functors on these categories, we will implicitly restrict to the  skeleton.
\end{enumerate}
\end{rem}

\begin{nota}
\label{nota:sgn_triv}
Working over a field $\kk$, for $n \in \nat$, let $\sgn_n$ (respectively $\triv_n$) denote the sign (resp. trivial) representation of $\sym_n$. 
\end{nota}

\subsection{Generalities on functor categories}

Fix $\kk$ a field and write $\kmod$ for the category of $\kk$-vector spaces.  

\begin{nota}
\label{nota:FC}
\ 
\begin{enumerate}
\item 
For $\cala$ an essentially small $\kk$-linear category, the category   $\cala \dash\modules$ of $\cala$-modules is the category of $\kk$-linear functors from $\cala$ to $\kmod$.
\item 
For $\calc$ an essentially small category, $\f (\calc)$ denotes the category of functors from $\calc$ to  $\kmod$. 
\end{enumerate}
\end{nota}

\begin{rem}
For $\calc$ as above, denote by $\kk \calc$ its $\kk$-linearization. By the universal property of $\kk$-linearization, the category $\f(\calc)$ is equivalent to $\kk \calc \dash\modules$ and hence will also be termed the category of $\kk \calc$-modules. 
\end{rem}

For $\cala$ an essentially small $\kk$-linear category,  the category $\cala \dash\modules$ is abelian, with structure inherited from that of $\kmod$.  Moreover, vector space duality $(-)^\sharp : \kmod \op \rightarrow \kmod$ induces the exact functor $(-)^\sharp : (\cala\dash\modules) \op \rightarrow \cala \dash\modules $.

\begin{rem}
\label{rem:tensor_product_F(C)}
The tensor product of $\kk$-vector spaces (denoted simply $\otimes$) induces a symmetric monoidal structure $(\f (\calc), \otimes , \kk)$, where $\kk$ denotes the constant functor with value $\kk$. This is given as follows: for $X$ an object of $\calc$ and $F, G$ two $\kk \calc$-modules, $(F\otimes G)(X) := F(X) \otimes G(X)$. Since we are working over a field, $\otimes $ is exact with respect to both variables. 

 For $\cala$ an essentially small $\kk$-linear category, in general this tensor product construction does not apply to $\cala\dash\modules$; indeed, there may not even exist a `constant functor'. However, one does have an exterior tensor product of modules, as explained below.

 For  $\calb$ a second essentially small $\kk$-linear category, one has the $\kk$-linear category $\cala \otimes \calb$ defined in the usual way.
 For example, if $\calc$ and $\cald$ are categories with $\kk$-linearizations $\kk \calc$ and $\kk\cald$ respectively, then $\kk \calc \otimes \kk \cald$ is naturally isomorphic to $\kk (\calc \times \cald)$. Taking $\calb$ to be $\cala \op$, one can define the category of $\cala$-bimodules to be $(\cala \otimes \cala \op)\dash\modules$.

Similarly, if $M$ and $N$ are $\cala$- and $\calb$-modules respectively, one has the exterior tensor product $M \boxtimes N$ which is an $\cala \otimes \calb$-module. 
\end{rem}
 
\begin{rem}
For $\calc$ an essentially small category, morphisms in $\f (\calc)$ (which is equivalent to $\kk\calc\dash\modules$) are denoted  either by $\hom_\calc(-, -)$ or by $\hom_{\kk \calc} (-,-)$ (if we are thinking in terms of $\kk \calc$-modules).
\end{rem}

\begin{exam}
\label{exam:bimodules_proj_inj}
For $\cala$ an essentially small $\kk$-linear category, $\cala (-,-)$ has a canonical $\cala$-bimodule structure, hence so does the dual $\cala (-,-)^\sharp$.

For $X$ an object of $\cala$, by restriction one has the $\cala$-module $\cala (X, -)$. By the $\kk$-linear Yoneda lemma, this is projective: it corepresents evaluation  $M \mapsto M (X)$ for $M$ an $\cala$-module. Dually, one has the $\cala$-module $\cala (-, X)^\sharp$; this is injective, since it represents the functor $M \mapsto M(X)^\sharp$. (In the case $\cala = \kk \calc$, this injective may also be written $\kk ^{\calc (-, X)}$.)    

In particular,  $\cala\dash\modules$ has both enough projectives and enough injectives, so that we may consider the derived functors of $\hom_\cala (-,-,)$, namely $\ext^*_\cala (-,-)$.
\end{exam}

\begin{rem}
\label{rem:otimes_cala}
If $\cala$ is an essentially small $\kk$-linear category, one has the functor $ - \otimes _\cala - : (\cala \op\dash\modules) \times (\cala \dash \modules) \rightarrow \kmod$. 
 This is the `multi-object' generalization of the usual tensor product over an associative algebra. 
 
We can thus  consider the (left) derived functors of $  - \otimes _\cala -$, which are denoted by $\tor^\cala_*(-,-)$.
 \end{rem}

\subsection{The disjoint union of finite sets and associated functors}
\label{subsect:disjoint_union}

In order to introduce the respective convolution products on $\kk \fb$-modules and $\kk (\tfb)$-modules in Section \ref{sect:day_convolution}, we require the pushforward functors induced by the disjoint union of finite sets. These are introduced in Definition \ref{defn:amalg_tmalg_left}.

The disjoint union of finite sets induces the functor $\amalg  : \fb \times \fb \rightarrow \fb$ and this defines a symmetric monoidal structure on $\fb$, with unit $\emptyset$. 
This extends to a symmetric monoidal structure $\tmalg$ on $\tfb$:
\begin{eqnarray*}
 \tmalg &:& (\tfb) \times (\tfb) \rightarrow \tfb
 \\
&&\big( (X,Y), (U,V)\big)  \mapsto  (X \amalg U, Y \amalg V).
\end{eqnarray*}
(The notation $\tmalg$ is introduced to distinguish this from the coproduct on $\fb$.)

\begin{exam}
\label{exam:m+n}
For $m, n \in \nat$, the finite set $\mathbf{m+n}$ denotes $\{1, \ldots , m+n \}$. This comes equipped with the obvious inclusions $\mathbf{m} \subset \mathbf{m+n}$ and $\mathbf{n} \subset \mathbf{m+n}$. Moreover, there is a bijection 
$
\mathbf{m} \amalg \mathbf{n}
\cong 
\mathbf{m+n}
$
 given  by using the above inclusion $\mathbf{m} \subset \mathbf{m+n}$ together with the {\em shifted} inclusion 
$\mathbf{n} \hookrightarrow \mathbf{m+n}$ given by $j \mapsto m+j$.
 (There is also a bijection $
\mathbf{n} \amalg \mathbf{m}
\cong 
\mathbf{m+n}$
reversing the rôles of $m$ and $n$.)
\end{exam}

The above functors yield  the restriction functors $\amalg^* : \f (\fb ) \rightarrow \f (\tfb)$ and $\tmalg ^* : \f (\tfb) \rightarrow  \f ((\tfb)^{\times 2}) $.
For a $\kk \fb$-module $F$ and a $\kk (\tfb)$-module $G$, these identify respectively as
\begin{eqnarray*} 
\amalg^* F (U, V)&= &F (U \amalg V) 
\\
\tmalg^* G ((X,Y) , (U,V))&= &G (X \amalg U, Y \amalg V). 
\end{eqnarray*}

We introduce the following functors which identify as the respective adjoints by  Proposition \ref{prop:amalg_tmalg_adjoints} below:

\begin{defn}
\label{defn:amalg_tmalg_left}
\ 
\begin{enumerate}
\item 
Let $\amalg_* : \f (\tfb) \rightarrow \f (\fb)$ be the functor given on objects, for $G$ a $\kk (\tfb)$-module, by 
\[
\amalg_* G (X):= \bigoplus_{ U \amalg V=X} G (U, V),
\]
where the sum is indexed by ordered decompositions of $X$ into two subsets. 
\item 
Let $\tmalg_* : \f ((\tfb)^{\times 2}) \rightarrow \f (\tfb)$ be the functor given on objects, for $H$ a $\kk ((\tfb)^{\times 2})$-module, by 
\begin{eqnarray}
\label{eqn:tmalg_*H}
\tmalg_* H (U, V):= \bigoplus_{\substack{U_1 \amalg U_2 =U\\  V_1 \amalg V_2=V}} H ((U_1, V_1), (U_2, V_2)),
\end{eqnarray}
where the sum is indexed over pairs of ordered decompositions of $U$ and of $V$ into two subsets. 
\end{enumerate}
\end{defn}

\begin{rem}
Both of the above functors can be constructed as (left) Kan extensions. This encodes the behaviour of the morphisms.
\end{rem}

\begin{prop}
\label{prop:amalg_tmalg_adjoints}
\
\begin{enumerate}
\item 
The functor $\amalg_* : \f (\tfb) \rightarrow \f (\fb)$ is both left and right adjoint to $\amalg^* : \f (\fb ) \rightarrow \f (\tfb)$.
\item 
The functor $\tmalg_* : \f ((\tfb)^{\times 2}) \rightarrow \f (\tfb) $ is both left and right adjoint to $\tmalg^* :\f (\tfb) \rightarrow  \f ((\tfb)^{\times 2})$.
\end{enumerate}
\end{prop}

\begin{proof}
The proof is sketched for $\amalg$; the case of $\tmalg$ is similar. 

The functor $\amalg_*$ identifies as a global form of the induction functor. Namely, if $M$ is supported on $(\m,\n)$, then $\amalg_* M$ is supported on $\mathbf{m+n}$ with value $M (\m, \n)\uparrow_{\sym_m \times \sym_n}^{\sym_{m+n}}$. Correspondingly, the functor $\amalg^*$ is a global form of the restriction functor; evaluating on $(\m, \n)$, this restriction functor is $(-)\downarrow_{\sym_m \times \sym_n}^{\sym_{m+n}}$. The fact that induction is left adjoint to restriction implies that $\amalg_*$ is left adjoint to $\amalg^*$.

Working over a finite group, coinduction is naturally isomorphic to induction. The above  therefore implies that $\amalg_*$ is also right adjoint to $\amalg^*$.
\end{proof}

\begin{exam}
\label{exam:units_amalg}
Write  $\kk _\mathbf{0}$ for  the $\kk \fb$-module supported on $\mathbf{0}= \emptyset$ with value $\kk$ and $\kk _{(\mathbf{0}, \mathbf{0})}$ for the $\kk (\tfb)$-module supported on $(\mathbf{0}, \mathbf{0})$ with value $\kk$. Then there are natural isomorphisms 
$
\amalg^* \kk_\mathbf{0} \cong   \kk _{(\mathbf{0}, \mathbf{0})}$ and 
$
\amalg_* \kk _{(\mathbf{0}, \mathbf{0})} \cong   \kk_\mathbf{0}$. 
\end{exam}

 \section{Day convolution}
\label{sect:day_convolution}

This section provides a review of the convolution products on $\kk \fb$-modules and on $\kk (\tfb)$-modules using the functors introduced in Section \ref{subsect:disjoint_union}. These are then used to define the corresponding symmetric and exterior power functors in $\kk (\tfb)$-modules.

\subsection{The convolution products}

The category of $\kk \fb$-modules is equipped with the Day convolution product $\odot$. This can be identified with the composite of the exterior tensor product $\boxtimes : \f (\fb ) \times \f (\fb) \rightarrow \f (\tfb)$ with $\amalg_* : \f(\tfb) \rightarrow \f(\fb)$. Explicitly, for $\kk \fb$-modules $F$ and $G$, and $X$ a finite set, $F \odot G$ is  given by
\[
F \odot G (X) := \bigoplus_{S_1 \amalg S_2 = X} F(S_1) \otimes G (S_2),
\]
where the sum is over ordered decompositions of $X$ into two  subsets. 

The convolution product $\odot$ defines a symmetric monoidal structure on the category of $\kk \fb$-modules, with unit $\kk _\mathbf{0}$. The symmetry $\tau : F \odot G \stackrel{\cong}{\rightarrow} G \odot F$ is induced (when evaluated on $X$ as above) by the isomorphism of vector spaces
\[
F(S_1) \otimes G (S_2) \stackrel{\cong}{\rightarrow} G(S_2) \otimes F(S_1)
\]
given by the symmetry in $\kk$-vector spaces, for each ordered decomposition $S_1\amalg  S_2=X$.

Likewise, one has the Day convolution product $\circledcirc$  for $\kk (\tfb)$-modules, which identifies as the composite of the exterior tensor product $\boxtimes : \f(\tfb) \times \f(\tfb ) \rightarrow \f (\fb^{\times 4})$ with $\tmalg_*$. Explicitly,  for $\kk (\tfb)$-modules $F$ and $G$,  $F \circledcirc G$ is given  by
\begin{eqnarray}
\label{eqn:convolution_bimodules}
F \circledcirc G (X,Y) := \bigoplus_{\substack{S_1 \amalg S_2 = X\\ T_1\amalg T_2 =Y}} F(S_1,T_1) \otimes G (S_2,T_2).
\end{eqnarray}
The convolution product  $\circledcirc$  defines a symmetric monoidal structure, with unit $\kk_{(\mathbf{0},\mathbf{0})}$ and symmetry defined similarly to that for  $\odot$.

\subsection{Symmetric and exterior powers in $\kk (\tfb)$-modules}
\label{subsect:sftb_lftb}

For $d \in \nat$ and $F$ a $\kk (\tfb)$-module, one has the $d$-iterated convolution product $F^{\circledcirc d}$;  the symmetric group $\sym_d$ acts on this by morphisms of $\kk (\tfb)$-modules, using the symmetry of $\circledcirc$. One can thus define the associated symmetric and exterior power functors:

\begin{defn}
\label{defn:stfb_ltfb}
For $d \in \nat$ and $F$ a $\kk (\tfb)$-module, define: 
\begin{eqnarray*}
\stfb^d F &:= &F^{\circledcirc d}/ \sym_d; 
\\
\ltfb^d F &:=& (F^{\circledcirc d}\otimes \sgn_d)/ \sym_d,
\end{eqnarray*}
where $\sym_d$ acts diagonally on $F^{\circledcirc d}\otimes \sgn_d$.
\end{defn}

\begin{exam}
\label{exam:stfb_ltfb}
Suppose that $F$ is a $\kk (\tfb)$-module such that $F (\m, \n)$ is zero unless $n=1$. Then $F^{\circledcirc d} (\m, \n)$ is zero unless $n=d$. In the latter case, one has 
\[
F^{\circledcirc d} (\m, \mathbf{d})=
\bigoplus_{ \substack{f\in \fin (\m, \n) \\ g \in \fb (\mathbf{d} , \n)}}
\bigotimes_{i \in \mathbf{d}} F (f^{-1} (i), \{g(i)\}).
\]
For example, for $d=2$, this corresponds to 
\[
F^{\circledcirc d} (\m, \mathbf{2})= \bigoplus_{U_1 \amalg U_2 = \m} \big( F (U_1, \{1\}) \otimes F(U_2, \{2\}) \  \oplus  \  F (U_1, \{2\}) \otimes F(U_2, \{1\}) \big).
\]

The action of $\sym_d$ on $F^{\circledcirc d}$  corresponds to the natural action on $\fb (\mathbf{d} , \n)$. In particular, this action is free. 
 It follows that, {\em as vector spaces}, 
\[
\stfb^d F (\m, \mathbf{d}) \cong \ltfb^d F  (\m, \mathbf{d}) \cong \bigoplus_{ \substack{f\in \fin (\m, \mathbf{d})}}
\bigotimes_{i \in \mathbf{d}} F (f^{-1} (i),\{1\}).
\]
The isomorphism $\stfb^d  (\m, \mathbf{d}) \cong \ltfb^d  (\m, \mathbf{d})$ is {\em not} in general $\sym_d$-equivariant, due to the presence of the $\sgn_d$  in the definition of $\ltfb^d$. 
\end{exam}

Generalizing the identification given in the above example, we have:

\begin{prop}
\label{prop:stfb_F}
For $F $ a $\kk (\tfb)$-module, $(X, Y)$ an object of $\tfb$, and  $d \in \nat$,  there is an isomorphism
$$
\stfb^d F(X, Y) \cong
\bigoplus_{\{ (S_i, T_i) \mid 1\leq i \leq d\} }  \bigotimes_{i=1}^d F (S_i,T_i),
$$
where the sum is over the (unordered) set of pairs $\{ (S_i, T_i) \mid 1 \leq i \leq d\}$ where $S_i \subseteq X$ (respectively $T_i \subseteq Y$)  are pairwise disjoint subsets such that  $\amalg_i S_i =X $ (resp. $\amalg_i T_i = Y$).
\end{prop}

\begin{proof}
Generalizing (\ref{eqn:convolution_bimodules}), one has 
$$
F^{\circledcirc d} (X, Y) 
= 
\bigoplus _{\substack{\amalg_{i=1}^d S_i = X \\ \amalg_{i=1}^d T_i = Y}} \bigotimes_{i=1}^d F(S_i, T_i).
$$  
The indexing set can be rewritten as the ordered  set of pairs $\{ (S_i, T_i) \mid 1 \leq i \leq d\}$ such that $\amalg_{i=1}^d S_i = X$ and $ \amalg_{i=1}^d T_i = Y$. 

On passing to the quotient by the action of $\sym_d$, one obtains the expression in the statement.
\end{proof}

\begin{rem}
There is an obvious counterpart of Proposition \ref{prop:stfb_F} for $\ltfb^d F$, similarly to Example \ref{exam:stfb_ltfb}.

In both cases, in order to write $ \bigotimes_{i=1}^d F (S_i,T_i)$, one has implicitly chosen an order of the sets $\{ (S_i, T_i) \mid 1 \leq i \leq d \}$. Changing the choice of order changes the order of the factors in the tensor product. In the case of $\ltfb^d$, there is potentially a  sign arising from $\sgn_d$.  
\end{rem}

 \section{Shifting $\kk (\tfb)$-modules}
\label{sect:technical}

This section introduces the shift functors on $\kk (\tfb)$-modules that are associated with the functor $\tmalg$ (using the notation introduced in Section \ref{subsect:disjoint_union}). In Section \ref{subsect:shift_Day_convolution}, we consider their behaviour with respect to the Day convolution product $\circledcirc$.

\subsection{Shift functors for $\kk(\tfb)$-modules}
\label{subsect:shift}

Shift functors are standard tool for studying $\kk \calc$-modules when $\calc$ is equipped with a (symmetric) monoidal structure. Here they are  used to encode `global' restriction functors.

For $\kk (\tfb)$-modules, using the functor $\tmalg$,  one has the following:

\begin{defn}
\label{defn:shift}
For $m, n \in \nat$, let $\shift{m}{n}$ be the endofunctor of $\f(\tfb)$ defined by precomposition with $- \tmalg (\m, \n) : \tfb \rightarrow \tfb$.
\end{defn}

The following statement resumes standard properties of these functors:

\begin{lem}
\label{lem:shift}
\ 
\begin{enumerate}
\item 
For $m, n \in \nat$, the functor $\shift{m}{n}$ is exact and symmetric monoidal with respect to the structure $(\f (\tfb), \otimes,  \kk)$. 
\item 
The functor $\shift{0}{0}$ is naturally isomorphic to the identity functor.
\item 
For $s, t\in \nat$, the composite functor $\shift{m}{n} \circ \shift{s}{t}$ is naturally isomorphic to $\shift{m+s}{n+t}$. 
\end{enumerate}
\end{lem}

\begin{rem}
\ 
\begin{enumerate}
\item 
For $m,n \in \nat$,  the shift functor $\shift{m}{n}$ is naturally isomorphic to the composite $\shift{1}{0} ^{\circ m} \circ \shift{0}{1}^{\circ n}$.  
\item 
At the level of representations, the functor $\shift{m}{n}$ identifies as follows. For $F $ a $\kk(\tfb)$-module, and $p, q \in \nat$, there is an isomorphism of $\kk (\sym_p \times \sym_q)$-modules
$$
\shift{m}{n} F (\mathbf{p}, \mathbf{q}) \cong F (\mathbf{p+m}, \mathbf{q+n}) \downarrow^{\sym_{p+m} \times \sym_{q+n}}_{\sym_p \times \sym_q},
$$
using the inclusions $\sym_p \subset \sym_{p+m}$ and $\sym_q \subset \sym_{q+n}$ induced by $\mathbf{p} \subset \mathbf{p} \amalg \m \cong \mathbf{p+m}$ and $\mathbf{q} \subset \mathbf{q} \amalg \n \cong \mathbf{q+n}$. Thus $\shift{m}{n}$ is given by the above family of restriction functors, for $p, q \in \nat$.
\end{enumerate}
\end{rem}

The above makes the following clear: 

\begin{prop}
\label{prop:adjoints_to_shift}
For $m, n \in \nat$, the endofunctor $\shift{m}{n}$ of $\f(\tfb)$ is both left and right adjoint to 
$$
-  \circledcirc (\kk \sym_m \boxtimes \kk \sym_n) \colon 
\f (\tfb) \rightarrow \f (\tfb),
$$
where $(\kk \sym_m \boxtimes \kk \sym_n)$ is considered as a $\kk (\tfb)$-module supported on $(\m, \n)$.
\end{prop}

\begin{proof}
It suffices to show that the functor $-  \circledcirc (\kk \sym_m \boxtimes \kk \sym_n)$ corresponds to the appropriate induction functor.

Using that $(\kk \sym_m \boxtimes \kk \sym_n)$ is supported on $(\m, \n)$, for a $\kk (\tfb)$-module $F$, one checks from the explicit description of $\circledcirc$ that one has an isomorphism
$$
\big( 
F \circledcirc (\kk \sym_m \boxtimes \kk \sym_n)
\big)
(X,Y) 
\cong 
F (\mathbf{p}, \mathbf{q})\uparrow _{\sym_p \times \sym_q}^{\aut (X) \times \aut(Y)},
$$ 
where $p = |X| -m$ and $q= |Y|-n$. It follows  that $\shift{m}{n}$ is right adjoint to $ -  \circledcirc (\kk \sym_m \boxtimes \kk \sym_n)$. 

That it is also left adjoint to it follows from the fact that, for finite groups, coinduction is naturally isomorphic to induction.
\end{proof}

\subsection{Behaviour of the shift functors with respect to $\circledcirc$}
\label{subsect:shift_Day_convolution}

It is a useful fact that the functors $\shift{1}{0}$ and $\shift{0}{1}$ act as `derivations' with respect to the Day convolution product $\circledcirc$ on $\f (\tfb)$. We also consider their behaviour on the symmetric and exterior power functors $\stfb^*$ and $\ltfb^*$ (see  Definition \ref{defn:stfb_ltfb}). By convention $\stfb^0 = \ltfb^0$ is the identity functor and $\stfb^n$ and $\ltfb^n$ are zero for  negative $n$.

\begin{prop}
\label{prop:shift_day_convolution}
For $F, G \in \ob \f (\tfb)$, there are  natural isomorphisms:
\begin{eqnarray*}
\shift{1}{0} (F \circledcirc G) & \cong & (\shift{1}{0}F) \circledcirc G \ \oplus \ F \circledcirc (\shift{1}{0}G) \\
\shift{0}{1} (F \circledcirc G) & \cong & (\shift{0}{1}F) \circledcirc G \ \oplus \ F \circledcirc (\shift{0}{1}G).
\end{eqnarray*}
For $t \in \nat$,  there are natural isomorphisms 
\begin{eqnarray*}
\shift{1}{0} \stfb^t F  & \cong & (\shift{1}{0}F) \circledcirc \stfb^{t-1}F \\
\shift{0}{1} \stfb^t F & \cong & (\shift{0}{1}F) \circledcirc \stfb^{t-1}F \\
\dbd \stfb^t F & \cong & (\dbd F) \circledcirc \stfb^{t-1} F \ \oplus \ (\shift{1}{0}F)  \circledcirc (\shift{0}{1}F) \circledcirc \stfb^{t-2} F.
\end{eqnarray*}
Likewise, there are natural isomorphisms 
\begin{eqnarray*}
\shift{1}{0} \ltfb^t F  & \cong & (\shift{1}{0}F) \circledcirc \ltfb^{t-1}F \\
\shift{0}{1} \ltfb^t F & \cong & (\shift{0}{1}F) \circledcirc \ltfb^{t-1}F \\
\dbd \ltfb^t F & \cong & (\dbd F) \circledcirc \ltfb^{t-1} F \ \oplus \ (\shift{1}{0}F)  \circledcirc (\shift{0}{1}F) \circledcirc \ltfb^{t-2} F.
\end{eqnarray*}
\end{prop}
 
\begin{proof}
For the first statement, consider $\shift{1}{0} (F \circledcirc G)(X, Y)$; by definition this is equal to  $(F \circledcirc G)(X\amalg \mathbf{1}, Y)$. The latter is given by 
\[
\bigoplus_{\substack{S_1 \amalg S_2 = X \amalg \mathbf{1} \\ T_1 \amalg T_2 = Y}} F(S_1, T_1) \otimes G (S_2,T_2).
\] 
For a given pair $(S_1, S_2)$, there are two distinct possibilities: either $S_1$ contains $1$ or $S_2$ does. It follows that the above can be rewritten as 
\[
\bigoplus_{\substack{U_1 \amalg U_2 = X  \\ T_1 \amalg T_2 = Y}} \big(  F(U_1\amalg \mathbf{1}, T_1) \otimes G (U_2,T_2)
\  \oplus \ 
  F(U_1, T_1) \otimes G (U_2\amalg \mathbf{1},T_2)
\big) 
\] 
and the latter identifies with $\big( (\shift{1}{0}F) \circledcirc G \ \oplus \ F \circledcirc (\shift{1}{0}G)\big) (X,Y)$, as required. The case of $\shift{0}{1}$ is proved similarly.

Now consider $\shift{1}{0} \stfb^d F$. In this case, by Proposition \ref{prop:stfb_F},  
$$
\stfb^d F(X\amalg \mathbf{1}, Y) \cong
\bigoplus_{\{ (S_i, T_i) \mid 1\leq i \leq d\} }  \bigotimes_{i=1}^d F (S_i,T_i),
$$
where the sum is over the (unordered) set of pairs $\{ (S_i, T_i) \mid 1 \leq i \leq d\}$ where $S_i \subseteq X\amalg \mathbf{1}$ (respectively $T_i \subseteq Y$)  are pairwise disjoint subsets such that  $\amalg_i S_i =X \amalg \mathbf{1}$ (resp. $\amalg_i T_i = Y$). Without loss of generality, we may assume that $1 \in S_1$, distinguishing the pair $(S_1, T_1)$. It follows that the expression is naturally isomorphic to  $(\shift{1}{0}F) \circledcirc \stfb^{t-1}F (X,Y) $, as required. The case of $\shift{0}{1}$ is proved similarly and that of $\shift{1}{1}$ then follows by using the first statement, since $\shift{1}{1} \cong \shift{1}{0} \circ \shift{0}{1}$ by Lemma \ref{lem:shift}.
 
The isomorphisms for $\shift{m}{n} \ltfb^t F$ are proved by the same arguments; the signs that arise from the definition of $\ltfb^tF$ only appear when  rearranging the order of the convolution product $\circledcirc$ and do not affect the conclusion.
\end{proof}

\section{The walled Brauer categories}
\label{sect:walled_brauer}

The purpose of this section is to recall the walled Brauer categories and their twisted $\kk$-linear variants. These are closely related to the more classical Brauer algebras (and categories). These are important structures; for example, they are used by Sam and Snowden in \cite{MR3376738},  and also by Raynor in \cite{2024arXiv241220260R}.

The motivation for considering these here is two-fold: they can be used to encode structures arising from (wheeled) operads (see Section \ref{sect:operads})
and they can also  be used in defining a suitable category of edge-directed graphs (see Section \ref{sect:graphs}).

\subsection{First definitions}

We first introduce an {\em ordered} version of the upward walled Brauer category; this is equipped with an $\nat$-grading of the morphisms.

\begin{defn}
\label{defn:uwbord}
Let $\uwbord$ be the category with objects pairs of finite sets $(X, Y)$. For a second object $(U, V)$,  $\hom_{\uwbord} ((X,Y), (U,V))$  is empty unless $|U|-|X|= |V|- |Y|$ is a non-negative integer, say $n \in \nat$ (the degree of the map),  and then
\[
\hom_{\uwbord} ((X,Y), (U,V))
:= 
\hom_{\tfb} ((X \amalg \n, Y \amalg \n), (U, V)).
\]
Given $f \in \hom_{\uwbord} ((X,Y), (U,V))$ represented by the pair of bijections $(f_1: X \amalg \n \rightarrow U, f_2 : Y \amalg \n \rightarrow V)$ and $g \in \hom_{\uwbord} ( (U,V), (W, Z))$ represented by the pair of bijections $(g_1: U \amalg \mathbf{k} \rightarrow W, g_2 : V \amalg \mathbf{k} \rightarrow Z)$, the composite $g \circ f \in \hom_{\uwbord} ((X,Y), (W, Z))$ is represented by the pair of bijections 
\begin{eqnarray*}
&& X \amalg (\mathbf{n+k}) \cong (X \amalg \n) \amalg \mathbf{k}
\stackrel{f_1 \amalg \mathbf{k}} {\rightarrow}
U \amalg \mathbf{k}
\stackrel{g_1}{\rightarrow }
W
\\
\ && Y \amalg (\mathbf{n+k}) \cong (Y \amalg \n) \amalg \mathbf{k}
\stackrel{f_2 \amalg \mathbf{k}} {\rightarrow}
V \amalg \mathbf{k}
\stackrel{g_2}{\rightarrow }
Z.
\end{eqnarray*}
\end{defn}

\begin{rem}
A morphism $f \in \hom_{\uwbord} ((X,Y), (U,V))$ is given by a morphism in $\hom_{\finj \times \finj}((X,Y), (U,V))$ together with two bijections $\n \cong U \backslash \mathrm{image}(X)$ and $\n \cong V \backslash \mathrm{image}(Y)$. Equivalently, the latter corresponds to a pair of bijections
\[
\n \cong  U \backslash \mathrm{image}(X) \stackrel{\alpha, \cong} {\rightarrow} V \backslash \mathrm{image}(Y),
\]
where $\alpha$ can be considered as a pairing between the complements, that is  ordered by the bijection with $\n$.

The symmetric group $\sym_n$ acts on $\hom_{\fin \times \fin} ((X \amalg \n, Y \amalg \n), (U, V))$ by the obvious diagonal action; this corresponds to reordering the pairing. 
\end{rem}

There is an analogous ordered version $\fiord$ of the category $\finj$:

\begin{defn}
\label{defn:fiord}
 Let $\fiord $ be the category with objects finite sets and, if $|Y|-|X|=n \in \nat$, then  $\hom_{\fiord}(X, Y) = \hom_\fb (X \amalg \n, Y)$; if $|X| > |Y|$ then $\hom_{\fiord}(X, Y) =\emptyset$. Composition is defined similarly to that in $\uwbord$.
\end{defn}

The following is clear from the definition:

\begin{lem}
\label{lem:fiord}
\ 
\begin{enumerate}
\item 
There is a full functor $\fiord \rightarrow \finj$ that is the identity on objects and which sends a morphism corresponding to a bijection $X \amalg \n \stackrel{\cong}{\rightarrow} Y$ to its restriction to $X$.
\item 
There is an inclusion $\uwbord \hookrightarrow \fiord \times \fiord$ as the wide subcategory containing only the morphisms of $\hom_{\fiord \times \fiord} ((U,V) , (X,Y))$ such that $|X|-|U|= |Y|-|V|$.
\end{enumerate}
\end{lem}

The  upward walled Brauer category is obtained from $\uwbord$ by forgetting the ordering of the pairing:

\begin{defn}
\label{defn:uwb}
Let $\uwb$, the upward walled Brauer category,  be the category with objects pairs of finite sets $(X, Y)$. For a second object $(U, V)$,  $\hom_{\uwb} ((X,Y), (U,V))$  is empty unless $|U|-|X|= |V|- |Y|$ is a non-negative integer,  say $n$, and then
\[
\hom_{\uwb} ((X,Y), (U,V))
:= 
\hom_{\tfb} ((X \amalg \n, Y \amalg \n), (U, V))/ \sym_n.
\]
Composition is induced by that in $\uwbord$.
\end{defn}

By construction, there is a commutative  diagram of  functors that are the identity on objects
\begin{eqnarray}
\label{eqn:uwbord_to_finj2}
\xymatrix{
\uwbord 
\ar[d]
\ar@{^(->}[r]
&
\fiord \times \fiord
\ar[d]
\\
\uwb
\ar[r]
&
 \finj \times \finj.
}
\end{eqnarray}
Here $\fiord \times \fiord \rightarrow \finj \times \finj$ is the product of copies of the functor $\fiord \rightarrow \finj$ of Lemma \ref{lem:fiord} and the top inclusion is also given by that Lemma.
 The functor $\uwbord \rightarrow \uwb$ is  given by the construction of $\uwb$ from $\uwbord$. 
 
The bottom horizontal morphism retains only the underlying pair of injective maps of a morphism of $\uwb$. This factors across the inclusion of the wide subcategory $(\finj \times \finj)^{\mathrm{diag.mor.}} \subset \finj \times \finj$ containing only morphisms from $(U,V)$ to $(X,Y)$ such that $|X|-|U|= |Y|-|V|$. This subcategory has an $\nat$-grading of the morphisms, defined as for $\uwbord$ and $\uwb$.

\begin{prop}
\label{prop:uwbord_to_finj2}
The functors $\uwbord \rightarrow \uwb$ and $\uwb \rightarrow (\finj \times \finj)^{\mathrm{diag.mor.}}$ derived from 
 (\ref{eqn:uwbord_to_finj2}) are bijections on morphisms in degrees $0$ and $1$.
 In degrees $>1$, both functors are non-bijective surjections on morphisms.
\end{prop}

\begin{proof}
The degree $0$ morphisms in $\uwbord$ (and hence in $\uwb$) identify with $\hom_{\tfb} ((X , Y), (U, V))$ and the latter is canonically isomorphic to $\hom_{\finj \times \finj} ((X , Y), (U, V))$ in degree zero.

In degree $1$, morphisms in $\uwbord$ (and hence in $\uwb$) identify with $\hom_{\tfb} ((X \amalg \mathbf{1} , Y \amalg \mathbf{1}), (U, V))$. The latter is canonically isomorphic to $\hom_{\finj \times \finj} ((X , Y), (U, V))$ in this case, since there is a unique way to extend a
 map in $\hom_{\finj \times \finj} ((X , Y), (U, V))$ to a bijection of the required form.

In degrees $>1$, it is clear that neither of the functors induce bijections on morphisms.
\end{proof}

\subsection{Abstract pairings}

The definition of morphisms in $\uwb$ involves pairings. In this section we extract this notion so as to be able to exhibit natural choices of `generators' for the morphisms  in $\uwb$. 

\begin{defn}
\label{defn:wpair}
For an object $(X,Y)$ of $\fb \times \fb$ and $n \in \nat$, the set of $n$ unordered walled pairs in $(X,Y)$ is 
\[
\wpair _n (X,Y):= \hom_{\finj \times \finj} ((\n, \n) , (X, Y)) / \sym_n, 
\]
using the diagonal action of $\sym_n$ on $(\n, \n)$.
\end{defn}

When working with finite totally-ordered sets, one can define the following category:

\begin{defn}
\label{defn:ouwb}
Let $\ouwb$ be the category with objects pairs of totally-ordered finite sets and $\ouwb ((X,Y), (U, V))$  the subset of $\uwb ((X,Y), (U, V))$ given by the maps such that the underlying injections $X \hookrightarrow U$ and $Y\hookrightarrow V$ are order-preserving. 
\end{defn}

Forgetting the total ordering induces a functor 
$\ouwb  \rightarrow  \uwb $. Moreover, we have the following identification:

\begin{lem}
\label{lem:pair_to_uwb}
Given $s, t \in \nat$, such that $n\leq s$ and $n \leq t$,  there is a bijection
\begin{eqnarray}
\label{eqn:pair_to_uwb}
\wpair _n (\mathbf{s},\mathbf{t})
\cong
\ouwb (((\mathbf{s-n}), (\mathbf{t-n})), (\mathbf{s}, \mathbf{t}) )
\end{eqnarray}
that sends a walled pair represented by $(i : \n \hookrightarrow \mathbf{s}, j : \n \hookrightarrow \mathbf{t})$ to the morphism of $\ouwb$ obtained by extending using the order preserving bijections $(\mathbf{s-n}) \cong \mathbf{s} \backslash \mathrm{image}(i)$ and $(\mathbf{t-n}) \cong \mathbf{t} \backslash \mathrm{image}(j)$.
\end{lem}

\begin{proof}
It is clear that an element of $\wpair_n (\mathbf{s}, \mathbf{t})$ can be extended to a morphism of $\uwb$; by exploiting the natural orders on the sets $\mathbf{a}$, for $a \in \{s, t, s-n, t-n\}$, this extension can be chosen uniquely in $\ouwb (((\mathbf{s-n}), (\mathbf{t-n})), (\mathbf{s}, \mathbf{t}) )$.
\end{proof}

This gives a useful characterization of generators of  $\uwb ((X, Y) , (U, V)) $ as a right  $\aut_{\tfb} (X,Y)$-set:

\begin{prop}
\label{prop:uwb_generators}
Given totally ordered sets $X, Y, U, V$ with $|U|-|X|= |V|-|Y|=n$,  the right $\aut_{\tfb} (X,Y)$-set $\uwb ((X, Y) , (U, V))$ is freely generated by the subset $\ouwb ((X, Y) , (U, V))$, which is isomorphic to $\wpair_n (U,V)$ by the isomorphism (\ref{eqn:pair_to_uwb}).
\end{prop}

For later usage, we introduce the morphisms $\iota_{x,y}$:

\begin{nota}
\label{nota:iota_xy}
For $m,n \in \nat$ positive integers and $x \in \m$, $y \in \n$, denote by  $\iota_{x,y}$ the unique morphism in $\ouwb ((\mathbf{m-1}, \mathbf{n-1}) , (\m, \n))$ corresponding to the pair $(x,y) \in \wpair_1 (\m, \n)$ under the bijection of Lemma \ref{lem:pair_to_uwb}. Via the forgetful functor, $\iota_{x,y}$ can be considered as a morphism in $\uwb$. 
\end{nota}

\subsection{The twisted $\kk$-linear version}

Consider the $\kk$-linear categories $\kk \uwbord$ and $\kk \uwb$, which are related by the $\kk$-linear functor 
$
\kk \uwbord 
\rightarrow 
\kk \uwb$
 induced by  $\uwbord \rightarrow \uwb$. 
 On morphisms of degree $n$, we have the explicit identification
\begin{eqnarray*}
\hom_{\kk \uwb} ((X,Y), (U,V))
&= &
 \kk \hom_{\tfb} ((X \amalg \n, Y \amalg \n), (U, V))\otimes_{\sym_n} \triv_n,
\end{eqnarray*}
where $\hom_{\tfb} ((X \amalg \n, Y \amalg \n), (U, V))$ identifies with $\hom_{\uwbord} ((X,Y) , (U,V))$.

This motivates the definition of the following twisted form of $\kk \uwb$, in which $\triv$ is replaced by $\sgn$:

\begin{defn}
\label{defn:uwbtw}
Let $\uwbtw$ be the $\kk$-linear category with objects pairs of finite sets $(X, Y)$. For a second object $(U, V)$,  $\hom_{\uwbtw} ((X,Y), (U,V))$  is empty unless $|U|-|X|= |V|- |Y|$ is a non-negative integer,  say $n\in \nat$, and then
\[
\hom_{\uwbtw} ((X,Y), (U,V))
:= 
 \kk \hom_{\tfb} ((X \amalg \n, Y \amalg \n), (U, V)) \otimes_{\sym_n} \sgn_n.
\]
Composition is induced by that in $\uwbord$.
\end{defn}

\begin{rem}
\ 
\begin{enumerate}
\item 
The verification that the composition of $\kk \uwbord$ passes to $\uwbtw$ is straightforward. 
\item 
By construction, there is a full functor $\kk \uwbord \rightarrow \uwbtw$ that is the identity on objects. 
\item 
The degree of morphisms of $\kk \uwbord$ passes to $\uwbtw$ and the wide subcategory of degree zero morphisms identifies with $\kk (\tfb)$. 
\item 
The morphism $\iota_{x,y}$ of Notation \ref{nota:iota_xy} yields a canonical element 
\[
\ 
[\iota_{x,y}] \in \uwbtw ((\mathbf{m-1}, \mathbf{n-1}) , (\m, \n)) = \kk \uwb ((\mathbf{m-1}, \mathbf{n-1}) , (\m, \n)),
\]
where the equality uses the identification $\sgn_1= \triv_1$. 
\end{enumerate}
\end{rem}

We also have a  twisted form $\fitw$ of $\kk \finj$:

\begin{defn}
\label{defn:fitw}
The $\kk$-linear category $\fitw$ has finite sets for objects; morphisms are given for $|Y|-|X|=n \in \nat$ by 
\[
\hom_{\fitw} (X, Y) =  \kk \hom_\fb (X \amalg \n, Y) \otimes_{\sym_n} \sgn_n
\]
and $\hom_{\fitw} (X, Y) = 0$ for $|X|> |Y|$. The integer $n$ corresponds to the degree. Composition is induced by that of $\fiord$.
\end{defn}

By construction there is a full, $\kk$-linear functor $\kk \fiord \rightarrow \fitw$ that is the identity on objects. Moreover, there is an inclusion $\kk \fb \hookrightarrow \fitw$ that is the identity on objects, corresponding to the subcategory of degree zero morphisms.

\begin{exam}
\label{exam:sgn_fitw-module}
There is an important example of a $\fitw$-module provided by the sign representations. This has underlying  $\kk \fb$-module $\sgn$, i.e.,  as a $\kk \fb$-module it identifies as $X \mapsto \Lambda^{|X|} (\kk X)$. For the inclusion $X \subset X \amalg \{z\}$, the associated linear map 
\[
\Lambda^{|X|} (\kk X)
\rightarrow 
\Lambda^{|X|+1} (\kk (X\amalg \{z\}))
\]  
is given by the wedge product on the right with $[z]$. This defines a $\fitw$-module; it will be denoted simply by $\sgn$. 
\end{exam}

\begin{prop}
\label{prop:uwbtw->fitw_otimes_kkfinj}
\ 
\begin{enumerate}
\item 
The $\kk$-linearization of the inclusion of the wide sub-category $\uwbord\hookrightarrow  \fiord \times \fiord$ composed with the tensor product of the functors $\kk \fiord \rightarrow \fitw$ and $\kk \fiord \rightarrow \kk \finj$ induces 
 a $\kk$-linear functor 
$$
\uwbtw \rightarrow \fitw \otimes \kk \finj
$$
that is the identity on objects. 
\item 
The $\kk$-linearization of the functor $\uwbord \rightarrow \fiord$ given on objects by $(X, Y) \mapsto X$ composed with $\kk \fiord \rightarrow \fitw$ induces  a $\kk$-linear functor 
$$
\uwbtw \rightarrow \fitw.
$$
\end{enumerate}
\end{prop}

\begin{proof}
For the first statement, one uses the isomorphism of $\kk$-linear categories $\kk (\fiord \times \fiord) \cong \kk \fiord \otimes \kk \fiord$. Composing the functors in the statement gives the $\kk$-linear functor 
$$
\kk \uwbord \rightarrow \fitw \otimes \kk \finj.
$$
To conclude, it remains to show that this factors across $\kk \uwbord \rightarrow \uwbtw$. 

This reduces to the following general fact about right $\kk \sym_n$-modules. Given two such, $M$ and $N$, one has the surjective $\kk$-linear maps $M \rightarrow M \otimes_{\kk \sym_n} \sgn_n$ and $N \rightarrow N \otimes_{\kk \sym_n} \triv_n$ given respectively by $m \mapsto m \otimes 1$ and $n \mapsto n \otimes 1$. Their tensor product gives the $\kk$-linear map 
$$
M \otimes N 
\rightarrow 
 (M \otimes_{\kk \sym_n} \sgn_n) \otimes (N \otimes_{\kk \sym_n} \triv_n).
 $$
 One checks directly that this factors across the surjection $M \otimes N \twoheadrightarrow (M \otimes N) \otimes_{\kk \sym_n} \sgn_n$, using the diagonal $\sym_n$ action on $M \otimes N$. This is applied taking $M = \kk \fb (X \amalg \n, U) $ and $N= \kk \fb (Y \amalg \n , V)$.

The second statement follows from the first. One has the canonical functor $\finj \rightarrow *$ to the terminal category. This induces $\kk \finj \rightarrow \kk*$ and hence $\fitw \otimes \kk \finj \rightarrow \fitw$. One checks that the composite
$$
\kk \uwbord \rightarrow 
\uwbtw \rightarrow \fitw \otimes \kk \finj \rightarrow \fitw
$$ 
thus obtained identifies with the functor described in the statement. 
\end{proof}

\begin{rem}
In Proposition \ref{prop:uwbtw->fitw_otimes_kkfinj} we have made a choice to twist the first factor. One can equally well do this with the second. This is explained formally in Section \ref{subsect:involutions} below (see Example \ref{exam:involution_sgn_boxtimes_triv}).
\end{rem}

Proposition \ref{prop:uwbtw->fitw_otimes_kkfinj} has the immediate consequence:

\begin{cor}
\label{cor:sgn_boxtimes_triv}
There is a $\uwbtw$-module given by the restriction along $\uwbtw \rightarrow \fitw \otimes \kk \finj$ of the exterior tensor product $\sgn \boxtimes \triv$. 
 This is isomorphic to the $\uwbtw$-module obtained by restriction along $\uwbtw \rightarrow \fitw $ of the $\fitw$-module $\sgn$.
\end{cor}

\subsection{Involutions}
\label{subsect:involutions}

For $\calc$ a category, the product category $\calc \times \calc$ admits the involution 
$ \sigma_{\calc \times \calc} : 
\calc \times \calc \rightarrow \calc \times \calc $ given on objects by 
$
(X, Y) \mapsto  (Y,X)$. This is an isomorphism of categories such that $\sigma_{\calc\times \calc}^2 = \id_{\calc\times \calc}$. For example, we have the involutions $\sigma_{\tfb}$ and $\sigma_{\fiord \times \fiord}$. Moreover,  the subcategory $\uwbord$ of  $\fiord \times \fiord$ inherits an  involution $\sigma_{\uwbord}$ from $\sigma_{\fiord \times \fiord}$. This is given on objects by $\sigma_{\uwbord} (X,Y) = (Y,X)$ and likewise on morphisms.

The involution of $\uwbord$ passes to the category $\uwb$ and the $\kk$-linear category  $\uwbtw$:

\begin{prop}
\label{prop:uwb_involution}
The involution on $\uwbord$ induces:
\begin{enumerate}
\item 
an involution 
$
\sigma_{\uwb}  \colon   \uwb \rightarrow \uwb 
$
 that is compatible with $\sigma_{\tfb}$ via the inclusion $\tfb \hookrightarrow \uwb$;
\item 
a $\kk$-linear involution 
$
\sigma_{\uwbtw}  \colon  \uwbtw \rightarrow \uwbtw
$
 that is  compatible with  the $\kk$-linear involution $\kk \sigma_{\tfb}$ via the inclusion  $\kk(\tfb) \hookrightarrow \uwbtw$.
 \end{enumerate}
\end{prop}

\begin{proof}
This is immediate from the definitions of the categories $\uwb$ and $\uwbtw$.
\end{proof}

This has the immediate consequence:

\begin{cor}
\label{cor:involutions_modules}
The $\kk$-linear involutions $\kk \sigma_\uwb$ and $\sigma_{\uwbtw}$ induce involutions of the module categories $\kk \uwb\dash \modules$ and $\uwbtw\dash\modules$ respectively. 
\end{cor}

\begin{exam}
\label{exam:involution_sgn_boxtimes_triv}
Under the involution of $\uwbtw\dash\modules$, the image of $\sgn \boxtimes \triv$ is the $\uwbtw$-module with underlying $\kk(\tfb)$-module $\triv \boxtimes \sgn$.
\end{exam}
\subsection{Relating modules over $\kk \uwbord$, $\kk \uwb$, and $\uwbtw$}

By construction, there are full, $\kk$-linear functors 
\[
\kk\uwb 
\longleftarrow 
\kk \uwbord 
\longrightarrow 
\uwbtw
\]
that are the identity on objects. Restriction along these induce the functors between module categories
\[
\kk\uwb \dash\modules 
\longrightarrow
\kk \uwbord \dash \modules
\longleftarrow
\uwbtw\dash\modules .
\]
Here, $\kk \uwbord \dash \modules$ identifies with $\f (\uwbord)$.

These are inclusions of full subcategories; their essential images can be identified explicitly, using the following notation. Recall $\iota_{x,y}$ introduced in Notation \ref{nota:iota_xy}; then, for $m,n \in \nat$, consider $\iota_{m+2,n+2} \circ \iota_{m+1, n+1}$ as a morphism of $\uwbord ((\m, \n), (\mathbf{m+2}, \mathbf{n+2}))$; write $\rho \in \aut _{\fb \times \fb} (\mathbf{m+2}, \mathbf{n+2})$ for the element given by the pair of transpositions $(m+1, m+2)$ and $(n+1, n+2)$. 

\begin{prop}
\label{prop:characterize_module_categories}
For $F$ in $\f (\uwbord)$, 
\begin{enumerate}
\item 
$F$ belongs to the essential image of $\f(\uwb) \cong \kk \uwb \dash \modules$ if and only if, for all $m,n \in \nat$, 
$$
F(\iota_{m+2,n+2} \circ \iota_{m+1, n+1}) = F (\rho \circ \iota_{m+2,n+2} \circ \iota_{m+1, n+1}); 
$$
\item 
$F$ belongs to the essential image of $\uwbtw \dash \modules$ if and only if, for all $m,n \in \nat$, 
$$
F(\iota_{m+2,n+2} \circ \iota_{m+1, n+1}) = - F (\rho \circ \iota_{m+2,n+2} \circ \iota_{m+1, n+1}).
$$
\end{enumerate}
\end{prop}

\begin{proof}
This is a consequence of the fact that the categories $\kk \uwb$ and $\uwbtw$ are homogeneous quadratic $\kk$-linear categories over $\kk (\tfb)$ (as reviewed below in Section \ref{sect:uwb_quadratic_koszul} 	 and the explicit quadratic presentation given in 
Proposition \ref{prop:kk_uwb_quadratic}).

Alternatively, it can be proved directly by using the construction of $\kk \uwb$ and $\uwbtw$ from $\kk \uwbord$. The key ingredient is that the symmetric groups are generated by their `adjacent' transpositions.
\end{proof}

The category $\f (\uwbord)$ is  equipped with the pointwise tensor product $\otimes$. Moreover, we can consider $\sgn \boxtimes \triv$ as an object of this category. The endofunctor $(\sgn \boxtimes \triv) \otimes - : \f (\uwbord) \rightarrow \f (\uwbord)$ is clearly an involution. 

\begin{cor}
\label{cor:untwist}
The endofunctor $(\sgn \boxtimes \triv) \otimes - : \f (\uwbord) \rightarrow \f (\uwbord)$ induces mutually inverse equivalences 
\[
\kk \uwb \dash \modules \stackrel{\simeq} {\longleftrightarrow} \uwbtw \dash\modules. 
\]
\end{cor}

\begin{proof}
It is clear that $(\sgn \boxtimes \triv)$ lives in the essential image of $\uwbtw \dash\modules$ in $\f (\uwbord)$ (cf. Corollary \ref{cor:sgn_boxtimes_triv}). Using the criterion of Proposition \ref{prop:characterize_module_categories}, one deduces that the functor $(\sgn \boxtimes \triv) \otimes - $ induces quasi-inverse equivalences between the two essential images. This yields the required equivalences. 
\end{proof}

\begin{rem}
\label{rem:counterpart_triv_boxtimes_sgn}
One has the counterpart of this result using the object $(\triv \boxtimes \sgn)$ of $\f (\uwbord)$ (compare Example \ref{exam:involution_sgn_boxtimes_triv}). 
 Namely, the endofunctor $(\triv \boxtimes \sgn) \otimes - \colon \f (\uwbord) \rightarrow \f (\uwbord)$ induces mutually inverse equivalences 
\[
\kk \uwb \dash \modules \stackrel{\simeq} {\longleftrightarrow} \uwbtw \dash\modules. 
\]
 
The two equivalences thus obtained are not the same. Their difference is given by the   $\kk \uwbord$-module $\sgn \boxtimes \sgn$, which  lies in the full subcategory corresponding to $\kk \uwb\dash\modules$. The functor $(\sgn \boxtimes \sgn) \otimes -$ defined on $\f (\uwbord)$ restricts to involutions of both of the  full subcategories of $\f (\uwbord)$, corresponding respectively to $\kk \uwb\dash\modules$ and $\uwbtw$-modules.
\end{rem}

\subsection{Downward counterparts}

The downward walled Brauer categories and their variants are defined as the respective opposite categories, so that $\dwbord:= (\uwbord)\op$, $\dwb := \uwb\op$, and $\dwbtw:= \uwbtw\op$. Hence there are the corresponding functors 
\begin{eqnarray*}
&&\dwbord \rightarrow \dwb \rightarrow \finj \op \times \finj\op
\\
&&
\kk \dwbord 
\rightarrow
\dwbtw .
\end{eqnarray*}

Likewise, the category $\odwb$ is the opposite category $\ouwb\op$, so that there is a forgetful functor $\odwb \rightarrow \dwb$.

\begin{rem}
The morphism $\iota_{x,y}$ of Notation \ref{nota:iota_xy} yields 
$\iota_{x,y}\op$  in $\odwb ( (\m, \n),(\mathbf{m-1}, \mathbf{n-1}) )$ (and hence, via the forgetful functor, in $\dwb$). This also yields $[\iota_{x,y}\op] \in \dwbtw ( (\m, \n),(\mathbf{m-1}, \mathbf{n-1}) )$. 
\end{rem}

As in the upward case, we have the inclusions of full subcategories:
\[
\kk\dwb \dash\modules 
\longrightarrow
\kk \dwbord \dash \modules
\longleftarrow
\dwbtw\dash\modules. 
\]
Here, $\kk \dwbord \dash \modules$ identifies with $\f (\dwbord)$. 
The respective essential images can be identified, as in Proposition \ref{prop:characterize_module_categories}.
  Moreover, there is a corresponding object $(\sgn \boxtimes \triv)$ of $\f (\dwbord)$, which lies in the full subcategory corresponding to $\dwbtw\dash\modules$. 
One  has the following counterpart of Corollary \ref{cor:untwist}:

\begin{cor}
\label{cor:untwist_dwb}
The endofunctor $(\sgn \boxtimes \triv) \otimes - : \f (\dwbord) \rightarrow \f (\dwbord)$ induces mutually inverse equivalences 
$
\kk \dwb \dash \modules \stackrel{\simeq} {\longleftrightarrow} \dwbtw \dash\modules$. 
\end{cor}

\section{Torsion modules and associated adjunctions}
\label{sect:torsion}

The purpose of this section is to review the notion of torsion for $\kk \uwb$-modules (by Corollary \ref{cor:untwist}, the results can be translated to $\uwbtw$-modules). This is the obvious generalization of the usual notion of torsion over an augmented $\kk$-algebra. (Cf. the usage of torsion $\finj$-modules, introduced in \cite[Section 2]{MR3430359}.)  

We  introduce the important adjunction of Proposition \ref{prop:kk^uwb_otimes_vanish_torsion} that is constructed by using $\kk^\uwb$, considered as a $\kk\uwb$-bimodule. This leads to the equivalence of categories given in Corollary \ref{cor:equivalence_uwbup_uwbdown}.

For most of this section, $\kk$ can be taken to be an arbitrary field; however, for some results, we require characteristic zero.

\subsection{Torsion}

Before introducing the notion of torsion, we give a few recollections. 

As usual, we say that a $\kk \uwb$-module $M$ has finite support if the set of pairs $(s, t)$ of natural numbers for which $M (\mathbf{s}, \mathbf{t})$ is not zero is finite. Recall that, for $m,n \in \nat$,  $\kk \uwb ((\m, \n), -) $ is projective, corepresenting the evaluation $M \mapsto M (\m,\n)$, and $\kk^{\uwb (-, (\m, \n))}$ is injective, representing $M \mapsto M(\m,\n)^\sharp$. In particular, an element $x \in M (\m,\n)$ (we shall term $x$ a section of $M$) is equivalent to a morphism $\kk \uwb ((\m,\n), -) \stackrel{x}{\rightarrow} M$, by Yoneda. 

\begin{defn}
\label{defn:torsion}
For $M$ a $\kk \uwb$-module, 
\begin{enumerate}
\item 
a section $x \in M (\m,\n)$ is torsion if the image of  $\kk \uwb ((\m,\n), -) \stackrel{x}{\rightarrow} M$ has finite support;
\item 
the module $M$ is torsion if every section is torsion; 
\item 
the module $M$ is torsion-free if it contains no non-zero torsion submodule.
\end{enumerate}
Write $\torskuwb$ for the full subcategory of torsion $\kk \uwb$-modules.
\end{defn}

\begin{exam}
\label{exam:non-torsion/torsion}
For $m, n \in \nat$,
\begin{enumerate}
\item 
the projective module $\kk \uwb ((\m, \n), -) $ is torsion-free; 
\item 
the injective module $\kk^{\uwb (-, (\m,\n))}$ is torsion; more precisely it has finite support. 
\end{enumerate}
\end{exam}

We record the following stronger property of $\kk^{\uwb (-, (\m,\n))}$ (which only depends on the fact that this module takes finite-dimensional values and has finite support):

\begin{prop}
\label{prop:inj_compact}
For $m, n \in \nat$, the injective $\kk^{\uwb (-, (\m,\n))}$ has a finite projective presentation
$$
P_1 \rightarrow P_0 \rightarrow \kk^{\uwb (-, (\m,\n))} \rightarrow 0,
$$
where $P_0$ and $P_1$ are finite direct sums of projectives of the form $\kk \uwb ((\mathbf{s}, \mathbf{t}), -)$, where $m-s = n-t$. 

In particular, the module   $\kk^{\uwb (-, (\m,\n))}$ is compact, i.e., the functor $\hom_{\kk \uwb}(\kk^{\uwb (-, (\m,\n))}, -) $ commutes with filtered colimits.
\end{prop}

\begin{proof}
We give a direct, hands-on proof.

We can take $P_0$ to be the projective $\bigoplus_{k \in \nat} \kk^{\uwb ((\mathbf{m-k}, \mathbf{n-k}) , (\m,\n)) } \otimes_\kk \kk \uwb ((\mathbf{m-k}, \mathbf{n-k}), -)$ equipped with the map $P_0 \rightarrow \kk^{\uwb (-, (\m,\n))}$ induced by evaluation. (Here the terms with $m-k<0$ or $n-k<0$ are understood to be zero, so the sum is finite.) 
 This is isomorphic to a finite direct sum of projectives of the required form and, by construction, the evaluation map is surjective.

Consider the kernel $K$ of $P_0 \rightarrow \kk^{\uwb (-, (\m,\n))}$; this takes finite-dimensional values and is supported on terms of the form $(\mathbf{s}, \mathbf{t})$ where $m-s = n-t$. Moreover, for $s>m$, $K(\mathbf{s}, \mathbf{t}) = P_0 (\mathbf{s}, \mathbf{t})$; it follows that the submodule of $K$ generated by $K (\mathbf{m+1}, \mathbf{n+1})$ contains all the sections   $K(\mathbf{s}, \mathbf{t})$ for $s>m$ (equivalently, $t>n$), since this is true for $P_0$, using the fact that morphisms of $\uwb$ are generated by those of degree zero and degree $1$.

Similarly to the construction of $P_0$, take $P_1$ to be the projective 
$$
\bigoplus_{k \in \nat} K(\mathbf{m+1-k}, \mathbf{n+1-k})  \otimes \kk \uwb ((\mathbf{m+1-k}, \mathbf{n+1-k}), -).
$$
This is isomorphic to a finite direct sum of projectives of the required form. Moreover,
by construction (by the previous remarks), the  map $P_1 \rightarrow K$ given by evaluation is surjective. Thus, this provides the required projective presentation.

The final statement is a formal consequence of this.
\end{proof}

We have the following well-known property of  torsion modules:

\begin{thm}
\label{thm:localizing}
The subcategory $\torskuwb$ is a localizing Serre subcategory of $\kk \uwb$-modules. Namely:
\begin{enumerate}
\item 
for $0\rightarrow M_1 \rightarrow M_2 \rightarrow M_3 \rightarrow 0$ a short exact sequence of $\kk \uwb$-modules, $M_2$ is torsion if and only if both $M_1 $ and $M_3$ are torsion; 
\item 
the subcategory $\torskuwb$ is closed under arbitrary coproducts.
\end{enumerate}
\end{thm}

This implies that one has the localization {\em à la Gabriel}
$$
\pi 
\colon 
\kk \uwb\dash\modules 
\rightarrow 
\kk \uwb\dash\modules /\torskuwb
$$
and this admits a right adjoint $\mathbf{s} \colon \kk \uwb\dash\modules /\torskuwb\rightarrow \kk \uwb\dash\modules $, the saturation functor (see \cite{MR232821} for the general theory). Thus, for any $\kk \uwb$-module $M$, there is a natural short exact sequence 
$$
0
\rightarrow 
\mathsf{tors} M 
\rightarrow 
M 
\rightarrow 
\mathbf{s}\pi M 
$$ 
where the right hand map is the adjunction unit; $\mathsf{tors} M$ is the largest torsion submodule of $M$ and $\mathsf{tors}$ yields the right adjoint to the inclusion of $\torskuwb$ into $\kk\uwb$-modules.  One has the corresponding short exact sequence 
$
0
\rightarrow 
\mathsf{tors} M 
\rightarrow 
M 
\rightarrow 
M / \mathsf{tors} M
\rightarrow 
0
$, 
in which $M/  \mathsf{tors} M$ is torsion-free and embeds in $\mathbf{s}\pi M$.

\subsection{An adjunction}
\label{subsect:torsion_adjunction}

We consider $\kk ^\uwb$ as a $\kk \uwb$-bimodule and hence the functor:
$$
\kk ^\uwb \otimes_{\kk \uwb} - \colon \kk\uwb\dash\modules \rightarrow \kk \uwb\dash\modules.
$$
Explicitly, for a $\kk \uwb$-module $M$, the evaluation of $\kk ^\uwb \otimes_{\kk \uwb} M $ on $(\m, \n)$ is given by $\kk^{\uwb ((\m, \n), -)} \otimes _{\kk \uwb} M$, using the right $\kk \uwb$-module structure of $\kk^{\uwb ((\m, \n), -)}$ when forming the tensor product.

\begin{prop}
\label{prop:kk^uwb_otimes_torsion}
For a $\kk \uwb$-module $M$, the $\kk\uwb$-module $\kk ^\uwb \otimes_{\kk \uwb} M $  is torsion. Hence
 $\kk ^\uwb \otimes_{\kk \uwb} -$ defines a functor 
$$
\kk ^\uwb \otimes_{\kk \uwb} - \colon 
\kk\uwb \dash\modules \rightarrow \torskuwb.
$$ 
\end{prop}

\begin{proof}
By construction of the tensor product $\otimes_{\kk \uwb}$, as a $\kk \uwb$-module,  $\kk ^\uwb \otimes_{\kk \uwb} M $ is a quotient of $$\bigoplus_{s,t} \kk^{\uwb (-, (\mathbf{s}, \mathbf{t}))} \otimes _{\kk} M(\mathbf{s}, \mathbf{t}).$$
By Theorem \ref{thm:localizing}, the latter is a torsion $\kk \uwb$-module, since each $ \kk^{\uwb (-, (\mathbf{s}, \mathbf{t}))}$ is (see Example \ref{exam:non-torsion/torsion}).  The second statement follows immediately.
\end{proof}

We have the following property: 

\begin{prop}
\label{prop:kk^uwb_otimes_vanish_torsion}
If $M$ is a torsion $\kk \uwb$-module, then $\kk ^\uwb \otimes_{\kk \uwb} M $ is zero. Hence, for any $\kk\uwb$-module $N$, the canonical surjection $N \twoheadrightarrow N/ \mathsf{tors}N$ induces an isomorphism
$$
\kk ^\uwb \otimes_{\kk \uwb} N
\stackrel{\cong}{\rightarrow}
\kk ^\uwb \otimes_{\kk \uwb} N /\mathsf{tors}N. 
$$
\end{prop}

\begin{proof}
For the first statement, it suffices to prove that $\hom_\kk (\kk ^{\uwb ((\m, \n), -)}  \otimes_{\kk \uwb} M, \kk)$ is zero for any $m, n \in \nat$. By the universal property of $\otimes_{\kk \uwb}$, this is isomorphic to 
$\hom_{\kk \uwb}(M, \kk \uwb ((\m, \n), -))$, using that $\kk^\uwb$ takes finite-dimensional values, so that the dual of $\kk^\uwb$ is $\kk \uwb$. Now, as in Example \ref{exam:non-torsion/torsion}, $\kk \uwb ((\m, \n), -)$ is torsion-free. This implies that $\hom_{\kk \uwb}(M, \kk \uwb ((\m, \n), -))=0$, since $M$ is torsion by hypothesis.

The second statement then follows from the right exactness of $\kk ^\uwb \otimes_{\kk \uwb}-$.
\end{proof}

The functor $\kk ^\uwb \otimes_{\kk \uwb} -$ has right adjoint $\hom_{\kk \uwb} (\kk ^\uwb, -)$. To clarify the functoriality, we make the latter explicit: for $M$ a $\kk \uwb$-module and $m, n \in \nat$:
 $$\hom_{\kk \uwb} (\kk ^\uwb, M)(\m, \n) =\hom_{\kk \uwb}(\kk ^{\uwb (-, (\m,\n))} , M) .$$
 Since the module $\kk ^{\uwb (-, (\m,\n))}$ is torsion, the latter is isomorphic to $\hom_{\kk \uwb}(\kk ^{\uwb (-, (\m,\n))} ,\mathsf{tors} M)$. Thus we may as well restrict the functor $\hom_{\kk \uwb} (\kk ^\uwb, -)$ to torsion $\kk \uwb$-modules. To summarize:

\begin{prop}
\label{prop:kk^uwb_adjunction}
There is an adjunction:
$$
\xymatrix{
\kk ^\uwb \otimes_{\kk \uwb} -
\ar@{}[r]| :  
&
\kk \uwb\dash\modules 
\ar@<1ex>[rr]
\ar@{}[rr] |\perp 
&&
\torskuwb
\ar@<1ex>[ll]
\ar@{}[r]|:
&
\hom_{\kk \uwb} (\kk ^\uwb, -).
}
$$
\end{prop}

\begin{exam}
\label{exam:kk^ubg_proj_inj}
For $m,n\in \nat$, by Yoneda's lemma, we have isomorphisms
\begin{enumerate}
\item 
$\kk ^\uwb \otimes_{\kk \uwb} \kk \uwb ((\m, \n), -) \cong \kk ^{\uwb (-, (\m, \n))}$; 
\item 
$\hom_{\kk \uwb} (\kk ^\uwb, \kk ^{\uwb (-, (\m, \n))}) \cong \kk \uwb ((\m, \n), -)$. 
\end{enumerate}
Thus the adjoint functors of Proposition \ref{prop:kk^uwb_adjunction} give a one-one correspondence between the set of standard projective modules and that of  standard injective modules. 
\end{exam}

\subsection{Restricting the adjunction}

Using the canonical inclusion $\kk (\tfb) \hookrightarrow \kk \uwb$, one has the  induction functor 
$$
\kk \uwb \otimes_{\kk (\tfb)} -\ : \ \f (\tfb) \rightarrow \f (\uwb)
$$
that is left adjoint to the restriction $\f (\uwb) \rightarrow \f (\tfb)$.
 Composing this with the functor $\kk ^\uwb \otimes_{\kk \uwb}-$ gives
$$
\kk ^\uwb \otimes_{\kk (\tfb)} - \ : \  \kk (\tfb)\dash\modules \rightarrow \kk \uwb \dash\modules.
$$
Then, for $N$ a $\kk (\tfb)$-module, we have the adjunction unit for the adjunction of Proposition \ref{prop:kk^uwb_adjunction} (applied to $\kk \uwb \otimes _{\kk (\tfb)} N$):
\begin{eqnarray}
\label{eqn:N_adj_unit}
\kk \uwb \otimes_{\kk (\tfb)} N
\rightarrow 
\hom_{\kk \uwb} (\kk^\uwb, \kk ^\uwb \otimes_{\kk (\tfb)} N).
\end{eqnarray}

We have the following generalization of the behaviour exhibited in Example \ref{exam:kk^ubg_proj_inj}:

\begin{prop}
\label{prop:equivalence_N_adj_unit}
Suppose that $\kk$ is a field of characteristic zero. Then the natural transformation (\ref{eqn:N_adj_unit}) is an isomorphism for all $N$. 
\end{prop}

\begin{proof}
We require to prove that, for any $m, n \in \nat$, the morphism 
$$
\kk {\uwb( (\m, \n), -) }  \otimes_{\kk (\tfb)} N
\rightarrow 
\hom_{\kk \uwb} (\kk^{\uwb(-, (\m,\n))}, \kk ^\uwb \otimes_{\kk (\tfb)} N)
$$
is an isomorphism. Now, since $\kk^{\uwb(-, (\m,\n))}$ is compact, by Proposition \ref{prop:inj_compact}, we may reduce to the case where $N$ is supported on a single object, say $(\mathbf{s}, \mathbf{t})$, and has finite dimension.

In the case that $N$ is $\kk (\sym_s \times \sym_t)$ (considered as a $\kk (\tfb)$-module), the morphism is an isomorphism; this encodes the behaviour exhibited in Example \ref{exam:kk^ubg_proj_inj}. This  generalizes to the case where $N$ is a finite direct sum of such, say  $N= \kk (\sym_s \times \sym_t)^{\oplus d}$, for $d \in \nat$. 

To conclude, by the hypothesis that $\kk$ has characteristic zero, $N$ is a direct summand of a $\kk (\tfb)$-module of the form $\kk (\sym_s \times \sym_t)^{\oplus d}$ for some $d \in \nat$. By the above, the morphism (\ref{eqn:N_adj_unit}) is a retract of an isomorphism, hence is an isomorphism.
\end{proof}

To state the consequence of this result, we introduce the following full subcategories of $\f (\uwb)$:

\begin{defn}
\label{defn:uwbup_uwbdown}
Let 
\begin{enumerate}
\item 
$\uwbup$ be the full subcategory of $\f(\uwb)$ corresponding to the essential image of $\kk\uwb \otimes_{\kk(\tfb)} -$;
\item 
$\uwbdown$ be the full subcategory of $\f(\uwb)$ corresponding to the essential image of $\kk^\uwb \otimes_{\kk (\tfb)} -$. 
\end{enumerate}
\end{defn}

\begin{cor}
\label{cor:equivalence_uwbup_uwbdown}
Suppose that $\kk$ is a field of characteristic zero. 
The adjunction of Proposition \ref{prop:kk^uwb_adjunction} restricts to give an equivalence of categories:
\[
\xymatrix{
\kk^\uwb \otimes_{\kk\uwb} -
\ar@{}[r]|:
&
\uwbup
\ar@<1ex>[r]^{\simeq}
\ar@{}[r]|\perp 
&
\uwbdown 
\ar@<1ex>[l]
&
\hom_{\kk\uwb} (\kk^\uwb, -). 
\ar@{}[l]|(.6):
}
\]
\end{cor}

\begin{proof}
It is clear that $\kk^\uwb \otimes_{\kk\uwb} -$ restricted to $\uwbup$ takes values in $\uwbdown$. That $\hom_{\kk\uwb} (\kk^\uwb, -)$ restricted  to $\uwbdown$ takes values in 
$\uwbup$ follows from Proposition \ref{prop:equivalence_N_adj_unit}. It follows that we have an adjunction as stated.

It remains to show that this induces an equivalence. This follows from Proposition \ref{prop:equivalence_N_adj_unit}.
\end{proof}

\begin{rem}
Let $M$ and $N$ be $\kk (\tfb)$-modules. It is instructive to consider the isomorphism between $\hom _{\kk \uwb} (\kk \uwb \otimes_{\kk (\tfb)} M, \kk \uwb\otimes_{\kk (\tfb)} N)$ and $\hom _{\kk \uwb} (\kk^\uwb \otimes_{\kk (\tfb)} M, \kk^\uwb\otimes_{\kk (\tfb)} N)$ directly. 

The former is isomorphic to $\hom _{\kk (\tfb)} ( M, \kk \uwb\otimes_{\kk (\tfb)} N)$ which is, in turn, isomorphic to 
$$\hom _{\kk (\tfb)} (\kk^\uwb \otimes_{\kk (\tfb)} M, N).$$ Finally, this is isomorphic to $\hom _{\kk \uwb} (\kk^\uwb \otimes_{\kk (\tfb)} M, \kk^\uwb\otimes_{\kk (\tfb)} N)$.

This chain of isomorphisms can be written down explicitly. For this, one can reduce to the case where $M$ and $N$ each have support of cardinality one and are finite-dimensional. Then, the result boils down to the fact that, for $m,n,s,t \in \nat$, there are canonical isomorphisms:
$$
\hom_{\kk \uwb} (\kk \uwb ((\mathbf{m}, \mathbf{n}), -) 
,
\kk \uwb ((\mathbf{s}, \mathbf{t}), -)
)
\cong 
\kk \uwb ((\mathbf{s}, \mathbf{t}),(\mathbf{m}, \mathbf{n}))
\cong 
\hom_{\kk \uwb} (\kk ^{\uwb (-,(\mathbf{m}, \mathbf{n}))} 
,
\kk ^{\uwb (-, (\mathbf{s}, \mathbf{t}))}
)
 $$
 given by Yoneda's lemma. This was used (implicitly) in the proof of Proposition \ref{prop:equivalence_N_adj_unit}.
\end{rem}

 \section{The quadratic and Koszul properties}
 \label{sect:uwb_quadratic_koszul}

The purpose of this section is firstly to explain that $\kk \uwb$ is a homogeneous quadratic category over $\kk (\tfb)$ (with a similar statement for the twisted version $\uwbtw$); this is the subject of Section \ref{subsect:quadraticity}. The   quadratic dual is identified in Section \ref{subsect:quadratic_dual_dualizing_complex}, where the associated Koszul dualizing complex is also introduced.  This allows various Koszul complexes to be constructed in Section \ref{subsect:koszul_complexes} (with a second variant in Section \ref{subsect:second_Koszul_complexes}). 

The significance of these Koszul complexes is explained by Section \ref{subsect:koszul_property}, which reviews the fact that $\kk \uwb$ is a Koszul $\kk$-linear category over $\kk (\tfb)$ when $\kk$ is a field of characteristic zero.
 
This material makes no claim to originality; most of the results could have been derived from the material in \cite[Section 3]{MR3376738}, for example. The presentation adopted here has been preferred so as to make the Koszul complexes explicit in a form suitable for the applications. 
 
\subsection{Quadraticity}
\label{subsect:quadraticity}

For a general reference on homogeneous quadratic rings, we refer to \cite[Chapter 1]{MR4398644}; that material extends easily to working with $\kk$-linear categories.

Recall that we have an $\nat$-grading of the morphisms of $\kk\uwb$ (respectively $\uwbtw$) by the degree; moreover, 
the wide subcategory of  degree zero morphisms in both cases identifies with $\kk(\tfb)$, so that there are inclusions of $\kk$-linear categories
\begin{eqnarray*}
\kk (\tfb) 
&\hookrightarrow & 
\kk \uwb
\\
\kk (\tfb) 
&\hookrightarrow & 
\uwbtw.
\end{eqnarray*}

The following  is well-known:

\begin{prop}
\label{prop:homogeneous_quadratic}
The $\kk$-linear categories $\kk \uwb$ and $\uwbtw$ are both homogeneous quadratic over $\kk (\tfb)$.
\end{prop}

In preparation for a sketch proof, we introduce the following notation:

\begin{nota}
\label{nota:i_mn}
For $m,n \in \nat$, let $i_{m,n} \in \uwbord ((\m, \n) , (\mathbf{m+1}, \mathbf{n+1}))$ denote the morphism represented by the canonical inclusions $\m \subset \mathbf{m+1}$ and $\n \subset \mathbf{n+1}$; this can also be considered as a morphism of  $ \uwb ((\m, \n) , (\mathbf{m+1}, \mathbf{n+1}))$. 
 Write $[i_{m,n}]$ for the corresponding generator of $\uwbtw ((\m, \n) , (\mathbf{m+1}, \mathbf{n+1}))$.
\end{nota}

\begin{proof}[Proof of Proposition \ref{prop:homogeneous_quadratic}.]
(Indications.)
We start by considering the $\kk$-linear category $\kk \uwbord$. From its definition, we have an $\nat$-grading of the morphisms and the wide subcategory of degree zero morphisms identifies with $\kk (\tfb)$. Moreover, it is clear that the morphisms of $\kk \uwbord$ are generated over $\kk (\tfb)$ under composition by the 
 $\kk (\tfb)$-bimodule of degree one morphisms. The latter is supported on arities of the form $((\m, \n), (\mathbf{m+1}, \mathbf{n+1}))$, for $m,n \in \nat$, with underlying $\kk$-module 
$$
\kk \uwbord ((\m, \n), (\mathbf{m+1}, \mathbf{n+1})).
$$
As a $\kk (\tfb)$-bimodule, this is generated by $[i_{m,n}]$, where $i_{m,n}$ is as in Notation \ref{nota:i_mn}. More precisely, it is generated as a $\kk (\sym_{m+1} \times \sym_{n+1})$-module by this element (see Lemma \ref{lem:degree_one} below).

One can check that there are no further relations (other than those encoded in the structure of the $\kk (\tfb)$-bimodule of degree one morphisms). In particular, $\kk \uwbord$ is a homogeneous quadratic category, freely generated over $\kk (\tfb)$ by the  $\kk (\tfb)$-bimodule of degree one morphisms.

The canonical $\kk$-linear functors 
$$
\kk \uwb \leftarrow \kk \uwbord \rightarrow \uwbtw
$$
respect the $\nat$-grading and are isomorphisms on degree zero and degree one. Since the symmetric groups are generated by their transpositions, it is straightforward to see that these functors are given by applying homogeneous quadratic relations. (This is made more precise in Proposition \ref{prop:kk_uwb_quadratic} below.) The result follows.
\end{proof}

We identify the degree one morphisms in the category $\kk\uwbord$ (and hence $\kk \uwb$ and $\uwbtw$):

\begin{lem}
\label{lem:degree_one}
For $m,n \in \nat$, there is an isomorphism of $\kk(\sym_{m+1} \times \sym_{n+1}) \otimes \kk (\sym_m \times \sym_n)\op$-modules:
\[
\kk \uwbord ((\m, \n) , (\mathbf{m+1}, \mathbf{n+1})) 
\cong 
\kk (\sym_{m+1} \times  \sym_{n+1})
\]
where the right hand side is equipped with the left regular module structure and the restricted right regular module structure, restricting to $\sym_m \times \sym_n \subset \sym_{m+1} \times  \sym_{n+1}$.
\end{lem}

\begin{proof}
A morphism of $\uwbord ((\m, \n) , (\mathbf{m+1}, \mathbf{n+1})) $ is a pair of injections $(\m \hookrightarrow \mathbf{m+1}, \n \hookrightarrow \mathbf{n+1})$; an injection $\m \hookrightarrow \mathbf{m+1}$ extends canonically to an automorphism of $\mathbf{m+1}$ and likewise for $n$. This induces the isomorphism as $\kk$-vector spaces; the module structures are checked directly.
\end{proof}

The following is a more explicit version  of Proposition  \ref{prop:homogeneous_quadratic}:

\begin{prop}
\label{prop:kk_uwb_quadratic}
With respect to the $\nat$-grading of morphisms, $\kk \uwb$ and $\uwbtw$ are both homogeneous  quadratic categories over $\kk (\tfb)$.
\begin{enumerate}
\item 
 The morphisms of the category $\kk \uwb$ (respectively $\uwbtw$) are generated over $\kk (\tfb)$ by the set of morphisms
$
\{ [i_{m,n}] \mid m, n \in \nat\}
$.
\item 
The relations  of the category $\kk \uwb$ (respectively $\uwbtw$)  are generated by the $\kk (\tfb)$-bimodule structure of the degree one morphisms, together with the relations for $m,n \in \nat$ given by 
$$
\begin{array}{ll}
\  [i_{m+1, n+1}] \circ [i_{m,n}] = [\rho] \circ [i_{m+1, n+1}] \circ [i_{m,n}] & \mbox{\ for $\kk \uwb$}
 \\
\   [i_{m+1, n+1}] \circ [i_{m,n}] =  -[\rho] \circ [i_{m+1, n+1}] \circ [i_{m,n}] & \mbox{\  	for $\uwbtw$,}
\end{array}
$$
 where $\rho \in \aut_{\tfb} (\mathbf{m+2}, \mathbf{n+2})$ is given by the pair of transpositions $(m+1, m+2)\in \sym_{m+2}$ and $(n+1, n+2) \in \sym_{n+2}$.  
\end{enumerate}
\end{prop}

\begin{proof}
The first statement has already been explained in the proof of Proposition \ref{prop:homogeneous_quadratic}. The relations simply impose the (anti)commutativity between `pairs' in the passage from $\kk \uwbord$ to $\kk \uwb$ and $\uwbtw$ respectively.
\end{proof}

The quadratic presentations of the categories $\kk \uwb$ (respectively $\uwbtw$) given in Proposition \ref{prop:kk_uwb_quadratic} yields the following interpretation of the category of $\kk \uwb$-modules (resp. $\uwbtw$-modules). Recall  the shift functors $\shift{*}{*}$ introduced in Section \ref{subsect:shift} and  that  $\shift{2}{2}$ is naturally isomorphic to $\shift{1}{1}\circ \shift{1}{1}$.

\begin{nota}
\label{nota:rho}
Write  $\delta_{\tau, \tau} $ for the natural automorphism of $\shift{2}{2}$ induced by $\tau = (12) \in \sym_2$.
\end{nota}

\begin{cor}
\label{cor:kuwb_uwbtw-modules}
The category of $\kk \uwb$-modules (respectively $\uwbtw$-modules)  is equivalent to the category with objects pairs $(F, i_*^F : F \rightarrow \dbd F)$ of an $\kk (\tfb)$-module $F$ equipped with a structure morphism $i_*^F$ in $\f (\tfb)$ such that the following relations hold respectively between natural transformations $F\rightarrow \shift{2}{2} F\cong \shift{1}{1}\shift{1}{1} F$:
$$
\begin{array}{ll}
(\dbd i_*^F )\circ i_*^F 
=  
\delta_{\tau, \tau} \circ (\dbd i_*^F )\circ i_*^F
& \mbox{for $\kk \uwb$-modules} 
\\
(\dbd i_*^F )\circ i_*^F 
=  
-
\delta_{\tau, \tau} \circ (\dbd i_*^F )\circ i_*^F
& \mbox{for $\uwbtw$-modules.} 
\end{array}
$$

In both cases, a morphism from $(F, i_*^F : F \rightarrow \dbd F)$ to $(G, i_*^G : G \rightarrow \dbd G)$ is a morphism $f: F \rightarrow G$ of $\kk (\tfb)$-modules such that the following diagram commutes:
\[
\xymatrix{
F 
\ar[r]^{i_*^F}
\ar[d]_f 
&
\dbd F
\ar[d]^{\dbd f}
\\
G 
\ar[r]_{i_*^G}
&
\dbd G.
}
\]
\end{cor}

\begin{proof}
We give the proof for the case of $\kk \uwb$-modules; the proof for $\uwbtw$-modules is similar.

Given a $\kk \uwb$-module $F$, one has the underlying $\kk (\tfb)$-module, also denoted $F$ here. This contains the information on the action of the degree zero morphisms of $\kk \uwb$. The degree one morphisms give the equivariant maps (equivariance corresponding to the $\kk (\tfb)$-structure)
\begin{eqnarray}
\label{eqn:equiv_uwb}
\kk \uwb ((\m, \n) , (\mathbf{m+1}, \mathbf{n+1}) ) 
\rightarrow 
\hom_\kk (F(\m, \n) , F(\mathbf{m+1}, \mathbf{n+1})),
\end{eqnarray}
for each $(m,n) \in \nat$.

Now, the conclusion of Lemma \ref{lem:degree_one} can be rewritten as the isomorphism of bimodules 
\[
\kk \uwb ((\m, \n) , (\mathbf{m+1}, \mathbf{n+1}) ) 
\cong 
\kk (\sym_{m+1} \times \sym_{n+1}) \otimes_{\kk(\sym_m \times \sym_n)} \kk (\sym_m \times \sym_n),
\]
where the left and right actions correspond to the respective regular actions and the right $\sym_m \times \sym_n$-action on $\kk (\sym_{m+1} \times \sym_{n+1})$ is the restricted one. 
From this, it follows that the equivariant map (\ref{eqn:equiv_uwb}) is equivalent to a morphism in 
\[
\hom_{\sym_m \times \sym_n} (F(\m, \n) , F(\mathbf{m+1}, \mathbf{n+1})\downarrow^{\sym_{m+1}\times \sym_{n+1}} _{\sym_m \times \sym_n}).
\]
Globally (i.e., for all $m,n \in \nat$), this corresponds to  the structure morphism $i_*^F : F \rightarrow \dbd F$. The quadratic relation from Proposition \ref{prop:kk_uwb_quadratic} implies that this structure morphism satisfies the stated relation. These constructions are clearly natural with respect to $F$, hence define a functor from $\f (\uwb)$ to the category defined in the statement. This functor  is clearly faithful.

The fact that $\kk \uwb$ is homogeneous quadratic means that one can construct a quasi-inverse: given the datum $(F, i_*^F)$ that satisfies the `quadratic relation', there is a unique  $\kk \uwb$-module with underlying $\kk (\tfb)$-module $F$ that induces this datum as above. 
\end{proof}

\begin{rem}
\label{rem:subtlety_dwb}
Similarly, the opposite categories $\kk \dwb$ and $\dwbtw$ are homogeneous quadratic over $\kk (\tfb)$. 
However, there is a subtlety to keep in mind when considering $\kk (\tfb)$ as a subcategory of $\kk \dwb$ (respectively $\dwbtw$). For example, the inclusion $\kk (\tfb) \hookrightarrow \kk \uwb$ on passage to the opposite categories gives $\kk (\tfb)\op \hookrightarrow \kk \dwb$. To replace $\kk (\tfb)\op$ by $\kk (\tfb)$ one uses the isomorphism of categories $\fb \cong \fb\op$ that is the identity on objects and sends an isomorphism $\alpha$ to its inverse $\alpha^{-1}$. 
\end{rem}

There is the following counterpart of Corollary \ref{cor:kuwb_uwbtw-modules}, in which $\tau$ denotes $(12) \in \sym_2$, as in Notation \ref{nota:rho}.

\begin{cor}
\label{cor:kdwb_dwbtw-modules}
The category of $\kk \dwb$-modules (respectively $\dwbtw$-modules)  is equivalent to the category with objects pairs $(F, i_F^* : F \rightarrow F \circledcirc (\triv_1 \boxtimes \triv_1))$ of an $\kk (\tfb)$-module $F$ equipped with a structure morphism $i_F^*$ in $\f (\tfb)$ such that the following relations hold respectively between natural transformations $F\rightarrow (F  \circledcirc (\triv_1 \boxtimes \triv_1)) \circledcirc (\triv_1 \boxtimes \triv_1) \cong F \circledcirc (\kk \sym_2 \boxtimes \kk \sym_2)$:
$$
\begin{array}{ll}
(i^*_F \circledcirc (\triv_1 \boxtimes \triv_1) )\circ i^*_F 
=  
(\id \circledcirc (\tau \boxtimes \tau)) (i^*_F \circledcirc (\triv_1 \boxtimes \triv_1) )\circ i^*_F 
& \mbox{ for $\kk \dwb$-modules} 
\\
(i^*_F \circledcirc (\triv_1 \boxtimes \triv_1) )\circ i^*_F 
=  - 
(\id \circledcirc (\tau \boxtimes \tau)) (i^*_F \circledcirc (\triv_1 \boxtimes \triv_1) )\circ i^*_F 
& \mbox{ for $\uwbtw$-modules.} 
\end{array}
$$

In both cases, a morphism from $(F, i^*_F : F \rightarrow  F\circledcirc (\triv_1 \boxtimes \triv_1) )$ to $(G, i_G^* : G \rightarrow  G\circledcirc (\triv_1 \boxtimes \triv_1))$ is a morphism $f: F \rightarrow G$ of $\kk (\tfb)$-modules such that the following diagram commutes:
\[
\xymatrix{
F 
\ar[r]^(.3){i_F^*}
\ar[d]_f 
&
 F\circledcirc (\triv_1 \boxtimes \triv_1)
\ar[d]^{f\circledcirc \id}
\\
G 
\ar[r]_(.3){i_G^*}
&
 G\circledcirc (\triv_1 \boxtimes \triv_1).
}
\]
\end{cor}

\begin{proof}
There is  an immediate counterpart of  Corollary \ref{cor:kuwb_uwbtw-modules} for the downward case  that is expressed in terms of the associated pairs $(F, \dbd F \rightarrow F)$. In the  statement of Corollary \ref{cor:kdwb_dwbtw-modules} above, the adjunction of Proposition \ref{prop:adjoints_to_shift} has been used to reformulate this. 
\end{proof}

\subsection{Quadratic duals and the Koszul dualizing complex}
\label{subsect:quadratic_dual_dualizing_complex}

We refer to \cite[Chapter 1]{MR4398644} for the notion of the right (respectively left) quadratic dual of a homogeneous quadratic ring; that material extends readily to the case of $\kk$-linear categories.

Given a homogeneous quadratic $\kk$-linear category, if this satisfies the requisite right projectivity hypothesis, one can form the right quadratic dual. This right projectivity hypothesis holds for both $\kk \uwb$ and $\uwbtw$.

\begin{prop}
\label{prop:quadratic_duals}
\ 
\begin{enumerate}
\item 
The right quadratic dual of $\kk \uwb$ is $\dwbtw$. 
\item 
The right quadratic dual of $\uwbtw$ is $\kk \dwb$. 
\end{enumerate}
\end{prop}

\begin{proof}
The proof of this result is straightforward. (Compare the proof of the corresponding results for twisted $\kk$-linear upward Brauer categories in \cite{P_cyclic}.)
\end{proof}

\begin{rem}
There are `mirror' statements for the categories $\kk \dwb$ and $\dwbtw$:
\begin{enumerate}
\item 
The left quadratic dual of $\kk \dwb$ is $\uwbtw$. 
\item 
The left quadratic dual of $\dwbtw$ is $\kk \uwb$. 
\end{enumerate}
\end{rem}

In this context, one can form a Koszul dualizing complex corresponding to the fact that $\dwbtw$ is the right quadratic dual of $\kk \uwb$. (This should be compared with the cases considered in \cite{P_cyclic}, where more detail is given.) We first introduce the underlying bimodule:

\begin{defn}
\label{defn:kzcx_uwb}
Let $\kzcx$ be the $\kk \uwb \otimes \uwbtw$-module given by 
\[
\kzcx := \kk \uwb \otimes_{\kk(\tfb)}  \dwbtw.
\]
Explicitly, $\kzcx = \bigoplus_{m,n\in \nat} \kk \uwb ((\m,\n), -) \otimes_{\kk(\sym_m \times \sym_n)} \dwbtw (- ,(\m,\n))$, using the obvious left $\kk \uwb$-module structure and the right $\dwbtw$-module structure (equivalently left $\uwbtw$-module structure) to define the $\kk \uwb \otimes \uwbtw$-module structure.
\end{defn}

The $\kk \uwb \otimes \uwbtw$-module $\kzcx$ is equipped with a differential that is a morphism of $\kk \uwb \otimes \uwbtw$-modules. This is the direct sum (for $m, n \in \nat$) of the morphisms of modules:
\begin{eqnarray*}
 \resizebox{\hsize}{!}{$
 \kk \uwb ((\m,\n), -) \otimes_{\kk(\sym_m \times \sym_n)} \dwbtw (- ,(\m,\n))
\rightarrow 
\kk \uwb ((\mathbf{m-1},\mathbf{n-1}), -) \otimes_{\kk(\sym_{m-1} \times \sym_{n-1})} \dwbtw (- ,(\mathbf{m-1},\mathbf{n-1})).
$}
\end{eqnarray*}
This is uniquely determined by the morphism of $\kk (\sym_m \times \sym_n)$-bimodules corresponding to the evaluation of the above on $((\m,\n), (\m,\n))$:
$$
\kk (\sym_m \times \sym_n) 
\rightarrow 
 \kk \uwb ((\mathbf{m-1},\mathbf{n-1}), (\m, \n)) \otimes_{\kk(\sym_{m-1} \times \sym_{n-1})} \dwbtw ((\m, \n) ,(\mathbf{m-1},\mathbf{n-1})).
 $$
It is given by 
$$
[e] 
\mapsto 
\sum_{(x,y) \in \wpair_1(\m, \n)} [\iota_{x,y}] \otimes [\iota_{x,y}\op], 
$$
where $[\iota_{x,y}]$ and $[\iota_{x,y}\op]$ are as in Notation \ref{nota:iota_xy}.

\begin{rem}
One can check directly that the above is  a morphism of $\kk (\sym_m \times \sym_n)$-bimodules. However, a better approach is to use the underlying quadratic duality framework, as exemplified by Proposition \ref{prop:quadratic_duals}: the morphism arises as a coevaluation map and the fact that it is a morphism of bimodules is automatic.
\end{rem}

\begin{rem}
\ 
\begin{enumerate}
\item 
For any $p,q,s,t \in \nat$,  there are only finitely many pairs $(m,n)$ for which the term 
$$ \kk \uwb ((\m, \n),(\mathbf{p}, \mathbf{q})) \otimes_{\kk(\sym_m \times \sym_n)} \dwbtw ((\mathbf{s}, \mathbf{t}),(\m, \n) )$$ is non-zero.
More precisely, this vanishes unless $p-m= q-n \in \nat$ and $s-m =t -n \in \nat$. In particular, it vanishes for $m > \min (p,s)$ or $n > \min (q,t)$. It also vanishes 
if $t- s \neq q-p$. It follows that the (total) bimodule $\kzcx$ takes finite-dimensional values.
\item 
One can  use the $\nat$-grading of either $\kk \uwb$ or $\dwbtw$ to define a grading of the complex.  In the applications, it is usually clear which is appropriate. 
\item 
One can use the $\nat$-grading of morphisms of $\uwb$ (and hence of $\kk \uwb$ and $\dwbtw$) to decompose $\kzcx$ as a direct sum of complexes, with summands indexed by the difference $m-n \in \zed$, for $m$, $n$  as above.
\end{enumerate}
\end{rem}

The complex of $\kk \uwb \otimes \uwbtw$-modules $\kzcx$ will be referred to as the Koszul dualizing complex.  Evaluated on any object of $\kk \uwb \otimes \uwbtw$, it is non-zero in only finitely many (co)homological degrees.

\subsection{Associated Koszul complexes}
\label{subsect:koszul_complexes}

One has the following Koszul complexes that are defined using $\kzcx$: 

\begin{enumerate}
\item 
for $M$ a $\dwbtw$-module, the complex of $\kk \uwb$-modules $\kzcx \otimes_{\dwbtw} M$; 
\item 
for $N$ a $\kk \uwb$-module, the complex of $\dwbtw$-modules $\hom_{\kk \uwb} (\kzcx, N)$; 
\item 
for $M'$ a $\kk\dwb$-module, the complex of $\uwbtw$-modules $\kzcx \otimes_{\kk \dwb} M'$; 
\item 
for $N'$ a $\uwbtw$-module, the complex of $\kk \dwb$-modules $\hom_{\uwbtw} (\kzcx, N')$.
\end{enumerate}

The above functors all extend to complexes of modules;  these are then related by the adjunctions:
\begin{eqnarray}
\kzcx \otimes_{\dwbtw}  -  &\dashv &\hom_{\kk \uwb} (\kzcx, -) \\
\kzcx \otimes_{\kk\dwb}  -  &\dashv & \hom_{\uwbtw} (\kzcx, -) ,
\end{eqnarray}
the first relating complexes of $\dwbtw$-modules and complexes of $\kk \uwb$-modules; the second relating complexes of $\kk \dwb$-modules and complexes of $\uwbtw$-modules. 

Hence, 
\begin{enumerate}
\item 
if $M$ is a $\dwbtw$-module, one has the adjunction unit:
\[
M \rightarrow \hom_{\kk \uwb} (\kzcx , \kzcx \otimes_{\dwbtw} M); 
\]
\item 
if $N$ is a $\kk \uwb$-module, one has the adjunction counit:
$$
\kzcx \otimes_{\dwbtw} \hom_{\kk \uwb} (\kzcx, N)
\rightarrow N.
$$
\end{enumerate}

The underlying functors identify as follows:

\begin{lem}
\label{lem:underlying_objects_Koszul_complexes}
\ 
\begin{enumerate}
\item 
For $M$ a $\dwbtw$-module, there is a natural isomorphism of graded $\kk \uwb$-modules $$\kzcx \otimes_{\dwbtw} M \cong \kk \uwb \otimes_{\kk (\tfb)} M.$$ 
\item 
For $N$ a $\kk \uwb$-module,  there is a natural isomorphism of graded $\dwbtw$-modules $$\hom_{\kk \uwb} (\kzcx, N) \cong \hom_{\kk (\tfb)} (\dwbtw, N).$$ 
\item 
For $M'$ a $\kk\dwb$-module,  there is a natural isomorphism of graded $\uwbtw$-modules $$\kzcx \otimes_{\kk \dwb} M'\cong \uwbtw \otimes_{\kk (\tfb)} M'.$$ 
\item 
For $N'$ a $\uwbtw$-module,  there is a natural isomorphism of graded $\kk \dwb$-modules $$\hom_{\uwbtw} (\kzcx, N')\cong \hom_{\kk (\tfb)} (\kk \dwb, N').$$
\end{enumerate}
In each case,  the differential of $\kzcx$ induces a Koszul-type differential on the right hand side.
\end{lem}

\begin{rem}
\label{rem:complex_in_uwbup}
The first statement of  Lemma \ref{lem:underlying_objects_Koszul_complexes} shows that  $\kzcx \otimes_{\dwbtw} M$ is a complex in the category $\uwbup$. A similar statement holds for $\kzcx \otimes_{\kk \dwb} M'$, replacing $\uwbup$ by its counterpart for $\uwbtw$-modules.
\end{rem}

\begin{exam}
\
\begin{enumerate}
\item 
The $\kk$-linear category $\dwbtw$ has augmentation $\dwbtw \longrightarrow \kk (\tfb)$ that is the identity on objects and, on morphisms, is  the projection onto degree zero. Thus, by restriction along the augmentation for the left module structure, we can consider $\kk (\tfb)$ as a $\dwbtw \otimes \kk (\tfb)\op$-module.

Using this structure,  $\kzcx \otimes_{\dwbtw} \kk (\tfb)$ is isomorphic to $\kk \uwb$, considered as a $\kk \uwb \otimes \kk (\tfb)\op$-module (using the restricted module structure on the right). The differential is zero. Going one step further, one has the adjunction unit 
\begin{eqnarray}
\label{eqn:unit}
\kk (\tfb)\rightarrow \hom_{\kk \uwb} (\kzcx ,\kk \uwb) \cong \hom_{\kk (\tfb)} (\dwbtw, \kk \uwb).
\end{eqnarray}
This is a morphism of $\dwbtw \otimes \kk (\tfb)\op$-modules. The codomain is equipped  with a `Koszul complex' differential.
\item 
Similarly, the $\kk$-linear category $\kk \uwb$ has augmentation $\kk \uwb \rightarrow \kk (\tfb)$ so that $\kk (\tfb)$ can be considered as a $\kk \uwb \otimes \kk (\tfb)\op$-module. Then the complex $ \hom_{\kk \uwb} (\kzcx, \kk (\tfb))$ identifies with  $ \hom_{\kk (\tfb)} (\dwbtw, \kk (\tfb))$, the left dual of $\dwbtw$; this has the structure of  a $\dwbtw \otimes \kk (\tfb)\op$-module. 

One then has the adjunction counit, which identifies as 
\begin{eqnarray}
\label{eqn:counit}
\kk \uwb \otimes_{\kk (\tfb)} \hom_{\kk (\tfb)} (\dwbtw, \kk (\tfb))
\rightarrow 
\kk 
\uwb,
\end{eqnarray}
where the domain is isomorphic to $
\kzcx \otimes_{\kk \uwb} \hom_{\kk (\tfb)} (\dwbtw, \kk (\tfb))$ and is equipped with a `Koszul complex'  differential. This is a morphism of $\kk \uwb \otimes \kk (\tfb)\op$-modules.
\end{enumerate}
\end{exam}

\subsection{The Koszul property}
\label{subsect:koszul_property}
We refer to  \cite[Chapter 2]{MR4398644} for the notion of a Koszul homogeneous quadratic ring; this generalizes to the $\kk$-linear category setting.  

We have the fundamental result: 

\begin{thm}
\label{thm:dwb_koszul}
For $\kk$ a field of characteristic zero, the categories $\kk \uwb$ and $\uwbtw$ are both  Koszul $\kk$-linear categories over $\kk (\tfb)$.
\end{thm}

For current purposes, the Koszul property of $\kk \uwb$ can be taken to be the property that both (\ref{eqn:unit}) and (\ref{eqn:counit}) are weak equivalences.

\begin{rem}
\ 
\begin{enumerate}
\item 
Theorem \ref{thm:dwb_koszul} is a consequence of \cite[Theorem 5.5]{MR3439686}, with a proof that  relies crucially upon \cite{MR2743762}.
\item 
The theorem can be proved directly by analysing the Koszul complexes appearing in (\ref{eqn:unit}) and (\ref{eqn:counit}). (Cf. \cite{P_cyclic}, where the case of twisted $\kk$-linear upward Brauer categories is treated.)
\end{enumerate}
\end{rem}

The Koszul property has important (co)homological consequences (see \cite[Chapter 2]{MR4398644}, for example).
For example:

\begin{cor}
\label{cor:cohomology_koszul_complexes}
\ 
\begin{enumerate}
\item 
For $M$ a $\dwbtw$-module, the cohomology of $\kzcx \otimes_{\dwbtw} M$ is naturally isomorphic to $\ext^* _{\dwbtw} (\kk (\tfb), M)$. 

In particular,  $\ext^* _{\dwbtw} (\kk (\tfb), \kk (\tfb))$ equipped with the Yoneda product is naturally isomorphic to $\kk \uwb$; moreover the $\ext^* _{\dwbtw} (\kk (\tfb), \kk (\tfb))$-module structure on $\ext^* _{\dwbtw} (\kk (\tfb), M)$ corresponds to the $\kk \uwb $-module structure on the cohomology of $\kzcx \otimes_{\dwbtw} M$.
\item 
For $M'$ a $\kk \dwb$-module, the cohomology of $\kzcx \otimes_{\kk\dwb} M'$ is naturally isomorphic to $\ext^* _{\kk\dwb} (\kk (\tfb), M')$. 

In particular,  $\ext^* _{\kk\dwb} (\kk (\tfb), \kk (\tfb))$ equipped with the Yoneda product is naturally isomorphic to $\uwbtw$; moreover the $\ext^* _{\kk\dwb} (\kk (\tfb), \kk (\tfb))$-module structure on $\ext^* _{\kk\dwb} (\kk (\tfb), M')$ corresponds to the $\uwbtw $-module structure on the cohomology of $\kzcx \otimes_{\kk\dwb} M'$.
\end{enumerate}
\end{cor}

\subsection{The second Koszul complexes}
\label{subsect:second_Koszul_complexes}

Remark \ref{rem:complex_in_uwbup} tells us that $\kzcx \otimes_{\dwbtw} M$ is a complex in $\uwbup$. The equivalence of Corollary \ref{cor:equivalence_uwbup_uwbdown} between $\uwbup$ and $\uwbdown$ suggests that we should  consider the complex 
\[
\kk ^\uwb \otimes_{\kk \uwb} \kzcx \otimes_{\dwbtw} M, 
\]
which lives in $\uwbdown$. By Lemma \ref{lem:underlying_objects_Koszul_complexes}, this has the form 
$$
\kk ^\uwb \otimes_{\kk (\tfb)} M,
$$
equipped with the appropriate Koszul-type differential derived from that of $\kzcx$.

A similar construction applies to $\kzcx \otimes_{\kk \dwb} M'$. In this case one considers:
$$
(\uwbtw)^\sharp \otimes_{\uwbtw} \kzcx \otimes _{\kk \dwb}M' 
\cong 
(\uwbtw)^\sharp \otimes_{\kk (\tfb)} M', 
$$
again equipped with the appropriate Koszul-type differential.

The Koszul property implies that these complexes also have nice homological interpretations. 

\begin{prop}
\label{prop:homology_is_Tor}
\ 
\begin{enumerate}
\item 
For $M$ a $\dwbtw$-module, the homology of the complex $\kk ^\uwb \otimes_{\kk \uwb} \kzcx \otimes_{\dwbtw} M$ is naturally isomorphic to $\tor_*^{\dwbtw} (\kk (\tfb), M)$. 
\item 
For $M'$ a $\kk \dwb$-module, the homology of the complex $(\uwbtw)^\sharp \otimes_{\uwbtw} \kzcx \otimes _{\kk \dwb}M' $ is naturally isomorphic to $\tor_*^{\kk \dwb} (\kk (\tfb), M')$.
\end{enumerate}
\end{prop}

\begin{rem}
One can go further. The $\kk \uwb$-module structure on the homology of $\kk ^\uwb \otimes_{\kk \uwb} \kzcx \otimes_{\dwbtw} M$  can be identified with the cap product structure of $\ext^* _{\dwbtw} (\kk (\tfb), \kk (\tfb))$-module on $\tor_*^{\dwbtw} (\kk (\tfb), M)$. A similar statement holds for $(\uwbtw)^\sharp \otimes_{\uwbtw} \kzcx \otimes _{\kk \dwb}M' $.
\end{rem}

\subsection{Universal coefficients spectral sequences}
\label{subsect:UCT_Koszul}

Recall that, for $M$ a $\dwbtw$-module,  $\kzcx \otimes_{\dwbtw} M$ is a complex in $\uwbup$. In particular, the terms of the complex are projective  $\kk \uwb$-modules.

Now, the functor $\kk^\uwb \otimes_{\kk \uwb} -$ is right exact but is not exact. As a consequence, there is a  universal coefficients spectral sequence that relates the homology of $\kzcx \otimes_{\dwbtw} M$ with that of $\kk^\uwb \otimes_{\kk \uwb} \kzcx \otimes_{\dwbtw} M$. This has the form 
\[
\tor_*^{\kk \uwb} (\kk ^\uwb , H_* (\kzcx \otimes_{\dwbtw} M))
\Rightarrow 
H_* (\kk^\uwb \otimes_{\kk \uwb} \kzcx \otimes_{\dwbtw} M).
\]

In particular, one has the edge morphism 
$$
\kk^\uwb \otimes_{\kk \uwb} H_* (\kzcx \otimes_{\dwbtw} M)
\rightarrow 
H_*(\kk^\uwb \otimes_{\kk \uwb} \kzcx \otimes_{\dwbtw} M).
$$

The above gives a fundamental tool for analysing the relationship between $H_* (\kzcx \otimes_{\dwbtw} M)$ and $H_*(\kk^\uwb \otimes_{\kk \uwb} \kzcx \otimes_{\dwbtw} M)$. Moreover, it underlines the importance of the $\kk \uwb$-module structure on $H_* (\kzcx \otimes_{\dwbtw} M)$.

\begin{rem}
One can also consider a universal coefficients spectral sequence to go in the other direction. 
 For example, one obtains an `edge morphism' of the form:
$$
H_* (\kzcx \otimes_{\dwbtw} M)
\rightarrow 
\hom_{\kk \uwb}( \kk ^\uwb,
H_*(\kk^\uwb \otimes_{\kk \uwb} \kzcx \otimes_{\dwbtw} M))
$$
that is adjoint to the one given above.
\end{rem}

\section{Functors on $\spmon$ and mixed tensor functors}
\label{sect:mixed_tensors}

In order to relate the approach of this paper to the work of Dotsenko \cite{MR4945404}, in this section we first recall the category $\spmon$ of finite-dimensional vector spaces with split monomorphisms as the morphisms; we then consider the category of functors $\f (\spmon)$. Functors on $\spmon$ were exploited in  \cite{MR4881588} in a similar vein. 

We then recall the relationship between the functor category $\f (\spmon)$ and  the category of $\kk \dwb$-modules, using the mixed tensor functors. The main  results  are essentially due to Sam and Snowden \cite[Section 3]{MR3376738}, with input going back to the first fundamental theorem of invariant theory.

Throughout this section, $\kk$ is a field of characteristic zero.

\subsection{Functors on $\spmon$ and stabilization}

We introduce the category $\spmon$ and functors on this; we then explain a naïve version of stabilization. 

\begin{defn}
\label{defn:spmon}
Let $\spmon$ be the category with objects finite-dimensional $\kk$-vector spaces and morphisms split monomorphisms: i.e., a morphism from $V$ to $W$ in $\spmon $ is a pair of $\kk$-linear morphisms $(i,p)$, where $V \stackrel{i}{\hookrightarrow} W \stackrel{p} {\twoheadrightarrow} V$ has composite $\id_V$; composition is defined in the obvious way.
\end{defn}

For $V$ an object of $\spmon$, the endomorphism monoid of $V$ identifies with the automorphism group, which in turn identifies as $\GL (V)$ (the usual  general linear group). In particular, $\spmon$ is an EI-category.  Moreover, it is clear that $\hom_{\spmon} (V, W)$ is non-zero if and only if $\dim V \leq \dim W$.

\begin{rem}
\label{rem:spmon_duality}
The association $V \mapsto V^\sharp$ (where $(-)^\sharp$ denotes vector space duality) yields a {\em covariant} endofunctor functor of $\spmon$. More precisely, $(-)^\sharp$ induces  an involution $\spmon\op \rightarrow \spmon$ that sends $(i,p)$ to $(p^\sharp , i ^\sharp)$.
\end{rem}

Recall that $\finj$ is the category of finite sets and injections. 

\begin{prop}
\label{prop:finj_spmon}
There is a faithful functor $\finj \rightarrow \spmon$ given on objects by $S \mapsto \kk S$. For $S \stackrel{\iota}{\hookrightarrow} T$ an injective map, the associated map $(i,p)$ is given by $i :=  \kk \iota$, with $p$  the linear retraction $\kk T \rightarrow \kk S$ such that, for $x \in T \backslash \iota (S)$,  $p[x]= 0$.  

The image of the skeleton $\{ \n \mid n \in \nat\}$ of $\finj$ gives a skeleton of $\spmon$.
\end{prop}

\begin{proof}
It is clear that $\kk (-)$, the free $\kk$-vector space functor, defines a faithful functor from $\finj$ to $\kk$-vector spaces. To prove the first statement, it suffices to show that this lifts as stated to a functor to $\spmon$. This is straightforward. The statement about the skeleta is clear. 
\end{proof}

This allows us (following  \cite{MR3376738}) to consider the group 
$$
\gl_\infty:= \bigcup_{n \in \nat} \GL (\kk \n).
$$

We consider the functor category $\f (\spmon)$, i.e., the category of functors from $\spmon$ to $\kk$-vector spaces. We also write $\rep (\gl_\infty)$ for the category of $\gl_\infty$-modules. Then we have a first version of a stabilization functor on $\f(\spmon)$:
\begin{eqnarray*}
\stabgl &: &
\f (\spmon) 
\rightarrow 
\rep (\gl_\infty)
\\
&& F \mapsto \lim_{\substack{\rightarrow \\ n \rightarrow \infty}} F (\kk \n).
\end{eqnarray*}

This stabilization does not see `torsion phenomena', where torsion is defined similarly to the case of $\kk \uwb$-modules, as follows:

\begin{defn}
\label{defn:torsion_spmon}
For $F$ an object of $\f (\spmon)$, 
\begin{enumerate}
\item 
a section $x \in F(V)$ is torsion if the subfunctor generated by $x$ has finite support; 
\item 
$F$ is torsion if every section is torsion; 
\item 
$F$ is torsion-free if it contains no non-zero torsion subfunctor.
\end{enumerate} 
The full subcategory of torsion functors is written $\f _\tors (\spmon)$.
\end{defn}

The analogue of Theorem \ref{thm:localizing} holds: the subcategory $\f _\tors (\spmon)$ is a localizing Serre subcategory of $\f (\spmon)$. In particular, one can form the quotient category $\f (\spmon)/ \f_\tors (\spmon)$. The stabilization functor factors across the quotient map as 
$$
\stabgl : \f (\spmon)/ \f_\tors (\spmon)
\rightarrow 
\rep (\gl_\infty).
$$
 
We shall refine this below by restricting to `algebraic' representations.

\subsection{The mixed tensor $\kk \dwb$-module}
\label{subsect:mixed_tensor}

 Mixed tensors, as introduced below, are important  for example in the work of Koike \cite{MR0991410} related to rational representations of the general linear groups. They also feature prominently in the work of Sam and Snowden, e.g., \cite{MR3376738}.

\begin{exam}
\label{exam:ordinary_mixed_tensors}
For $m,n \in \nat$, there is a functor $T^{m,n} \in \ob \f (\spmon)$ given by $T^{m,n} (V) := V^{\otimes m} \otimes (V^\sharp)^{\otimes n}$. 
Explicitly, for a morphism $(i,p) \in \hom_{\spmon } (V, W)$, the induced linear map is $i^{\otimes m} \otimes (p^\sharp)^{\otimes n} : V^{\otimes m} \otimes (V^\sharp)^{\otimes n}
\rightarrow W^{\otimes m} \otimes (W^\sharp)^{\otimes n}$. This is termed a {\em mixed tensor functor}.
\end{exam}

We have the following important property, which is proved by a direct verification.

\begin{prop}
\label{prop:mixed_torsion-free}
For $m,n \in \nat$, the functor $T^{m,n}$ is torsion-free. 
\end{prop}

The evaluation map $V \otimes V^{\sharp} \rightarrow \kk$ is natural with respect to $V \in \ob \spmon$, hence defines a natural transformation $T^{1,1} \rightarrow T^{0,0} = \kk$. More generally, these induce iterated contraction maps between mixed tensor functors, leading to Proposition \ref{prop:mixed_tensors} below, which is well-known (cf. Remark \ref{rem:dwb_mixed_literature} below).

\begin{prop}
\label{prop:mixed_tensors}
The association $(\m, \n ) \mapsto T^{m,n}$ defines a functor from $\dwb$ to $\f (\spmon)$; equivalently $T^{\bullet, \bullet}$ is an object of $\f (\spmon \times \dwb)$.

If $\kk$ is a field of characteristic zero, the associated $\kk$-linear functor $\kk \dwb \rightarrow \f (\spmon)$ is fully faithful.
\end{prop}

\begin{proof}
The proof that one obtains a functor $\kk \dwb \rightarrow \f (\spmon)$ is straightforward. That this functor is fully faithful is a consequence of the first fundamental theorem of invariant theory for the general linear groups. (For a treatment of the FFT in characteristic zero, see \cite[Chapter 9]{MR2265844}, for example.)
\end{proof}

\begin{rem}
\label{rem:dwb_mixed_literature}
Working over $\kk=\mathbb{C}$, Proposition \ref{prop:mixed_tensors},  
this is related to  \cite[Lemma 1.2]{MR0991410}; likewise, it is related to \cite[Theorem 5.8]{MR1280591}.  
It is also implicit in \cite[Theorem 3.2.11]{MR3376738}.
\end{rem}

Hence  we have
\begin{eqnarray*}
\hom_{\spmon}( -, T^{\bullet, \bullet}) & : & \f (\spmon) \op \rightarrow \f (\dwb) \\
\hom_{\dwb} (-, T^{\bullet, \bullet}) & :& \f (\dwb)\op \rightarrow \f (\spmon).
\end{eqnarray*}
These functors are adjoint, as follows for example from the general results of \cite[Section 2.1]{MR3376738}. Explicitly, for $M$ a $\kk \dwb$-module and $F$ an object of $\f (\spmon)$, there are natural isomorphisms:
$$
\hom_{\spmon} (F, \hom_{\dwb} (M, T^{\bullet, \bullet}) )
\cong 
\hom_{\spmon \times \dwb} (F \boxtimes M, T^{\bullet, \bullet}) 
\cong 
\hom_{\dwb} (M, \hom_{\spmon} (F, T^{\bullet, \bullet})).
$$

We note the following: 

\begin{prop}
\label{prop:homT_vanishes_on_torsion}
If $F$ is a torsion functor on $\spmon$, then $\hom_{\spmon} (F, T^{\bullet,\bullet}) $ is zero in $\kk \dwb$-modules. 
\end{prop}

\begin{proof}
This is an immediate consequence of the fact that $T^{m,n}$ is torsion-free for each $m,n \in \nat$, by Proposition \ref{prop:mixed_torsion-free}.
\end{proof}

The functoriality of the mixed tensor functors with respect to $\dwb$ leads to the {\em traceless tensors}. For the following,  recall the morphisms $\iota_{x,y}$ introduced in Notation \ref{nota:iota_xy} and their opposites $\iota_{x,y}\op$.

\begin{defn}
\label{defn:traceless}
For $a,b \in \nat$, let $\traceless{a,b}\subset T^{a,b}$ be the subfunctor (in $\f(\spmon)$) given by the traceless tensors, namely the kernel of the map
\[
T^{a,b} \rightarrow \bigoplus_{(x,y)\in \wpair_1 (\mathbf{a}, \mathbf{b})} T^{a-1, b-1},
\]
where the maps are given by the naturality of $T^{\bullet, \bullet}$ with respect to $\kk \dwb$, using the morphisms  $[\iota_{x,y}\op] \in \kk \dwb ((\mathbf{a}, \mathbf{b}), (\mathbf{a-1}, \mathbf{b-1}))$.
\end{defn}

\begin{rem}
From the definition, it is clear that the action of $\sym_a \times \sym_b$ on $T^{a,b}$ in $\f (\spmon)$ restricts to one on $\traceless{a,b}$.
\end{rem}

\begin{exam}
\label{exam:traceless_1,1}
Taking $(a,b) = (1,1)$ and noting that $T^{0,0}$ identifies as the constant functor $\kk$, one has the exact sequence 
$$
0
\rightarrow 
\traceless{1,1} \rightarrow 
T^{1,1} 
\rightarrow 
\kk.
$$
The right hand map is not quite surjective; its cokernel is $\kk_{0}$, supported on  $0 \in \ob \spmon$. 
If one passes to the quotient category $\f (\spmon)/ \f_\tors (\spmon)$, i.e., working up to torsion, then this yields a short exact sequence. 

This illustrates a general phenomenon when filtering the mixed tensor functors:  $T^{a,b}$ has a finite filtration with subquotients (up to torsion) of the form $T^{m,n}$ where $a-m = b-n \in \nat$. More explicitly:
\begin{eqnarray}
\label{eqn:decompose_Tab}
T^{a,b}
\sim
\bigoplus
_{m,n}
\kk \uwb ((\mathbf{m},\mathbf{n}) , (\mathbf{a},\mathbf{b}))\otimes_{\kk (\sym_m \times \sym_n) } \traceless{m,n}
,
\end{eqnarray}
where $\sim$ indicates the identification of the associated graded up to torsion. (This is a consequence of the results reviewed in Section \ref{subsect:alg_rep}, for example.)
\end{exam}

\subsection{Generalized Schur functors}
\label{subsect:gen_schur_fn}

We have the following application of the mixed tensor functors (more precisely, Proposition \ref{prop:mixed_tensors}), which  can be thought of as a form of generalized Schur functor (this terminology is inspired by \cite{MR3876732}).

\begin{cor}
\label{cor:T_uwb_to_spmon}
The mixed tensors yield a functor 
\[
T^{\bullet, \bullet} \otimes_{\kk\uwb} - : \kk\uwb\dash \modules  \rightarrow \f (\spmon), 
\]
given for $M$ a $\kk \uwb$-module by $T^{\bullet , \bullet} \otimes_{\kk\uwb} M : V \mapsto T^{\bullet , \bullet} (V) \otimes _{\kk \uwb} M$. 
 This functor is right exact. 
\end{cor}

\begin{exam}
\label{exam:gen_schur}
Let $m,n$ be natural numbers. Write $\puwb_{m,n}$ for the projective $\kk \uwb$-module $\kk \uwb ((\mathbf{m}, \n), -)$. 
\begin{enumerate}
\item 
There is a natural isomorphism $T^{\bullet, \bullet} \otimes_{\kk \uwb} \puwb_{m,n} \cong T^{m,n}$.
\item 
Consider $\kk (\sym_m \times \sym_n)$ as a $\kk \uwb$-module supported on $(\mathbf{m}, \n)$. Then, if $(m,n) \neq (0,0)$, one has the equality:
$$
T^{\bullet, \bullet} \otimes_{\kk \uwb} \kk (\sym_m \times \sym_n) =0.
$$
This can be seen as follows: the canonical inclusions $\mathbf{m} \subset \mathbf{m+1}$ and $\n \subset \mathbf{n+1}$ induce a morphism $\puwb_{m+1, n+1} \rightarrow \puwb_{m,n}$, by Yoneda's lemma. It is easily checked that the cokernel of this is isomorphic to $ \kk (\sym_m \times \sym_n)$ as a $\kk \uwb$-module. On applying the right exact functor $T^{\bullet, \bullet} \otimes_{\kk \uwb} -$, by the previous point, this yields the exact sequence 
$$
T^{m+1, n+1} 
\rightarrow 
T^{m,n}
\rightarrow 
T^{\bullet, \bullet} \otimes_{\kk \uwb} \kk (\sym_m \times \sym_n)
\rightarrow 
0.
$$
Here the first map is induced by the contraction map associated to the pair $(m+1, n+1)$. This is surjective unless $m=n=0$, in which case the cokernel of $\kk_{\mathbf{0}, \mathbf{0}}$.
\end{enumerate} 
\end{exam}

This example generalizes to give the following: 

\begin{prop}
\label{prop:gen_schur_torsion}
If $M$ is a torsion $\kk \uwb$-module such that $M (\mathbf{0}, \mathbf{0})=0$, then 
$
T^{\bullet, \bullet} \otimes_{\kk \uwb} M =0.
$
\end{prop}

\begin{proof}
One first proves the result when $M$ has finite support and takes finite-dimensional values. In the case that $M$ is supported on a single $(\mathbf{m}, \mathbf{n})$ this follows readily from the case given in the Example \ref{exam:gen_schur}, $T^{\bullet, \bullet} \otimes_{\kk \uwb} \kk (\sym_m \times \sym_n) =0$ when $(m,n) \neq(0,0)$. The finite support case is then established by the evident inductive argument on the size of the support, using the right exactness of 
 $T^{\bullet, \bullet} \otimes_{\kk \uwb}-$. 
 
For the general case, one uses the fact that $M$ is a quotient of a coproduct of copies of such modules, whence the result follows by right exactness again.
\end{proof}

Recall that we have the induction functor 
$$
\kk \uwb \otimes_{\kk (\tfb)} - : 
\f(\tfb)
\rightarrow 
\f(\uwb).
$$

The following is clear: 

\begin{lem}
\label{lem:compose_induction_gen_Schur}
The composite of the induction functor $\kk \uwb \otimes_{\kk (\tfb)} -  \colon \f (\tfb) \rightarrow \f (\uwb)$ with the generalized Schur functor $T^{\bullet, \bullet} \otimes_{\kk \uwb} - \colon \f (\uwb) \rightarrow \f (\spmon)$ is naturally isomorphic to the functor 
$$
T^{\bullet, \bullet} \otimes_{\kk (\tfb)} - \colon  \f (\tfb) \rightarrow \f (\spmon),
$$
using the restricted right action of $\kk (\tfb)$ on $T^{\bullet, \bullet}$. This functor is exact.
\end{lem}

\begin{rem}
The functor $T^{\bullet, \bullet} \otimes_{\kk (\tfb)} - $ is described explicitly as follows. It sends a $\kk (\tfb)$-module $N$ to 
$$
V \mapsto \bigoplus_{m,n} (V^{\otimes m} \otimes (V^\sharp)^{\otimes n}) \otimes_{\kk (\sym_m \times \sym_n)} N (\m, \n).
$$  
\end{rem}

Returning to the general case, we may consider applying the functor $\hom_{\spmon} (-,T^{\ast, \ast}) $ to a generalized Schur functor (using the notation $\ast$ for the wild-cards, 
 to distinguish it from the $\bullet$ below). We have the following:

\begin{prop}
\label{prop:TT_versus_k^uwb}
For $M$ a $\kk \uwb$-module, there is a natural isomorphism of $\kk \dwb$-modules
$$
\hom_{\spmon} (T^{\bullet, \bullet} \otimes_{\kk \uwb} M, T^{\ast, \ast}) 
\cong 
\hom_{\kk} (\kk^\uwb \otimes_{\kk \uwb} M , \kk).
$$
\end{prop}

\begin{proof}
By the universal property of $\otimes_{\kk \uwb}$, there is a natural isomorphism 
$$
\hom_{\spmon} (T^{\bullet, \bullet} \otimes_{\kk \uwb} M, T^{\ast, \ast}) 
\cong 
\hom_{\kk \uwb} (M, \hom_{\spmon} (T^{\bullet, \bullet}, T^{\ast, \ast})).
$$
By Proposition \ref{prop:mixed_tensors}, $\hom_{\spmon} (T^{\bullet, \bullet}, T^{\ast, \ast})$ is isomorphic to $\kk \uwb ((\ast,\ast), (\bullet, \bullet))$, which is considered as a (left) $\kk \uwb$-module with respect to $(\bullet, \bullet)$. The right hand side can thus safely be written as $\hom_{\kk \uwb} (M, \kk \uwb)$. 

Since $\uwb$ has finite hom sets, we may regard $\kk \uwb$ as the $\kk$-linear dual of $\kk^{\uwb}$. Using the universal property of $\otimes_{\uwb}$ again, we have the isomorphism
$$
\hom_{\kk \uwb} (M, \kk \uwb)
\cong 
\hom_\kk (\kk^\uwb \otimes_{\kk \uwb} M, \kk),
$$
giving the result. 
\end{proof}

\subsection{Algebraic representations and weak stabilization}
\label{subsect:alg_rep}

Following \cite[Section 3.2]{MR3876732}, we say that an object of $\f (\spmon)$ is {\em algebraic} if it is a subquotient of an (arbitrary) direct sum of objects of the form $T^{m,n}$ (for varying $m,n$). The full subcategory of algebraic functors is denoted by $\falg (\spmon)$. 

\begin{exam}
\label{exam:algebraic_functors}
\ 
\begin{enumerate}
\item 
For $m,n \in \nat$, the traceless mixed tensor functor $\traceless{m,n}$ is algebraic. More generally, if $M$ (respectively $N$) is a right $\kk \sym_m$-(resp. $\kk \sym_n$-)module, then $(M \boxtimes N) \otimes_{\kk (\sym_m \times \sym_n) }\traceless{m,n}$ is algebraic.
\item 
For $M$ a $\kk \uwb$-module, the object $T^{\bullet, \bullet} \otimes_{\kk \uwb} M$ of $\f (\spmon)$ 
 given by Corollary \ref{cor:T_uwb_to_spmon} is an algebraic functor. 
\end{enumerate}
\end{exam}

One has the full subcategory of torsion algebraic functors $\falg_\tors (\spmon)$, which is again a localizing Serre subcategory, so that one can form the quotient category $\falg (\spmon)/ \falg_\tors (\spmon)$.

\begin{rem}
\label{rem:simples_modulo_torsion}
The set of isomorphism classes of simple objects of this category is identified by Sam and Snowden in \cite[Proposition 3.1.4]{MR3376738}; these are represented by direct summands of $\traceless{a,b}$, for varying $a,b \in \nat$.  

Moreover, in \cite[Proposition 3.1.5]{MR3376738}, they show that $T^{a,b}$ has finite length in $\f (\spmon)/ \f _\tors (\spmon)$. (This is certainly not true in $\f (\spmon)$ (apart from in the exceptional case $a=b=0$.) Indeed,  one can identify the composition factors of $T^{a,b}$  in $\f (\spmon)/ \f _\tors (\spmon)$ using  the identification $\sim$  given in Example \ref{exam:traceless_1,1} equation (\ref{eqn:decompose_Tab}).
\end{rem}

There is an analogous subcategory of algebraic representations
$$
\rep (\gl) 
\subset 
\rep (\gl_\infty).
$$
(See \cite[Section 3.1.1]{MR3376738}, although we allow arbitrary direct sums as in \cite{MR3876732}.)
 Then, by the counterpart of \cite[Theorem 2.5]{MR3876732} (see Section 3.2 of {\em loc. cit.}), the stabilization functor induces an equivalence of categories 
$$
\stabgl :
\falg (\spmon) / \falg _\tors (\spmon)
\stackrel{\simeq}{\rightarrow}
\rep (\gl).
$$

\begin{rem}
\label{rem:stabilization_algebraic}
We will be working with $\falg (\spmon)$ hence the above gives the appropriate notion of stabilization; we can either consider the stabilization of an algebraic functor $F$ as being its image in $\rep (\gl)$ or the image of $F$ in the quotient category $\falg (\spmon) / \falg _\tors (\spmon)$.
\end{rem}

We have the following, which gives a further reason for restricting to algebraic functors: 

\begin{thm}
\label{thm:injectivity_Tmn_falg/tors}
For $m, n \in \nat$, the functor $T^{m,n}$ is injective in $\falg (\spmon) / \falg _\tors (\spmon)$ and hence also in $\falg (\spmon)$.
\end{thm}

\begin{proof}
This is a consequence of \cite[Proposition 3.2.14]{MR3376738} (paying attention to the fact that {\em loc. cit.} restricts to finite length objects).
 (Compare \cite[Corollary 2.6]{MR3876732}, which is the counterpart of this result for $\mathrm{Rep} (\mathbf{O})$, algebraic representations of the infinite orthogonal group.)
\end{proof}

\begin{cor}
\label{cor:hom_T_exact_on_falg/tors}
The functor $\hom_{\spmon} (- , T^{\ast, \ast}) $ induces an exact functor
$$
\hom_{\spmon} (- , T^{\ast, \ast}) : 
\falg (\spmon) / \falg _\tors (\spmon) \op
\rightarrow 
\kk \dwb\dash\modules.
$$
This restricts to an equivalence of categories between the respective full subcategories of finite length objects.
\end{cor}

\begin{exam}
\label{exam:hom_traceless}
For $k,l\in \nat$, 
\[
\hom_{\spmon} (\traceless{a,b}, T^{k,l}) \cong 
\left\{
\begin{array}{ll}
\kk [\sym_a \times \sym_b] & a=k,\  b=l \\
0 & \mbox{otherwise.}
\end{array}
\right.
\]
(Working over $\kk = \cx$, this follows from  \cite[Proposition 3.1.10]{MR3376738}, which relies upon  \cite[Proposition 3.1.8]{MR3376738}. It is also a consequence of \cite[Theorem 2.2]{MR2743762}.)

This exhibits the relationship between the simple objects in the  categories $\falg (\spmon) / \falg _\tors (\spmon) $ and $\kk \dwb\dash\modules$ (cf. Remark \ref{rem:simples_modulo_torsion}).
\end{exam}

\begin{rem}
\label{rem:weak_stabilization}
By Corollary \ref{cor:hom_T_exact_on_falg/tors}, we may consider the functor 
$$
\hom_{\spmon} (- , T^{\ast, \ast}) : 
\falg (\spmon) \op
\rightarrow 
\kk \dwb\dash\modules
$$
as an approximation to the stabilization functor $\stabgl : \falg (\spmon) \rightarrow \rep (\gl)$. We shall refer to this as {\em weak stabilization}.
\end{rem}

\subsection{Application to complexes of $\kk \uwb$-modules}
\label{subsect:weak_stabilization_complexes}

Consider $\fc$, a complex of $\kk \uwb$-modules. Then, applying the generalized Schur functor gives the complex in $\f (\spmon)$:
$$
T^{\bullet, \bullet} \otimes_{\kk \uwb} \fc.
$$
More precisely, this is a complex in $\falg (\spmon)$, using the observation of  Example \ref{exam:algebraic_functors}. This immediately implies the following:

\begin{lem}
\label{lem:homology_algebraic}
For $\fc$ a complex of $\kk \uwb$-modules, the homology of the complex $T^{\bullet, \bullet} \otimes_{\kk \uwb} \fc$ is a graded object in $\falg (\spmon)$.
\end{lem}

We can also apply the functor $\kk ^\uwb \otimes_{\kk \uwb} -$ to $\fc$, which yields the complex of $\kk \uwb$-modules 
$
\kk ^\uwb \otimes_{\kk \uwb} \fc
$.  
These can be related, by applying Proposition \ref{prop:TT_versus_k^uwb}:

\begin{cor}
\label{cor:homology_weak_stabilization}
For $\fc$ a complex of $\kk \uwb$-modules, there is a  natural isomorphism of complexes of $\kk \dwb$-modules:
$$
\hom_{\spmon} ( T^{\bullet, \bullet} \otimes_{\kk \uwb} \fc, T^{\ast, \ast}) 
\cong 
\hom_\kk (\kk ^\uwb \otimes_{\kk \uwb} \fc, \kk).
$$
Hence, passing to homology, there is a natural isomorphism of graded $\kk \dwb$-modules:
$$
\hom_{\spmon} ( H_* (T^{\bullet, \bullet} \otimes_{\kk \uwb} \fc), T^{\ast, \ast})
\cong 
\big( H_* (\kk ^\uwb \otimes_{\kk \uwb} \fc)\big)^\sharp,
$$
i.e., the weak stabilization of the homology $H_* (T^{\bullet, \bullet} \otimes_{\kk \uwb} \fc)$ is isomorphic to the dual of $H_* (\kk ^\uwb \otimes_{\kk \uwb} \fc)$.
\end{cor}

\begin{proof}
The first statement is an immediate consequence of Proposition \ref{prop:TT_versus_k^uwb}. 

Now, Lemma \ref{lem:homology_algebraic} implies that the homology $H_* (T^{\bullet, \bullet} \otimes_{\kk \uwb} \fc)$ is a graded object of $\falg(\spmon)$. Since, for each $m,n \in \nat$, $T^{m,n}$ is injective in $\falg (\spmon)$, by Theorem \ref{thm:injectivity_Tmn_falg/tors}, the weak stabilization $\hom_{\spmon} ( H_* (T^{\bullet, \bullet} \otimes_{\kk \uwb} \fc), T^{\ast, \ast})$ is isomorphic to the homology of the complex
$$
\hom_{\spmon} ( T^{\bullet, \bullet} \otimes_{\kk \uwb} \fc, T^{\ast, \ast}) .
$$
The result then follows from the first part, using that $\hom_\kk (-, \kk)$ is exact working over a field.
\end{proof}

\section{Generalized Schur functors and Koszul complexes}
\label{sect:schur_koszul}

In Section \ref{subsect:koszul_complexes}, we introduced the Koszul complex $
\kzcx \otimes_{\dwbtw} M$ associated to a $\dwbtw$-module $M$. This is a complex of $\kk \uwb$-modules, so we can compose this with the generalized Schur functor introduced in Section \ref{subsect:mixed_tensor}. The purpose of this section is to investigate this composite, which is a complex of $\kk\spmon$-modules. This is in preparation for the application in Section \ref{sect:spmon_operads}.

Throughout this section, $\kk$ is a field of characteristic zero.

\subsection{Composing with the generalized Schur functor}

As recalled above, if $M$ is a $\dwbtw$-module, we have the Koszul complex 
$
\kzcx \otimes_{\dwbtw} M
$ 
in $\kk \uwb$-modules.  
 We also have the generalized Schur functor 
$$
T^{\bullet, \bullet} \otimes _{\kk \uwb} - 
\colon 
\kk \uwb \dash \modules 
\rightarrow 
\f (\spmon).
$$
Composing these  gives the functor with values in complexes in $\f(\spmon)$:
\begin{eqnarray}
\label{eqn:T_K_complex}
T^{\bullet, \bullet} \otimes _{\kk \uwb} \kzcx \otimes_{\dwbtw} - 
: 
\dwbtw\dash\modules 
\rightarrow 
\mathrm{ChCx} (\f (\spmon)).
\end{eqnarray}

\begin{rem}
If using the grading arising from $\kk \uwb$, then one has to put a compatible grading on $T^{\bullet, \bullet}$. One possible choice is to place $T^{m,n}$ in homological degree $n$.
\end{rem}

\begin{prop}
\label{prop:T_Koszul_dwbtw}
Let  $M$ be a $\dwbtw$-module. Then the functor (\ref{eqn:T_K_complex}) satisfies the following properties.
\begin{enumerate}
\item 
The underlying object of $T^{\bullet , \bullet} \otimes_{\kk\uwb} \kzcx \otimes_{\dwbtw} M $ identifies as 
$
T^{\bullet , \bullet} \otimes_{\kk(\tfb)}  M.
$
\item 
The complex splits as  a direct sum of complexes (indexed by $m-n \in \zed$) 
of the form 
\begin{eqnarray}
\label{eqn:T_K_complex_mn}
\ldots 
\rightarrow 
T^{m,n} \otimes_{ \kk (\sym_m \times \sym_n)} M (\m, \n) 
\rightarrow 
T^{m-1,n-1} \otimes_{\kk(\sym_{m-1} \times \sym_{n-1})} M (\mathbf{m-1}, \mathbf{n-1}) 
\rightarrow 
\ldots
\end{eqnarray}
where $n$ determines the homological degree. The differential is  induced by that of $\kzcx$.
\item 
The functor $T^{\bullet , \bullet} \otimes_{\kk \uwb} \kzcx \otimes_{\dwbtw} - $ is  an exact functor from $\dwbtw$-modules to homological complexes in $\f (\spmon)$. 
\end{enumerate}
\end{prop}

\begin{proof}
The first statement follows directly from the form of $\kzcx$. The second follows by using the explicit identification of the differential of $\kzcx$ given in Section \ref{subsect:quadratic_dual_dualizing_complex}; in particular, this shows that the differential maps the component of $ T^{\bullet , \bullet} \otimes_{\kk(\tfb)}  M$ indexed by $(m,n)$ to that indexed by $(m-1, n-1)$, whence the difference $m-n$ is preserved, giving the stated splitting.

Finally, since $\kk$ is a field of characteristic zero by hypothesis, the functor $T^{\bullet , \bullet} \otimes_{\kk (\tfb)} -$ is exact. 
\end{proof}

\begin{rem}
\label{rem:T_K_complex_differential_V}
Evaluating the complex (\ref{eqn:T_K_complex_mn}) on $V$(considered as an object of $\spmon$) gives: 
$$
\ldots 
\rightarrow 
V^{\otimes m} \otimes (V^\sharp)^{\otimes n} \otimes_{ \kk (\sym_m \times \sym_n)} M (\m, \n) 
\rightarrow 
V^{\otimes m-1} \otimes (V^\sharp)^{\otimes n-1} \otimes_{\kk(\sym_{m-1} \times \sym_{n-1})} M (\mathbf{m-1}, \mathbf{n-1}) 
\rightarrow 
\ldots .
$$
The differential involves the contraction map $V \otimes V^\sharp \rightarrow \kk$ as well as the $\dwbtw$ structure maps $M (\m, \n) \rightarrow  M (\mathbf{m-1}, \mathbf{n-1})$; this is described explicitly below.

By Corollary \ref{cor:kdwb_dwbtw-modules}, 
the $\dwbtw$-module structure of $M$ is encoded in the $\kk (\sym_m \times \sym_n)$-module maps (for varying $m$, $n$)
$$
M (\m, \n) \rightarrow   M (\mathbf{m-1}, \mathbf{n-1})\uparrow_{\sym_{m-1} \times \sym_{n-1}}^{\sym_m \times \sym_n}.
$$
Then the differential  is  the composite:
\begin{eqnarray*}
&&(V^{\otimes m} \otimes (V^\sharp)^{\otimes n} )\otimes_{ \kk (\sym_m \times \sym_n)} M (\m, \n) 
\longrightarrow 
\\ 
&&(V^{\otimes m} \otimes (V^\sharp)^{\otimes n}) \otimes_{ \kk (\sym_m \times \sym_n)}  \big( M (\mathbf{m-1}, \mathbf{n-1})\uparrow_{\sym_{m-1} \times \sym_{n-1}}^{\sym_m \times \sym_n}\big)
\\
&&\cong \quad   
\big( (V^{\otimes m-1} \otimes (V^\sharp)^{\otimes n-1}) \otimes_{ \kk (\sym_{m-1} \times \sym_{n-1})}  M (\mathbf{m-1}, \mathbf{n-1})
\big) 
\otimes (V \otimes V^\sharp) 
\longrightarrow \\
&&(V^{\otimes m-1} \otimes (V^\sharp)^{\otimes n-1}) \otimes_{\kk(\sym_{m-1} \times \sym_{n-1})} M (\mathbf{m-1}, \mathbf{n-1}),
\end{eqnarray*}
where the first map is given by the above structure map  and the second is induced by the contraction $ V \otimes V^\sharp  \rightarrow \kk$.
\end{rem}

\subsection{Weak stabilization of the homology}

We consider the complex $T^{\bullet , \bullet} \otimes_{\kk\uwb} \kzcx \otimes_{\dwbtw} M$ as above, which has underlying object $T^{\bullet , \bullet} \otimes_{\kk(\tfb)}  M$ in graded {\em algebraic} $\kk \spmon$-modules.  As in Section \ref{subsect:weak_stabilization_complexes}, we also consider the complex 
$$
\kk ^\uwb \otimes_{\kk \uwb} \kzcx \otimes_{\dwbtw} M
$$
of $\kk \uwb$-modules, which has underlying graded object $\kk ^\uwb \otimes_{\kk (\tfb)}  M$.

On applying $\hom_{\spmon} (-, T^{\ast,\ast})$ to $T^{\bullet , \bullet} \otimes_{\kk\uwb} \kzcx \otimes_{\dwbtw} M$, we have the following identification of complexes (corresponding to the first part of Corollary \ref{cor:homology_weak_stabilization}):

\begin{lem}
\label{lem:hom_Tastast_complexes}
For $M$ a $\dwbtw$-module, there is a natural isomorphism of complexes of $\kk \dwb$-modules:
\begin{eqnarray*}
\hom_{\f(\spmon)} ( T^{\bullet, \bullet} \otimes _{\kk \uwb} \kzcx \otimes_{\dwbtw} M, T^{*,*})
\cong 
\hom_\kk (\kk^\uwb \otimes_{\kk \uwb} \kzcx \otimes_{\dwbtw} M, \kk), 
\end{eqnarray*}
where $\kk^\uwb \otimes_{\kk \uwb} \kzcx \otimes_{\dwbtw} M$ is the complex considered in Section \ref{subsect:second_Koszul_complexes}.
\end{lem}

We then have the following conceptual identification of the weak stabilization of $H_* (T^{\bullet, \bullet} \otimes _{\kk \uwb} \kzcx \otimes_{\dwbtw} M)$:

\begin{prop}
\label{prop:weak_stabilization_homology}
For $M$ a $\dwbtw$-module, the weak stabilization of the homology $H_* (T^{\bullet, \bullet} \otimes _{\kk \uwb} \kzcx \otimes_{\dwbtw} M)$, which is a graded object in $\falg (\spmon)$, identifies as graded objects in $\kk \dwb$-modules:
$$
\hom_{\spmon} (H_* (T^{\bullet, \bullet} \otimes _{\kk \uwb} \kzcx \otimes_{\dwbtw} M), T^{\ast, \ast})
\cong 
\hom_\kk (H_* (\kk ^\uwb \otimes_{\kk \uwb} \kzcx \otimes_{\dwbtw} M ), \kk).
$$
Moreover, the right hand side is naturally isomorphic to  the $\kk$-linear dual of 
$
\tor^{\dwbtw}_* ( \kk (\tfb) , M).
$
\end{prop}

\begin{proof}
The first statement is a special case of the second part of Corollary \ref{cor:homology_weak_stabilization}. The second then follows from Proposition \ref{prop:homology_is_Tor}.
\end{proof}

\subsection{Unstable homology - the universal coefficients spectral sequence}

It is natural to seek to understand the homology of $T^{\bullet, \bullet} \otimes _{\kk \uwb} \kzcx \otimes_{\dwbtw} M$ in terms of that of $\kzcx \otimes_{\dwbtw} M$. (Here we choose to grade the complexes  homologically.) The generalized Schur functor $T^{\bullet, \bullet} \otimes _{\kk \uwb} - $ is right exact but is not exact, which complicates matters. 

\begin{rem}
As observed above, $H_* (T^{\bullet, \bullet} \otimes _{\kk \uwb} \kzcx \otimes_{\dwbtw} M)$ is a graded object in $\falg (\spmon)$. By Proposition \ref{prop:weak_stabilization_homology} we already have an  understanding of the {\em weak stabilization} of this homology.

As will become apparent in the application to operadic structures, it is of significant interest to study these functors (in each homological degree) as objects of $\falg (\spmon)$, i.e., {\em before} stabilization. The weak stabilization loses much information, for example: 
\begin{enumerate}
\item 
weak stabilization does not `see' the torsion submodule;
\item 
for a  composition factor in $\falg(\spmon)/ \falg_\tors(\spmon)$ (i.e., detected by the weak stabilization), a priori one has  no information on the support of the corresponding subquotient in $\falg (\spmon)$.  
\end{enumerate} 
We thus seek other ways of examining this structure.
\end{rem}

 For $M$ a $\dwbtw$-module, there is a natural morphism  of graded algebraic $\kk\spmon$-modules: 
\begin{eqnarray}
\label{eqn:edge_Kunneth_T**}
T^{\bullet, \bullet} \otimes _{\kk \uwb} H_* (\kzcx \otimes_{\dwbtw} M) 
\rightarrow 
H_* (T^{\bullet, \bullet} \otimes _{\kk \uwb} \kzcx \otimes_{\dwbtw} M) 
\end{eqnarray}
but this need not be an isomorphism, since $T^{\bullet, \bullet} \otimes _{\kk \uwb} -$ is not exact.

Since the terms of $\kzcx \otimes_{\dwbtw} M$ are projective as $\kk \uwb$-modules, the above refines to give a universal coefficients spectral sequence, as in Section \ref{subsect:UCT_Koszul}:

\begin{prop}
\label{prop:Kunneth_T_Ext}
For $M$ a $\dwbtw$-module, there is a natural universal coefficients spectral sequence:
$$
\tor_*^{\kk \uwb} (T^{\bullet, \bullet} , H_* (\kzcx \otimes_{\dwbtw} M) ) 
\Rightarrow 
H_* (T^{\bullet, \bullet} \otimes _{\kk \uwb} \kzcx \otimes_{\dwbtw} M) .
$$
The morphism (\ref{eqn:edge_Kunneth_T**}) appears as an edge homomorphism of this spectral sequence.
\end{prop}

\section{Operads and $\kk \dwb$-modules}
\label{sect:operads}

The main purpose of this section is to introduce a $\kk \dwb$-module (respectively a $\dwbtw$-module) that is naturally associated to an operad. 

We first revisit operads, encoding the partial composition operations as a structure map in $\kk (\tfb)$-modules; we also include a `wheeled' contribution, inspired by work of Dotsenko \cite{MR4945404}. 
 Then we exhibit the two fundamental associated structures: the $\kk \dwb$-module  $\stfb^* (\opd \oplus |\dbd \opd |)$ in Theorem \ref{thm:stfb}, and the $\dwbtw$-module $\ltfb^*  \opd \circledcirc \stfb^*|\dbd \opd |$ in Theorem \ref{thm:ltfb}. Once again, the inclusion of the wheeled term (corresponding to $|\dbd \opd |$) is inspired by Dotsenko's work.

\subsection{Operads - recollections and reformulations}

We work with operads in $\kk$-vector spaces. Recall that an operad $\opd$ has an underlying $\kk \fb\op$-module so that, for each $n \in \nat$, $\opd (n)$ is a $\kk \sym_n\op$-module. (Using the isomorphism of categories $\fb \cong \fb\op$ that is the identity on objects and sends a morphism $f$ to $f^{-1}$, $\opd$ can be considered as a $\kk \fb$-module.)

The category $\f (\fb\op)$ is equipped with the monoidal (not symmetric) structure provided by the composition product $\circ$, which has unit $I$, the $\kk\fb\op$-module supported on $\mathbf{1}$ with value $\kk$. An operad is a monoid in this category, thus is given by triple $(\opd, \eta : I \rightarrow \opd, \mu : \opd \circ \opd \rightarrow \opd)$, where $\eta$ is the unit and $\mu$ the product, satisfying the usual axioms. Here, we do not require that an operad should have a unit; the category of such will be denoted $\nuopds$.

Below, we use a definition of operads that corresponds to working with the partial compositions.  For this, rather than work in $\kk \fb$-modules (after modifying variance as above), we work with $\kk (\tfb)$-modules, by setting: 
\[
\opd (\m , \n) := 
\left \{ 
\begin{array}{ll}
\opd (m) & n=1, \mbox{ considered as a $\kk \sym_m \cong \kk (\sym_m \times \sym_1)$-module}
\\
0 & \mbox{otherwise}.
\end{array}
\right.
\]

\begin{rem}
\label{rem:passage_tfb}
This usage of $\kk(\tfb)$-modules may seem artificial. However, it is necessary for allowing the structures to be encoded by modules over the appropriate (twisted) $\kk$-linearization of the downward walled Brauer category. Moreover, this framework becomes essential when one generalizes from operads to {\em dioperads}, as in \cite{2026arXiv260413750P}.
\end{rem}

\begin{exam}
\label{exam:opd_I}
In the $\kk (\tfb)$-module framework, $I$ is supported on $(\mathbf{1}, \mathbf{1})$ with value $\kk$; the latter is denoted $\kk _{(\mathbf{1}, \mathbf{1})}$. 
\end{exam}

\begin{rem}
The following presentation should be compared with the recollections provided by Dotsenko in \cite[Section 2.1]{MR4945404}. The essential difference is that Dotsenko does not work in the $\kk (\tfb)$-module framework. 
\end{rem}

In the $\kk (\tfb)$-module framework, the partial composition structure maps are encoded by the morphism of $\kk (\tfb)$-modules
\[
\pc : 
\shift{1}{0} \opd \circledcirc \shift{0}{1} \opd 
\rightarrow 
\opd
\]
using the Day convolution product $\circledcirc$ (see the recollections in Section \ref{sect:day_convolution}) and the shift functors $\shift{*}{*}$ (see Section \ref{sect:technical}).

\begin{rem}
\label{rem:vanishing_double_shifts}
For $\opd$ an operad (considered as a $\kk (\tfb)$-module), 
\begin{enumerate}
\item 
$\shift{1}{0} \opd$ is supported on objects of $\tfb$ of the form $(\mathbf{s}, \mathbf{1})$, whereas $\shift{0}{1} \opd$ is supported on objects of the form $(\mathbf{t}, \mathbf{0})$;
\item 
$\dbd \opd$ is supported on objects of the form $(\mathbf{s}, \mathbf{0})$; for example, $\dbd I =\kk_{(\mathbf{0}, \mathbf{0})}$, the $\kk (\tfb)$-module supported on $(\mathbf{0}, \mathbf{0})$ with value $\kk$.
\end{enumerate}
This implies the equalities $\shift{0}{1} (\shift{0}{1}\opd)=0$ and  $\shift{0}{1} (\dbd \opd)=0$ that are used below.
\end{rem}

The morphism $\pc$ can be iterated in the following two ways, forming the composites:
$$
\xymatrix@R=0em{
\shift{1}{0}(\shift{1}{0} \opd \circledcirc \shift{0}{1} \opd) \circledcirc  \shift{0}{1} \opd
\ar[rr]^(.6){\shift{1}{0} \pc \circledcirc \id}
&&
\shift{1}{0} \opd \circledcirc \shift{0}{1} \opd 
\ar[r]^(.65){\pc} 
&
\opd
\\
\shift{1}{0} \opd \circledcirc \shift{0}{1} (\shift{1}{0} \opd \circledcirc \shift{0}{1} \opd)
\ar[rr]_(.6){\id \circledcirc \shift{0}{1} \pc }
&&
\shift{1}{0} \opd \circledcirc \shift{0}{1} \opd 
\ar[r]_(.65){\pc}
&
\opd.
}
$$
By Proposition \ref{prop:shift_day_convolution}, there are isomorphisms
\begin{eqnarray*}
\shift{1}{0}(\shift{1}{0} \opd \circledcirc \shift{0}{1} \opd)&\cong & (\shift{2}{0} \opd \circledcirc \shift{0}{1} \opd ) \oplus (\shift{1}{0} \opd \circledcirc \shift{1}{1} \opd)\\
\shift{0}{1}(\shift{1}{0} \opd \circledcirc \shift{0}{1} \opd)&\cong & \shift{1}{1} \opd \circledcirc \shift{0}{1} \opd.
\end{eqnarray*}
Thus $\pc \circ (\shift{1}{0} \pc \circledcirc \id)$  has the two components: 
\begin{eqnarray}
\label{eqn:(10_11)01}
&&
(\shift{1}{0} \opd \circledcirc \shift{1}{1} \opd )\circledcirc \shift{0}{1}\opd 
\rightarrow 
\opd 
\\
\label{eqn:(20_01)_01}
&&
 (\shift{2}{0} \opd \circledcirc \shift{0}{1} \opd )\circledcirc \shift{0}{1}\opd 
\rightarrow 
\opd,
\end{eqnarray}
and  $\pc \circ (\id \circledcirc \shift{0}{1} \pc)$ can be written 
\begin{eqnarray}
\label{eqn:10(11_01)}
&&
\shift{1}{0} \opd \circledcirc (\shift{1}{1} \opd \circledcirc \shift{0}{1}\opd)
\rightarrow 
\opd. 
\end{eqnarray}

\begin{rem}
\label{rem:operad_axioms}
The axioms of the operad $\opd$ (without unit) can be restated as follows (using the associativity of $\circledcirc$ and its symmetry):
\begin{enumerate}
\item 
the morphisms (\ref{eqn:(10_11)01}) and (\ref{eqn:10(11_01)}),  $\shift{1}{0} \opd \circledcirc \shift{1}{1} \opd \circledcirc \shift{0}{1}\opd 
\rightarrow 
\opd $, are equal; 
\item 
the morphism (\ref{eqn:(20_01)_01}), $ \shift{2}{0} \opd \circledcirc \shift{0}{1} \opd \circledcirc \shift{0}{1}\opd 
\rightarrow 
\opd
$,
is symmetric with respect to the interchange of the factors $\shift{0}{1} \opd \circledcirc \shift{0}{1}\opd $.
\end{enumerate}
\end{rem}

The following is standard (usually stated in the unital case): 

\begin{lem}
\label{lem:associative_delta11_opd}
For an operad $\opd$, the partial composition operation $\pc$ induces a product 
\begin{eqnarray}
\label{eqn:pc_delta11_opd}
\pct : 
\dbd \opd
\circledcirc 
\dbd \opd
\rightarrow 
\dbd \opd
\end{eqnarray}
that makes $\dbd \opd$ into an  associative monoid in the symmetric monoidal category $(\f(\tfb), \circledcirc, \kk_{(\mathbf{0}, \mathbf{0})})$. 
\end{lem}

\begin{rem}
Since $\dbd \opd$ is supported on objects of the form $(\mathbf{s}, \mathbf{0})$, Lemma \ref{lem:associative_delta11_opd} has an equivalent formulation in $\kk \fb$-modules with respect to the symmetric monoidal category $(\f (\fb), \odot,\kk_\mathbf{0})$.
\end{rem}

\begin{nota}
\label{nota:commutator_quotient}
Denote by $|\dbd \opd|$ the `commutator quotient', defined as the coequalizer of the two maps
\[
\xymatrix{
\dbd \opd
\circledcirc 
\dbd \opd
\ar@<.5ex>[r]^(.6){\pct} 
\ar@<-.5ex>[r]_(.6){\pct \circ \tau}
&
\dbd \opd,
}
\]
where $\tau$ denotes the symmetry for $\circledcirc$.
\end{nota}

\begin{rem}
\label{rem:if_unit}
Consider $I$ as a $\kk (\tfb)$-module with its usual operad structure and suppose that $\eta : I \rightarrow \opd$ is a morphism of operads (for instance, this could arise from an operad $\opd$ {\em with} unit). Then, the composite 
$$\dbd I
\stackrel{\shift{1}{1}\eta}{\longrightarrow} 
\dbd \opd
\twoheadrightarrow 
|\dbd \opd|
$$ is zero. For example, this applies taking $\eta$ to be the identity on $I$.
\end{rem}

There is more structure:

\begin{lem}
The structure map $\pc$ induces 
\begin{eqnarray}
\label{eqn:action}
&&\shift{1}{0} (\dbd \opd) \circledcirc \shift{0}{1} \opd
\rightarrow 
\dbd \opd
\\
&& 
\label{eqn:2_action}
\shift{1}{0} (\dbd \opd \circledcirc \dbd \opd ) \circledcirc \shift{0}{1} \opd
\rightarrow 
\dbd \opd \circledcirc \dbd \opd.
\end{eqnarray}

The structure map $\pct$ is compatible with these maps via the commutative diagram:
\[
\xymatrix{
\shift{1}{0} (\dbd \opd \circledcirc \dbd \opd ) \circledcirc \shift{0}{1} \opd
\ar[rr]^{\shift{1}{0} \pct \circledcirc \id}
\ar[d]_{(\ref{eqn:2_action})}
&\ &
\shift{1}{0} (\dbd \opd) \circledcirc \shift{0}{1} \opd
\ar[d]^{(\ref{eqn:action})}
\\
\dbd \opd \circledcirc \dbd \opd
\ar[rr]_{\pct}
&\ &
\dbd \opd.
}
\]
Hence, (\ref{eqn:action}) induces 
\[
\alpha : 
\shift{1}{0} (|\dbd \opd|) \circledcirc \shift{0}{1} \opd
\rightarrow 
|\dbd \opd |.
\]
\end{lem}

\begin{rem}
The operad axioms (as recalled in Remark \ref{rem:operad_axioms}) imply compatibility conditions between the action $\alpha$ and the structure map $\pc$.
\end{rem}

\subsection{The $\kk \dwb$-module structure}

Using the functors $\stfb^d$ introduced in Definition \ref{defn:stfb_ltfb}, one can form $\stfb^* (\opd \oplus |\dbd \opd |) := \bigoplus_{d \in \nat} \stfb ^d (\opd\oplus |\dbd \opd |)$ in $\kk (\tfb)$-modules  (see Example \ref{exam:stfb_ltfb} for indications on the structure of such objects). The projection $\opd \oplus |\dbd\opd| \twoheadrightarrow \opd$ induces a surjection in $\kk (\tfb)$-modules:
\[
\stfb^* (\opd \oplus |\dbd \opd |)
\twoheadrightarrow 
\stfb^* \opd.
\]

The following statement is a  consequence of the `exponential' property of the functor $\stfb^*$.

\begin{lem}
\label{lem:stfb_exponential}
There is a natural isomorphism in $\kk(\tfb)$-modules:
\[
\stfb^* (\opd \oplus |\dbd \opd |) \cong \stfb^* \opd \circledcirc \stfb^* |\dbd \opd |.
\]
\end{lem}

\begin{rem}
Since $\shift{1}{1} \opd$ is supported on objects of the form $(\mathbf{s}, \mathbf{0})$, $\stfb^* |\dbd \opd |$  can be identified with $\sfb^* |\dbd\opd|$, forming the symmetric algebra in $\kk \fb$-modules with respect to the convolution product $\odot$; this is then considered as a $\kk (\tfb)$-module in the obvious way. 
\end{rem}

\begin{thm}
\label{thm:stfb}
For $\opd$ a non-unital operad, the structure maps 
\begin{eqnarray*}
\pc &: &
\shift{1}{0} \opd \circledcirc \shift{0}{1} \opd 
\rightarrow 
\opd
\\
\alpha &:&\shift{1}{0} (|\dbd \opd|) \circledcirc \shift{0}{1} \opd
\rightarrow 
|\dbd \opd |
\end{eqnarray*}
induce a natural $\kk \dwb$-module structure on $\stfb^* (\opd \oplus |\dbd \opd |)$. 

Moreover, via the surjection $\stfb^* (\opd \oplus |\dbd \opd |)
\twoheadrightarrow 
\stfb^* \opd$, this induces a natural $\kk \dwb$-module structure on $\stfb^* \opd$.
\end{thm}

\begin{proof}
By construction, $\stfb^* (\opd \oplus |\dbd \opd|) $ has a natural $\kk (\tfb)$-module structure that depends only upon the underlying structure of  $\opd \oplus |\dbd \opd|$. By the $\kk\dwb$ analogue of Corollary \ref{cor:kuwb_uwbtw-modules} (cf. also Corollary \ref{cor:kdwb_dwbtw-modules} and its proof), to construct the $\kk \dwb$-module structure, we need to specify the structure morphism 
$$i^* : \dbd \stfb^* (\opd \oplus |\dbd \opd|) \rightarrow 
\stfb^* (\opd \oplus |\dbd \opd|)$$
 of $\kk (\tfb)$-modules and to check the compatibility condition corresponding to the quadratic relation. (Recall that  $\dbd |\dbd \opd|=0$ and $\shift{0}{1} |\dbd \opd|=0$; this is used without further comment below.)

Now, by Proposition \ref{prop:shift_day_convolution}, for $F$ a $\kk(\tfb)$-module, we have 
\[
\dbd \stfb^* F \cong  \dbd F \circledcirc \stfb^* F \ \oplus \ (\shift{1}{0}F)  \circledcirc (\shift{0}{1}F) \circledcirc \stfb^* F.
\]
This gives 
\begin{eqnarray*}
\dbd \stfb^* (\opd \oplus |\dbd \opd|) 
&\cong& 
\dbd \opd \circledcirc \stfb^* (\opd \oplus |\dbd \opd|) 
\\
&&\ \oplus \ 
(\shift{1}{0}\opd ) \circledcirc (\shift{0}{1} \opd) \circledcirc \stfb^* (\opd \oplus |\dbd \opd|) \\
&&\ \oplus \  
(\shift{1}{0}|\dbd \opd| ) \circledcirc (\shift{0}{1}\opd) \circledcirc \stfb^* (\opd \oplus |\dbd \opd|).
\end{eqnarray*}
Hence, to describe the structure morphism $i^*$, it suffices to describe it on each of the direct summands above. In each case, we make use of the quotient map 
\begin{eqnarray}
\label{eqn:product_map}
(\opd \oplus |\dbd \opd|)
\circledcirc 
\stfb^* (\opd \oplus |\dbd \opd|) 
\twoheadrightarrow 
\stfb^{*>0} (\opd \oplus |\dbd \opd|) 
\end{eqnarray}
 induced by the natural product of the `symmetric algebra' $\stfb^*$.

The respective components of $i^*$ are given by the following composites, in which 
 $\pi$ denotes the projection $\dbd \opd \twoheadrightarrow |\dbd \opd|$ and (in each case)
 the second map is given by  the  product (\ref{eqn:product_map}):
\begin{eqnarray*}
&&(\shift{1}{0}\opd ) \circledcirc (\shift{0}{1}\opd) \circledcirc \stfb^* (\opd \oplus |\dbd \opd|)
\stackrel{\pc \circledcirc \id}{\longrightarrow} 
\opd \circledcirc \stfb^* (\opd \oplus |\dbd \opd|)
\rightarrow 
\stfb^* (\opd \oplus |\dbd \opd|)
\\
&&(\shift{1}{0}|\dbd \opd| ) \circledcirc (\shift{0}{1}\opd) \circledcirc \stfb^* (\opd \oplus |\dbd \opd|)
\stackrel{\alpha \circledcirc \id}{\longrightarrow}
|\dbd \opd |\circledcirc \stfb^* (\opd \oplus |\dbd \opd|)
\rightarrow 
\stfb^* (\opd \oplus |\dbd \opd|)
\\
&&\dbd \opd \circledcirc \stfb^* (\opd \oplus |\dbd \opd|)
\stackrel{\pi}{\twoheadrightarrow }
|\dbd \opd| \circledcirc \stfb^* (\opd \oplus |\dbd \opd|)
\rightarrow 
\stfb^* (\opd \oplus |\dbd \opd|)
.
\end{eqnarray*}
These maps define the required morphism $i^* : \dbd \stfb^* (\opd \oplus |\dbd \opd|) \rightarrow 
\stfb^* (\opd \oplus |\dbd \opd|)$; this is   natural  with respect to $\opd$. 

It remains to check that $i^*$ satisfies the `quadratic relation' analogous to that in Corollary \ref{cor:kdwb_dwbtw-modules}.  This is a straightforward consequence of the axioms of an operad  (see Remark \ref{rem:operad_axioms})  and their consequences for the structure of $|\dbd \opd|$. 

To be concrete, we illustrate the proof schematically as follows. An element of $\opd (\m, \mathbf{1})$ can be viewed as labelling a non-planar corolla with $m$ leaves (labelled by $\m$) and a single root; these are  indicated below by black and white nodes respectively. An element of $| \dbd \opd | (\n, \mathbf{0})$ can be represented by  a (non-oriented) circle with $n$ leaves attached, labelled by $\n$.  
For example (restricting to two leaves in each case and omitting the labelling):  

\noindent
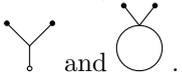
\begin{tikzpicture}[scale = .3]
 \draw [fill=black] (-1,2) circle [radius = .1];
 \draw [fill=black] (1,2)  circle [radius = .1];
 \draw (-1,2) -- (0,1) -- (1,2) (0,1)-- (0,0);
 \draw [fill=white] (0,0) circle [radius = .1];
 \end{tikzpicture}
and 
 \begin{tikzpicture}[scale = .3]
 \coordinate (centre) at (0,0);
\draw (centre) circle[radius =1];  
\draw (canvas polar cs:angle=90,radius=1cm) -- (canvas polar cs:angle=70,radius=2cm);
\draw [fill=black]  (canvas polar cs:angle=70,radius=2cm) circle [radius = .1];
\draw (canvas polar cs:angle=90,radius=1cm) -- (canvas polar cs:angle=110,radius=2cm);
 \draw [fill=black]  (canvas polar cs:angle=110,radius=2cm) circle [radius = .1];
 \end{tikzpicture}
 .

This schematic representation of elements of $| \dbd \opd | (\n, \mathbf{0})$ can be understood via the quotient map $\dbd \opd \twoheadrightarrow |\dbd\opd|$ as the correspondence:
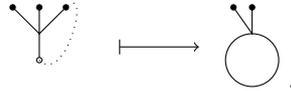
\begin{center}
\begin{tikzpicture}[scale = .35]
\draw [fill=black] (1,2) circle [radius = .1];
 \draw [fill=black] (3,2)  circle [radius = .1];
  \draw [fill=black] (2,2)  circle [radius = .1];
 \draw (1,2) -- (2,1) -- (3,2) (2,2)-- (2,0);
 \draw [fill=white] (2,0) circle [radius = .1]; 
 \draw [dotted] (2,0) .. controls (3,-1) and (4,3) .. (3,2); 

\draw (10,0) circle[radius =1];  
\draw (9.3,2) -- (10,1) -- (10,2);
\draw [fill=black]  (9.3,2) circle [radius = .1];
\draw [fill=black]  (10,2) circle [radius = .1];

  \draw [|->] (5,.5) -- (8,.5);

  \node [right] at (11,-1) {.};
 \end{tikzpicture}
 \end{center}
Here, the dotted edge joins the nodes distinguished by $\shift{1}{1}$ of $\shift{1}{1} \opd$ on the left; the passage to the quotient $|\shift{1}{1}\opd|$ is indicated by the passage to the circle. (Note that the circle is not oriented: the quotient forgets the original distinction between the identified leaf and root.)
 
Similarly, the structure maps $\pc$ and $\alpha$ can be represented as follows 
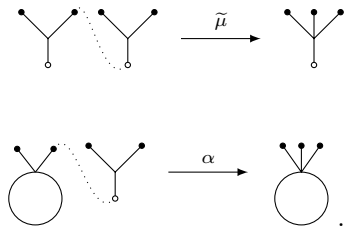
\begin{center}
 \begin{tikzpicture}[scale = .35]

\draw [dotted] (1,2) .. controls (1.5,3) and (2,-1) .. (3,0); 

 \draw [fill=black] (-1,2) circle [radius = .1];
 \draw [fill=black] (1,2)  circle [radius = .1];
 \draw (-1,2) -- (0,1) -- (1,2) (0,1)-- (0,0);
 \draw [fill=white] (0,0) circle [radius = .1];
 
\draw [fill=black] (2,2) circle [radius = .1];
 \draw [fill=black] (4,2)  circle [radius = .1];
 \draw (2,2) -- (3,1) -- (4,2) (3,1)-- (3,0);
 \draw [fill=white] (3,0) circle [radius = .1]; 
 
\draw [fill=black] (9,2) circle [radius = .1];
 \draw [fill=black] (11,2)  circle [radius = .1];
 \draw [fill=black] (10,2)  circle [radius = .1];
 \draw [-latex] (5,1) -- (8,1);
 \draw (9,2) -- (10,1) -- (11,2) (10,2)-- (10,0) ;
 \draw [fill=white] (10,0) circle [radius = .1];
 
 \node[above] at (6.5,1) {$\scriptstyle{\pc}$};
 \end{tikzpicture}
 \end{center}

\begin{center}
\begin{tikzpicture}[scale = .35]
 \coordinate (centre) at (0,0);
\draw (centre) circle[radius =1];  
\draw (canvas polar cs:angle=90,radius=1cm) -- (canvas polar cs:angle=70,radius=2cm);
\draw [fill=black]  (canvas polar cs:angle=70,radius=2cm) circle [radius = .1];
\draw (canvas polar cs:angle=90,radius=1cm) -- (canvas polar cs:angle=110,radius=2cm);
 \draw [fill=black]  (canvas polar cs:angle=110,radius=2cm) circle [radius = .1];
\draw [dotted] (canvas polar cs:angle=70,radius=2cm) .. controls (1.5,3) and (2,-1) .. (3,0);

\draw [fill=black] (2,2) circle [radius = .1];
 \draw [fill=black] (4,2)  circle [radius = .1];
 \draw (2,2) -- (3,1) -- (4,2) (3,1)-- (3,0);
 \draw [fill=white] (3,0) circle [radius = .1];

\draw (10,0) circle[radius =1];  
\draw (9.3,2) -- (10,1) -- (10.7,2) (10,1)-- (10,2);
\draw [fill=black]  (9.3,2) circle [radius = .1];
\draw [fill=black]  (10.7,2) circle [radius = .1];
\draw [fill=black]  (10,2) circle [radius = .1];

  \draw [-latex] (5,1) -- (8,1);
  \node[above] at (6.5,1) {$\scriptstyle{\alpha}$};
  \node [right] at (11,-1) {.};
 \end{tikzpicture}
 \end{center}
Here, the dotted edge links the nodes that are distinguished by $\shift{1}{0}$ and $\shift{0}{1}$ respectively. The map is induced by the operad partial composition operation; this can be thought of as the contraction of the dotted edge (omitting the nodes at the endpoints).

Extending this schematic notation, a basis element of $\stfb^* (\opd \oplus |\dbd \opd|)$ can be represented by a `forest' of (non-planar) corollas and circles, where the leaves and roots are labelled. The action of $\kk \dwb$ is then by contracting edges, where an edge links a leaf to a root, according to the labelling.   

For the proof, it suffices to consider such diagrams that are connected, and  with two dotted edges. Schematically, the possibilities are indicated by the following diagrams, bearing in mind that the leaves i.e., the black nodes (respectively the roots, i.e., the white nodes)  are labelled bijectively by a finite set. 

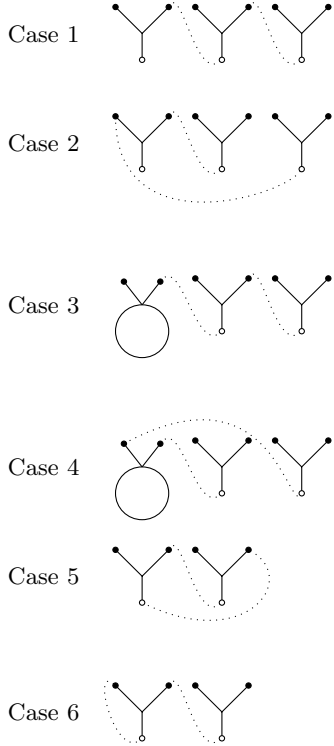
\begin{tikzpicture}[scale = .35]

\draw [dotted] (1,2) .. controls (1.5,3) and (2,-1) .. (3,0); 
\draw [dotted] (4,2) .. controls (4.5,3) and (5,-1) .. (6,0); 
 \draw [fill=black] (-1,2) circle [radius = .1];
 \draw [fill=black] (1,2)  circle [radius = .1];
 \draw (-1,2) -- (0,1) -- (1,2) (0,1)-- (0,0);
 \draw [fill=white] (0,0) circle [radius = .1];
 
\draw [fill=black] (2,2) circle [radius = .1];
 \draw [fill=black] (4,2)  circle [radius = .1];
 \draw (2,2) -- (3,1) -- (4,2) (3,1)-- (3,0);
 \draw [fill=white] (3,0) circle [radius = .1]; 
 
\draw [fill=black] (5,2) circle [radius = .1];
 \draw [fill=black] (7,2)  circle [radius = .1];
 \draw (5,2) -- (6,1) -- (7,2) (6,1)-- (6,0);
 \draw [fill=white] (6,0) circle [radius = .1];
 
 \node[left] at (-2,1) {\small{Case 1}};
 \end{tikzpicture}

\begin{tikzpicture}[scale = .35]

\draw [dotted] (1,2) .. controls (1.5,3) and (2,-1) .. (3,0); 
\draw [dotted] (-1,2) .. controls (-1,-3)  and (5.5,-1) .. (6,0); 
 \draw [fill=black] (-1,2) circle [radius = .1];
 \draw [fill=black] (1,2)  circle [radius = .1];
 \draw (-1,2) -- (0,1) -- (1,2) (0,1)-- (0,0);
 \draw [fill=white] (0,0) circle [radius = .1];
 
\draw [fill=black] (2,2) circle [radius = .1];
 \draw [fill=black] (4,2)  circle [radius = .1];
 \draw (2,2) -- (3,1) -- (4,2) (3,1)-- (3,0);
 \draw [fill=white] (3,0) circle [radius = .1]; 
 
\draw [fill=black] (5,2) circle [radius = .1];
 \draw [fill=black] (7,2)  circle [radius = .1];
 \draw (5,2) -- (6,1) -- (7,2) (6,1)-- (6,0);
 \draw [fill=white] (6,0) circle [radius = .1];
 
  \node[left] at (-2,1) {\small{Case 2}};
 \end{tikzpicture}
 
\begin{tikzpicture}[scale = .35]
 \coordinate (centre) at (0,0);
\draw (centre) circle[radius =1];  
\draw (canvas polar cs:angle=90,radius=1cm) -- (canvas polar cs:angle=70,radius=2cm);
\draw [fill=black]  (canvas polar cs:angle=70,radius=2cm) circle [radius = .1];
\draw (canvas polar cs:angle=90,radius=1cm) -- (canvas polar cs:angle=110,radius=2cm);
 \draw [fill=black]  (canvas polar cs:angle=110,radius=2cm) circle [radius = .1];

\draw [dotted] (canvas polar cs:angle=70,radius=2cm) .. controls (1.5,3) and (2,-1) .. (3,0); 
\draw [dotted] (4,2) .. controls (4.5,3) and (5,-1) .. (6,0); 
 
\draw [fill=black] (2,2) circle [radius = .1];
 \draw [fill=black] (4,2)  circle [radius = .1];
 \draw (2,2) -- (3,1) -- (4,2) (3,1)-- (3,0);
 \draw [fill=white] (3,0) circle [radius = .1]; 
 
\draw [fill=black] (5,2) circle [radius = .1];
 \draw [fill=black] (7,2)  circle [radius = .1];
 \draw (5,2) -- (6,1) -- (7,2) (6,1)-- (6,0);
 \draw [fill=white] (6,0) circle [radius = .1];
 
  \node[left] at (-2,1) {\small{Case 3}};
 \end{tikzpicture}

\begin{tikzpicture}[scale = .35]
 \coordinate (centre) at (0,0);
\draw (centre) circle[radius =1];  
\draw (canvas polar cs:angle=90,radius=1cm) -- (canvas polar cs:angle=70,radius=2cm);
\draw [fill=black]  (canvas polar cs:angle=70,radius=2cm) circle [radius = .1];
\draw (canvas polar cs:angle=90,radius=1cm) -- (canvas polar cs:angle=110,radius=2cm);
 \draw [fill=black]  (canvas polar cs:angle=110,radius=2cm) circle [radius = .1];

\draw [dotted] (canvas polar cs:angle=70,radius=2cm) .. controls (1.5,3) and (2,-1) .. (3,0); 
\draw [dotted] (canvas polar cs:angle=110,radius=2cm) .. controls (5,5)  and (5.5,-1) .. (6,0); 
 
\draw [fill=black] (2,2) circle [radius = .1];
 \draw [fill=black] (4,2)  circle [radius = .1];
 \draw (2,2) -- (3,1) -- (4,2) (3,1)-- (3,0);
 \draw [fill=white] (3,0) circle [radius = .1]; 
 
\draw [fill=black] (5,2) circle [radius = .1];
 \draw [fill=black] (7,2)  circle [radius = .1];
 \draw (5,2) -- (6,1) -- (7,2) (6,1)-- (6,0);
 \draw [fill=white] (6,0) circle [radius = .1];
 
  \node[left] at (-2,1) {\small{Case 4}};
 \end{tikzpicture}

\begin{tikzpicture}[scale = .35]

\draw [dotted] (1,2) .. controls (1.5,3) and (2,-1) .. (3,0); 
\draw [dotted] (0,0) .. controls (4,-2) and (6,1) .. (4,2); 
 \draw [fill=black] (-1,2) circle [radius = .1];
 \draw [fill=black] (1,2)  circle [radius = .1];
 \draw (-1,2) -- (0,1) -- (1,2) (0,1)-- (0,0);
 \draw [fill=white] (0,0) circle [radius = .1];
 
\draw [fill=black] (2,2) circle [radius = .1];
 \draw [fill=black] (4,2)  circle [radius = .1];
 \draw (2,2) -- (3,1) -- (4,2) (3,1)-- (3,0);
 \draw [fill=white] (3,0) circle [radius = .1]; 
 
  \node[left] at (-2,1) {\small{Case 5}};
 \end{tikzpicture}

\begin{tikzpicture}[scale = .35]

\draw [dotted] (1,2) .. controls (1.5,3) and (2,-1) .. (3,0); 
\draw [dotted] (0,0) .. controls (-1,-1) and (-2,3) .. (-1,2); 
 \draw [fill=black] (-1,2) circle [radius = .1];
 \draw [fill=black] (1,2)  circle [radius = .1];
 \draw (-1,2) -- (0,1) -- (1,2) (0,1)-- (0,0);
 \draw [fill=white] (0,0) circle [radius = .1];
 
\draw [fill=black] (2,2) circle [radius = .1];
 \draw [fill=black] (4,2)  circle [radius = .1];
 \draw (2,2) -- (3,1) -- (4,2) (3,1)-- (3,0);
 \draw [fill=white] (3,0) circle [radius = .1]; 
 
  \node[left] at (-2,1) {\small{Case 6}};
 \end{tikzpicture}
 
That the edge contractions `commute' is seen as follows:
\begin{enumerate}
\item 
for Cases 1 and 3, one uses the equality of (\ref{eqn:(10_11)01}) and (\ref{eqn:10(11_01)}); 
\item  
for Cases 2 and 4, one uses the symmetry of (\ref{eqn:(20_01)_01}); 
\item 
Case 5 relies on having passed to the commutator quotient $|\dbd \opd|$; 
\item 
for Case 6, one uses the definition of the action $\alpha$, which is induced by $\pc$. 
\end{enumerate} 

One deduces that the quadratic relations are satisfied, as required. 
\end{proof}

\begin{rem}
\ 
\begin{enumerate}
\item 
This can be compared with the construction of the (wheeled) PROP associated to a (wheeled) operad. Related to the cyclic operad case, this occurs in the work of Hinich and Vaintrob \cite{MR1913297}. 
\item 
This is analogous to the case of modular operads, for which one obtains a module over $\kk \db$, where $\db$ is the downward Brauer category (see \cite{P_cyclic}, which treats the case of cyclic operads). This follows from the results of Stoll \cite{MR4541945}, who exhibits modular operads as algebras over the Brauer {\em properad}.  
\item 
Such ideas also occur, for example, in the work of Raynor \cite{2024arXiv241220260R}.
\end{enumerate}
\end{rem}

\subsection{The $\dwbtw$-module structure}

There is a twisted version of Theorem \ref{thm:stfb}. Heuristically, this can be obtained by  placing $\opd$ in degree one and $|\dbd \opd|$ in degree zero and keeping track of the  Koszul signs. In particular, with this grading, on applying $\stfb^*(-)$ using Koszul signs, and using the functors $\ltfb^d$ introduced in Definition \ref{defn:stfb_ltfb}, one obtains 
\[
\ltfb^*  \opd \circledcirc \stfb^*|\dbd \opd |
\]
using the  functors $\ltfb^*$ and $\stfb^*$ in the {\em ungraded} setting. (Compare the expression in Lemma \ref{lem:stfb_exponential}.)

This can be formalized by using the counterpart of Corollary \ref{cor:untwist_dwb} using the object $\triv \boxtimes \sgn$ of $\f (\dwbord)$ in place of $\sgn \boxtimes \triv$; this induces an equivalence of categories between $\kk \dwb$-modules and $\dwbtw$-modules.

The following is the crucial ingredient.

\begin{lem}
\label{lem:twist_stfb_operad}
Consider $\triv \boxtimes \sgn$ as a $\kk (\tfb)$-module. Then there is an isomorphism of $\kk (\tfb)$-modules:
$$
(\triv \boxtimes \sgn) \otimes \stfb^* (\opd \oplus |\dbd \opd|) 
\cong 
\ltfb^*  \opd \circledcirc \stfb^*|\dbd \opd |.
$$
\end{lem}

\begin{proof}
This is a case of the following general result. Suppose that $F_0$ and $F_1$ are $\kk (\tfb)$-modules such that, for $i \in \{0, 1\}$,  $F_i$ is supported on objects of the form $(\mathbf{s}, \mathbf{i})$, where $s \in \nat$. Then we claim that there is a natural isomorphism
\begin{eqnarray}
\label{eqn:compare_twist}
(\triv \boxtimes \sgn) \otimes \stfb^* (F_1 \oplus F_0) 
\cong 
\ltfb^* F_1  \circledcirc \stfb^* F_0.
\end{eqnarray}

This is shown as follows. Lemma \ref{lem:stfb_exponential} generalizes to give the isomorphism of $\kk(\tfb)$-modules 
$$
\stfb^* (F_1 \oplus F_0) 
\cong 
\stfb^* F_1 \circledcirc \stfb^* F_0.
$$
Using the isomorphism of the underlying $\kk$-vector spaces of sections given in Example \ref{exam:stfb_ltfb}, it follows that, for any $m, n \in \nat$, evaluating the two sides of (\ref{eqn:compare_twist}) on $(\m,\n)$ gives  isomorphic $\kk$-vector spaces.

It remains to check that one has an isomorphism of $\kk (\sym_m \times \sym_n)$-modules. This follows by analysing the definitions of $\stfb^*$ and $\ltfb^*$, extending the example given in  Example \ref{exam:stfb_ltfb}.
\end{proof}

Using this Lemma, Theorem \ref{thm:stfb} implies the following:

\begin{thm}
\label{thm:ltfb}
For $\opd$ an operad, the structure maps 
\begin{eqnarray*}
\pc &: &
\shift{1}{0} \opd \circledcirc \shift{0}{1} \opd 
\rightarrow 
\opd
\\
\alpha &:&\shift{1}{0}(|\dbd \opd|) \circledcirc \shift{0}{1} \opd
\rightarrow 
|\dbd \opd |
\end{eqnarray*}
induce a natural $ \dwbtw$-module structure on $\ltfb^*  \opd \circledcirc \stfb^*|\dbd \opd |$. 

Moreover, via the surjection $\ltfb^* \opd \circledcirc \stfb^*|\dbd \opd |
\twoheadrightarrow 
\ltfb^* \opd$ induced by sending $|\dbd\opd|$ to zero,  this induces a natural $\dwbtw$-module structure on $\ltfb^* \opd$.
\end{thm}

\begin{rem}
One can also prove Theorem \ref{thm:ltfb} directly, as for Theorem \ref{thm:stfb}, keeping track of the signs that arise. 
We are using the convention that, in $\pc$, the term $\shift{0}{1} \opd$ appears on the right of $\circledcirc$. We sketch the argument here (which was the author's original approach), since this may be illuminating.

Proposition \ref{prop:shift_day_convolution} yields the decomposition
\begin{eqnarray}
\label{eqn:dbd_lftb_stfb}
\dbd \big(\ltfb^* \opd \circledcirc \stfb^*  (|\dbd \opd|)\big) 
&\cong& 
\dbd \opd \circledcirc \big( \ltfb^* \opd \circledcirc \stfb^*( |\dbd \opd|)\big) 
\\
\nonumber
&&\ \oplus \ 
(\shift{1}{0}\opd ) \circledcirc (\shift{0}{1} \opd) \circledcirc
\big( \ltfb^* \opd \circledcirc \stfb^*( |\dbd \opd|)\big) 
 \\
 \nonumber
&&\ \oplus \  
(\shift{1}{0}|\dbd \opd| ) \circledcirc (\shift{0}{1}\opd) \circledcirc 
\big( \ltfb^* \opd \circledcirc \stfb^*( |\dbd \opd|)\big).
\end{eqnarray}
(Note that these isomorphisms take into account the twist by the sign representations used in defining $\ltfb^*$.) 

The product map (\ref{eqn:product_map}) is replaced by the corresponding product 
\begin{eqnarray}
\label{eqn:product_map_lftb}
(\opd \oplus |\dbd \opd|)
\circledcirc 
\big( \ltfb^* \opd \circledcirc \stfb^*( |\dbd \opd|)\big) 
\rightarrow 
\ltfb^* \opd \circledcirc \stfb^*( |\dbd \opd|). 
\end{eqnarray}

Using the above ingredients, the structure map $i^*$ is defined exactly as in Theorem \ref{thm:stfb}, {\em mutatis mutandis}.

It remains to check that the quadratic relations for $\dwbtw$ are satisfied, i.e., that reversing the order of contraction of the edges introduces a sign $-1$. 
This is a consequence of the isomorphism (\ref{eqn:dbd_lftb_stfb}). In terms of the schematic representation in the proof of Theorem \ref{thm:stfb}, this can be seen as follows. By definition, permuting a wheel and a corolla does not introduce a sign; {\em a contrario}, transposing two corollas introduces a sign $-1$.
 By inspection, in order to reverse the order of edge contraction, in each case we are required to transpose two corollas. 
\end{rem}

\begin{rem}
\label{rem:operadic_suspension}
One can generalize the results of this section to working with non-unital operads in graded $\kk$-vector spaces. 
 It is then natural to ask how Theorems \ref{thm:stfb} and \ref{thm:ltfb} relate when considering a non-unital operad $\opd$ and its operadic suspension. When including the `wheeled' term $|\dbd \opd|$ (respectively for the suspended operad), in general there is no direct relation, since the formation of the commutator quotient $|-|$ is sensitive to the Koszul sign arising from the operadic suspension. (This observation is due to Vladimir Dotsenko relating to his work.)
\end{rem}


\subsection{The case of an associative algebra}
\label{subsect:assoc}

A non-unital operad that is supported on $(\mathbf{1}, \mathbf{1})$ is equivalent to an associative  (not necessarily unital) algebra in $\kk$-vector spaces. Hence, given a (not-necessarily unital) associative algebra $A$, one can apply Theorems \ref{thm:stfb} and \ref{thm:ltfb}.

\begin{rem}
\label{rem:analogy_Loday}
One can view the construction of this section as being a non-commutative analogue of the Loday construction for commutative, associative algebras. 
\end{rem}

We consider $A$ as  a $\kk (\tfb)$-bimodule supported on $(\mathbf{1}, \mathbf{1})$,  $|\dbd A|$ is the usual commutator quotient 
\[
|A |:= A/[A,A],
\]
supported on $(\mathbf{0}, \mathbf{0})$.

\begin{exam}
\label{exam:sftb_A_oplus_|A|}
Applying Theorem \ref{thm:stfb} yields the $\kk \dwb$-module 
\[
\stfb^* (A \oplus |A|) \cong \stfb^* (A) \circledcirc \stfb^* ( |A|), 
\]
which is supported on objects of the form $(\mathbf{t}, \mathbf{t})$. More particularly, $\stfb^* ( |A|)$ is isomorphic to the symmetric algebra $S^* (|A|)$ supported on $(\mathbf{0}, \mathbf{0})$, whereas $\stfb^* (A) = \bigoplus_{t \in \nat} \stfb^t (A) $, where $ \stfb^t (A)$ is supported on  $(\mathbf{t}, \mathbf{t})$. As a $\kk$-vector space we have the identification 
\[
 \stfb^t (A) (\mathbf{t}, \mathbf{t}) \cong A^{\otimes t} \otimes \kk \sym_t.
\]
One way of visualizing this is as 
\[
\bigoplus_{\pi \in \wpair_t (\mathbf{t}, \mathbf{t})}  A^{\otimes t},
\]
where, for a pairing $\pi$, the tensor product is indexed by the pairs in $\pi$.

For example, taking $t=2$, there are two possible pairings, $\{ (1,1), (2,2)\}$ and $\{ (1,2) , (2,1)\}$ so that 
\[
\stfb^2 (A) (\mathbf{2}, \mathbf{2}) = A_{(1,1)}\otimes A_{(2,2)} \  \oplus \ A_{(1,2)}\otimes A_{(2,1)},
\]
where the expression for the tensor product reflects a choice of the order, and the indices correspond to the pairs. The action of the group $\sym_2 \times \sym_2$ can be read off easily from this expression. 

Consider the restriction of the $\kk\dwb$-action to  $\stfb^2 (A)$, in particular the action of $\kk \dwb ((\mathbf{2}, \mathbf{2}) , (\mathbf{1}, \mathbf{1}))$. Since we already understand the action of the symmetric groups, it suffices to consider the action of $[\iota_{2,2}\op]$, which is defined so as to `contract' the pair $(2,2)$.

On the first summand, this acts via the canonical quotient map $A \rightarrow |A|$ giving 
\[
A_{(1,1)}\otimes A_{(2,2)} \rightarrow A_{(1,1)} \otimes |A|.
\]
On the second factor, because of our choice for  the conventions in defining the structure maps, we first swap the factors and then act via the product of $A$:
\[
A_{(1,2)}\otimes A_{(2,1)}
\cong 
A_{(2,1)}\otimes A_{(1,2)}
\rightarrow 
A_{(1,1)}.
\]

This analysis extends to give the full $\kk \dwb$-module structure on $\stfb^* (A \oplus |A|)$, following the recipe given in the proof of Theorem \ref{thm:stfb}.
\end{exam} 
 
\begin{rem}
\ 
\begin{enumerate}
\item
If $\dim |A|\geq 1$, then $S^* (|A|)$ has infinite dimension, whence so does $\stfb^* (A \oplus |A|)$.
\item 
If $A$ has finite dimension, then $\stfb^t (A)$ has finite dimension for all $t \in \nat$.
\item 
The quotient map $A \rightarrow |A|$ induces a `differential' on $\stfb^* (A \oplus |A|)$; this is made more precise in Section 9.
\end{enumerate}
\end{rem}

\begin{exam}
\label{exam:lftb_A_oplus_|A|}
Similarly to Example \ref{exam:sftb_A_oplus_|A|}, one can  consider the $\dwbtw$-module $\ltfb ^* (A) \circledcirc \stfb^* (|A|)$ as in Theorem \ref{thm:ltfb}. The analysis of this structure is very similar, but paying attention to the `Koszul-type' signs that arise. 
\end{exam}

\begin{rem}
\label{rem:dotsenko}
In \cite{MR4945404}, Dotsenko also specializes to the case of an associative algebra (but with unit). His homological results then yield the Loday-Tsygan-Quillen theorem on the homology of the Lie algebra of infinite matrices with entries in the algebra corresponding to the operad. The structures appearing here can be viewed as precursors of this Lie algebra, via the general method of Section \ref{sect:spmon_operads}. 
\end{rem}

\section{Graphs}
\label{sect:graphs}

The purpose of this section is to present a definition of the category of edge-directed graphs, by exploiting the category $\uwb$. This is a variant of the category of graphs as defined in \cite[Appendix A]{MR3636409}, using the category $\uwb$ to encode the edge structure rather than involutions. Edges are directed; this  is encoded by the partitioning  of the half edges of a graph into positive (entering) and negative (exiting) subsets (this also applies to the legs (or hairs) of the graph).

We then restrict to edge-directed graphs with a flow (using the terminology introduced here). This is in preparation for analysing the  Koszul complexes (see Section \ref{sect:uwb_quadratic_koszul}) obtained from the modules  derived from operads see Section \ref{sect:operads}), in particular explaining why we obtain  a  {\em hairy flow-graph complex}. 

\subsection{Edge-directed graphs}
In this section we introduce edge-directed graphs; these are more general than is required for the applications arising in the operadic framework; the appropriate {\em graphs with a flow} for that framework are introduced in Section \ref{subsect:flow-graph}. 

Graphs are defined in terms of their half edges and vertices; this allows for legs (or hairs), namely those half edges that do not form part of an edge.

\begin{rem}
All graphs that we consider are finite (they have finitely-many half edges and vertices) and have no isolated vertices. They are not necessarily connected.
\end{rem}

Recall that $\fs$ is the category of finite sets and surjective maps; this is a wide subcategory of $\fin$, the category of finite sets and all maps. The category $\fs$ has skeleton $\{ \n \mid n \in \nat\}$.

\begin{defn}
\label{defn:ed-graphs}
An edge-directed graph is a quintuple of finite sets $(V_\Gamma, X_\Gamma^+ ,X_\Gamma^- , L_\Gamma^+ , L_\Gamma^-)$ (here $V_\Gamma$ is the set of vertices, $X_\Gamma:= X_\Gamma^+ \amalg X_\Gamma^-$ is the set of half edges, and $L_\Gamma:= L_\Gamma^+ \amalg L_\Gamma^-$ is the set of legs of the graph), 
together with structure maps:
\begin{eqnarray*}
p_\Gamma & \in & \fs (X_\Gamma, V_\Gamma) \\
f_\Gamma & \in & \uwb ((L_\Gamma^+, L_\Gamma^-) , (X_\Gamma^+, X_\Gamma^-)).
\end{eqnarray*} 

For $\Gamma'$ a second graph with $L_\Gamma^+=L_{\Gamma'}^+$ and $L_\Gamma^-=L_{\Gamma'}^-$, a morphism $\Phi : \Gamma \rightarrow \Gamma '$ is given by a pair of maps $\Phi^X \in \uwb ( (X_{\Gamma'}^+, X_{\Gamma'}^-), (X_\Gamma^+, X_\Gamma^-))$ and $\Phi^V \in \fs (V_\Gamma, V_{\Gamma'})$ such that the following conditions are satisfied.
\begin{enumerate}
\item 
The following diagram commutes in $\uwb$:
\begin{eqnarray}
\label{eqn:PhiX}
\xymatrix{
& (L_\Gamma^+, L_\Gamma^-)
\ar[ld]_{f_{\Gamma'}}
\ar[rd]^{f_{\Gamma}}
\\
 (X_{\Gamma'}^+, X_{\Gamma'}^-)
\ar[rr]_{\Phi^X}
&&
(X_\Gamma^+, X_\Gamma^-).
}
\end{eqnarray}
\item 
The following diagram commutes in $\fin$:
\begin{eqnarray}
\label{eqn:PhiV}
\xymatrix{
X_{\Gamma'}
\ar@{^(->}[r]^{\widetilde{\Phi^X}}
\ar@{->>}[d]_{p_{\Gamma'}}
&
X_\Gamma 
\ar@{->>}[d]^{p_\Gamma}
\\
V_{\Gamma'}
&
V_{\Gamma},
\ar@{->>}[l]^{\Phi^V}
}
\end{eqnarray}
where $\widetilde{\Phi^X} : X_{\Gamma'}^+ \amalg X_{\Gamma'}^- \rightarrow  X_{\Gamma}^+ \amalg X_{\Gamma}^- $ is the disjoint union of the two injections underlying $\Phi^X$.
\item 
For each ordered pair $(h_1, h_2) \in (X_{\Gamma}^+ , X_{\Gamma}^-)$  appearing in the morphism $\Phi^X$, we have $\Phi^V p_\Gamma (h_1) = \Phi^V p_\Gamma (h_2)$.
\end{enumerate}
\end{defn}

\begin{rem}
\ 
\begin{enumerate}
\item 
The map $p_\Gamma$ determines to which vertex a half edge is attached; $f_\Gamma$ labels the legs (distinguishing the $+$ and $-$ legs) and determines the edges, which are given by the pairs appearing in the definition; these pair an element of $X_\Gamma^+$ with one of $X_\Gamma^-$, which determines the `edge direction'.
\item 
The surjectivity of $p_\Gamma$ corresponds to the hypothesis that $\Gamma$ has no isolated vertices. 
\item 
Above, we only define morphisms of graphs that fix the leg structure. 
\item 
The set of pairs involved in defining the morphism $\Phi^X$ corresponds to the set of edges of $X_\Gamma$ that are `contracted' by the morphism $\Phi$; the commutativity of (\ref{eqn:PhiX}) ensures that these pairs are indeed `edges' of $\Gamma$. The final condition on the morphism $\Phi$ ensures that the endpoints (i.e., vertices) of an edge that is contracted are identified under $\Phi^V$, reaffirming the intuition of `contraction'.
\end{enumerate}
\end{rem}

We have the following finiteness property:

\begin{lem}
\label{lem:graph_finiteness}
For a given quintuple  $(V, X^+ ,X^- , L^+ , L^-)$, the set of graphs $\Gamma$ on this quintuple is 
$$\uwb ((L^+, L^-), (X^+, X^-)) \times \fs (X^+ \amalg X^- , V),$$
 in particular this is a finite set. 

Moreover, for a quadruple  $(X^+ ,X^- , L^+ , L^-)$, the set of graphs $\Gamma$ with this structure for some $V$ in the skeleton of $\fs$ is finite. 
\end{lem}

\begin{proof}
The first statement is immediate. The second follows since $\fs (X^+ \amalg X^- , V) $ is empty if $|V| > |X^+| + |X^-|$.
\end{proof}

Given a graph $\Gamma$, we use the shorthand $X_\Gamma^+ \backslash L_\Gamma^+$ and $X_\Gamma^- \backslash L_\Gamma^-$ to indicate the respective complements of the images of the injections underlying $f_\Gamma$.

\begin{prop}
\label{prop:graph_automorphisms}
For $\Phi : \Gamma \rightarrow \Gamma'$ a morphism between graphs $\Gamma$ and $\Gamma'$ on $(L_\Gamma^+, L_\Gamma^-)$, $\Phi$ is an isomorphism if and only if both $\Phi^X$ and $\Phi^V$ are isomorphisms. If so, then $\Phi^V$ is uniquely determined by $\Phi^X$. 

The group of automorphisms of $\Gamma$ (fixing the leg structure) is the subgroup formed by elements $(\alpha_+, \alpha_-)\in \aut (X_\Gamma^+ \backslash L_\Gamma^+) \times \aut (X_\Gamma^- \backslash L_\Gamma^-)$ that are compatible with the vertex and the edge structure, in the following sense:
\begin{enumerate}
\item  
for all $x_+, x_+' \in X_\Gamma^+ \backslash L_\Gamma^+$ (respectively $x_-, x_-' \in X_\Gamma^- \backslash L_\Gamma^-$), $p_\Gamma \alpha_+(x_+) = p_\Gamma \alpha_+(x'_+)$
if and only if $p_\Gamma x_+ = p_\Gamma x'_+$  (respectively  $p_\Gamma \alpha_-(x_-) = p_\Gamma \alpha_-(x'_-)$
if and only if $p_\Gamma x_- = p_\Gamma x'_-$ );
\item 
for all edges  $(h_+, h_-) \in  X_\Gamma^+ \backslash L_\Gamma^+ \times X_\Gamma^- \backslash L_\Gamma^-$ of $\Gamma$, $(\alpha_+ h_+, \alpha_- h_-)$ is an edge of $\Gamma$.
\end{enumerate}
\end{prop}

\begin{proof}
The equivalent condition for $\Phi$ to be an isomorphism is immediate; the commutativity of (\ref{eqn:PhiV}) together with the surjectivity of $p_\Gamma$ then shows that $\Phi^X$ 
 determines $\Phi^V$ when $\Phi$ is an isomorphism.
 
 The statement about graph automorphisms simply translates the commutativity conditions of (\ref{eqn:PhiX}) and (\ref{eqn:PhiV}) in terms of the underlying pair of automorphisms in $\aut (X_\Gamma^+ \backslash L_\Gamma^+) \times \aut (X_\Gamma^- \backslash L_\Gamma^-)$, without referring to $\Phi^V$. 
\end{proof}

One can contract a single edge of a graph $\Gamma$ (assuming that one exists and that this edge is not `isolated'):

\begin{prop}
\label{prop:single_edge_contraction}
Let $\Gamma$ be a graph such that $f_\Gamma$ is not a bijection, and choose $(h_+, h_-)$ a pair arising in $f_\Gamma$ (i.e., an edge of $\Gamma$).
Suppose that at least one of $|p_\Gamma^{-1} (p_\Gamma h_+) |$, $|p_\Gamma^{-1} (p_\Gamma h_-) |$ has cardinality greater than one.

Set 
\begin{eqnarray*}
&&X_{\Gamma'}^+ := X_\Gamma \backslash \{h_+\} \mbox { and } X_{\Gamma'}^- := X_\Gamma \backslash \{h_-\} 
\\
&&
V_{\Gamma'} := V_\Gamma/ _{(p_\Gamma h_+ \sim p_\Gamma h_-)}.
\end{eqnarray*}

Then there is a unique graph $\Gamma'$ with legs $(L_\Gamma^+, L_\Gamma^-)$ such that $\Phi^X$ defined by the inclusions $X_{\Gamma'}^\pm \subset X_\Gamma^\pm$ and $\Phi^V$ the quotient map $V_\Gamma \twoheadrightarrow V_{\Gamma'}$ define a morphism of graphs preserving the leg structure.
\end{prop}

\begin{proof}
The inclusions $X_{\Gamma'}^\pm \subset X_\Gamma^\pm$ define a morphism $\Phi^X \in \uwb ((X_{\Gamma'}^+, X_{\Gamma'}^-),(X_{\Gamma}^+, X_{\Gamma}^-))$; clearly there is a unique $f_{\Gamma'}$ that makes the diagram (\ref{eqn:PhiX}) commute. Likewise, $\Phi^X$ and $\Phi^V$ induce a map $p_{\Gamma'}$ that makes the diagram (\ref{eqn:PhiV}). The only point that remains to check here is the surjectivity; this is where the cardinality hypothesis is used.
\end{proof}

\begin{rem}
\label{rem:graph_morphisms}
\ 
\begin{enumerate}
\item 
There is a further type of morphism between edge-directed graphs that is allowed by Definition \ref{defn:ed-graphs}:  vertex identification. Namely, given a graph $\Gamma$ and any surjective map $\Phi^V : V_\Gamma \rightarrow V'$, there is a unique graph $\Gamma'$ on $(V', X_\Gamma^+ ,X_\Gamma^- , L_\Gamma^+ , L_\Gamma^-)$    such that $\Phi^X = \id$ and $\Phi^V$ define a morphism of graphs $\Phi : \Gamma \rightarrow \Gamma'$. (These do not arise when considering graph complexes here.)
\item 
All morphisms of the category of edge-directed graphs can be obtained by successively composing isomorphisms (as in Proposition \ref{prop:graph_automorphisms}),  single edge contractions (as in Proposition \ref{prop:single_edge_contraction}), and vertex identifications.
\end{enumerate}
\end{rem}

There is further naturality, by allowing the legs to be brought into play:

\begin{prop}
\label{prop:edge-directed_naturality_ell}
Let $\Gamma$ be an edge-directed graph with legs $(L_\Gamma^+, L_\Gamma^-)$ and take $\ell \in \uwb ((L^+, L^-), (L_\Gamma^+, L_\Gamma^-))$ for a pair of finite sets $(L_+, L_-)$. Then there is an edge-directed graph $\Gamma^\ell$ on $(V_\Gamma, X_\Gamma^+, X_\Gamma^-, L^+, L^-)$ with structure maps $f_\Gamma \circ \ell \in \uwb ((L^+, L^-), (X_\Gamma^+, X_\Gamma^-))$ and $p_\Gamma$. 

The association  $\Gamma \mapsto \Gamma^\ell$ defines a functor from the category of edge-directed graphs with legs $(L_\Gamma^+, L_\Gamma^-)$ to that of edge-directed graphs with legs $(L^+, L^-)$.
\end{prop}

\begin{proof}
That $\Gamma^\ell$ is an edge-directed graph is immediate. The naturality of the construction follows by inspection of the conditions in the definition of a morphism $\Phi$.
\end{proof}

\begin{rem}
This structure forms part of a larger category of edge-directed graphs,  without fixed leg structure. 
\end{rem}

\begin{rem}
\label{rem:underlying_graph}
If $\Gamma$ is an edge-directed graph with legs $(L_\Gamma^+, L_\Gamma^-)$, we can construct the underlying graph without legs by replacing $X_\Gamma^+$ by $X_\Gamma^+ \backslash L_\Gamma^+$ and $X_\Gamma^-$ by $X_\Gamma^- \backslash L_\Gamma^-$ (and removing any isolated vertices thus created) by using the obvious structure morphisms.

This is useful since it allows us to use the terminology for graphs without legs. For instance, when considering trees, the leaves (or root) are the $1$-valent vertices of the underlying graph; these may have legs attached in the full graph.  
\end{rem}

\subsection{Connected components}

As usual, one can decompose edge-directed graphs into connected components. The basic identity is provided by the following:

\begin{lem}
\label{lem:equiv_Vertex}
For $\Gamma$ an edge-directed graph, there is an equivalence relation $\sim$ on $V_\Gamma$ generated by the elementary relation $\elrel$
$$
p_\Gamma h_+ \elrel p_\Gamma h_- \mbox{ if $(h_+,h_-)$ is an edge of $\Gamma$.}  
$$ 
\end{lem}

This allows the following Definition to be given:

\begin{defn}
\label{defn:connectivity}
An edge-connected graph $\Gamma$ is connected if $V_\Gamma$ forms a single $\sim$-equivalence class.
\end{defn}

\subsection{Graphs with a flow}
\label{subsect:flow-graph}

For the applications arising from operads (more generally, wheeled operads), we impose a further, very restrictive condition on the edge-directed graphs that we consider: we require that each vertex has at most one `$-$' half edge attached (where this exists, it is  the unique outgoing half edge of the vertex).

\begin{defn}
\label{defn:graphs_with_a_flow}
A graph with a flow is an edge-directed graph $\Gamma$ (as in Definition \ref{defn:ed-graphs}) such that, for each vertex $v \in V_\Gamma$, the set $p_\Gamma^{-1} (v) \cap X_\Gamma^-$ has cardinality at most one. 

The category of graphs with a flow (with fixed leg structure $(L_\Gamma^+, L_\Gamma^-)$) is the full subcategory of edge-directed graphs on the graphs with a flow.
\end{defn} 

We isolate the fundamental property which justifies the terminology flow in the following Lemma:

\begin{lem}
\label{lem:flow_trichotomy}
Let $\Gamma$ be a graph with a flow and $v$ be a vertex in $V_\Gamma$. Then precisely one of the following holds:
\begin{enumerate}
\item 
$p_\Gamma^{-1} (v) \cap X_\Gamma^- =\emptyset$; 
\item 
$|p_\Gamma^{-1} (v) \cap X_\Gamma^-| = |p_\Gamma^{-1} (v) \cap L_\Gamma^-| =1;$
\item 
there exists a unique edge $(h_+, h_-)$ with $p_\Gamma (h_-) = v$ (so that $w \elrel v$, where $w:= p_\Gamma (h_+)$).
\end{enumerate}
\end{lem}

\begin{proof}
By definition, $|p_\Gamma^{-1} (v) \cap X_\Gamma^-|$ is either zero or one. In the case $|p_\Gamma^{-1} (v) \cap X_\Gamma^-|=1$,   either $|p_\Gamma^{-1} (v) \cap L_\Gamma^- |=1$ or $p_\Gamma^{-1} (v) \cap L_\Gamma^- =\emptyset$. If $p_\Gamma^{-1} (v) \cap L_\Gamma^- =\emptyset$, the element $h_-\in p_\Gamma^{-1} (v) \cap X_\Gamma^-$ must belong to a unique edge in $\Gamma$, which we write $(h_+, h_-)$.  This yields the three cases in the statement.
\end{proof}

\begin{rem}
Lemma \ref{lem:flow_trichotomy} implies that, if $w \elrel v$, there is a unique edge $(h_+, h_-)$ realizing this and $w$ is determined by $v$. We view this edge as the unique outgoing edge from $v$:
\begin{center}
\begin{tikzpicture}[scale = 0.5]
\node [below]  at (0,-.1) {$\scriptstyle{v}$}; 
 \draw [fill=black] (0,0) circle [radius = .1];
\node [below] at (2,-.1) {$\scriptstyle{w}$};
 \draw [fill=black] (2,0) circle [radius = .1];
\draw [thick, -latex, gray] (.1,0) -- (1.9, 0) ;
\node [right] at (2,-.1) {.};
\end{tikzpicture}
\end{center}

In this context, consider the underlying graph of $\Gamma$ (as in Remark \ref{rem:underlying_graph}). If this is a tree, it has a flow: each vertex has at most one outgoing edge. This determines a unique root vertex, namely the unique vertex with no outgoing edge. We say that $\Gamma$ is a directed tree.
\end{rem}

The restrictive nature of graphs with a flow  is illustrated by the following:

\begin{prop}
\label{prop:trichotomy_flow}
Suppose that $\Gamma$ is a graph with a flow that is connected.  Then precisely one of the following holds: 
\begin{enumerate}
\item 
$|L_\Gamma^-|= 1$, the underlying graph of  $\Gamma$ is a directed tree and the unique $-$ leg is attached to the root vertex; 
\item 
$L_\Gamma^-=\emptyset$ and the underlying graph of  $\Gamma$ is a directed tree and the root vertex is the unique vertex with no $-$ half edge attached;
\item 
$L_\Gamma^-=\emptyset$ and the underlying graph of $\Gamma$ has genus one, containing a unique directed cycle.
\end{enumerate}
\end{prop}

\begin{proof}
(Sketch.) 
If $\Gamma$ has no edge, then the result is immediate: $\Gamma$ is a corolla and this fits into one of the first two cases, according to whether $L_\Gamma^-$ is empty or not.  

Now suppose that $\Gamma$ has at least one edge. 
Lemma \ref{lem:flow_trichotomy} implies that $\Gamma$ has at most one vertex that has no outgoing edge and any $-$ leg must be attached to this vertex. It follows that $|L_\Gamma^- |\leq 1$. 

If $|L_\Gamma^-|= 1$, one deduces from the Lemma that the underlying graph of $\Gamma$ is a directed tree with the unique $-$ leg  attached to the root vertex. 

If $|L_\Gamma^-|= 0$ and there exists a vertex with no outgoing edge, then again the underlying graph of  $\Gamma$ is a directed tree;

Otherwise, if $|L_\Gamma^-|=0$, there is no vertex with no outgoing edge. Since $V_\Gamma$ is a finite set, it follows that the underlying graph contains a directed cycle of the form 
$$
v= v_1 \elrel v_2 \elrel \ldots \elrel v_t = v
$$
for some $t \in \nat$, in which the $v_i$ (for $1 \leq i <t$) are distinct. Again using Lemma \ref{lem:flow_trichotomy}, one shows that this direct cycle (up to cyclic relabelling) 
 is unique and that the graph has genus one.
\end{proof}

The three possibilities given in  Proposition \ref{prop:trichotomy_flow} are illustrated schematically by the following three examples (without labelling the vertices and the half edges):

\ 

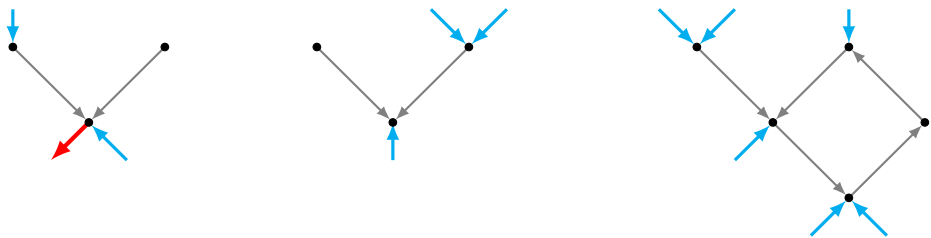
\begin{center}
\begin{tikzpicture}[scale=.4]

\draw [thick, -latex, gray] (0,2) -- (1.93, .07);
\draw [thick, -latex, gray] (4,2) -- (2.07,.07);
\draw [ultra thick, -latex, red] (2,0) -- (1, -1);
\draw [very thick, - latex, cyan] (3,-1) -- (2.07, -.07); 
\draw [very thick, -latex, cyan] (0,3) -- (0, 2.1);
\draw [fill=black] (0,2) circle [radius = .1];
\draw [fill=black] (2,0) circle [radius = .1];
\draw [fill=black] (4,2) circle [radius = .1];

\draw [thick, -latex, gray] (8,2) -- (9.93, .07);
\draw [thick, -latex, gray] (12,2) -- (10.07,.07);
\draw [very thick, -latex, cyan] (13,3) -- (12.07, 2.07); 
\draw [very thick, -latex, cyan] (11,3) -- (11.93, 2.07);
\draw [very thick, -latex, cyan] (10,-1) -- (10, -.01);
\draw [fill=black] (8,2) circle [radius = .1];
\draw [fill=black] (10,0) circle [radius = .1];
\draw [fill=black] (12,2) circle [radius = .1];

\draw [thick, -latex, gray] (18,2) -- (19.93, 0.07);
\draw [very thick, -latex, cyan] (17,3) -- (17.93, 2.07); 
\draw [very thick, -latex, cyan] (19,3) -- (18.07, 2.07);
\draw [thick, -latex, gray] (20,0) -- (21.93, -1.93);
\draw [thick, -latex, gray] (22,-2) -- (23.93, -.07);
\draw [thick, -latex, gray] (24,0) -- (22.07, 1.93); 
\draw [thick, -latex, gray] (22,2) -- (20.07, 0.07);
\draw [very thick, -latex, cyan] (22,3) -- (22, 2.1);
\draw [very thick, -latex, cyan] (21,-3) -- (21.93, -2.07);
\draw [very thick, -latex, cyan] (23,-3) -- (22.07, -2.07);
\draw [very thick, -latex, cyan] (19,-1) -- (19.93, -.07);

\draw [fill=black] (18,2) circle [radius = .1];
\draw [fill=black] (22,2) circle [radius = .1];
\draw [fill=black] (20,0) circle [radius = .1];
\draw [fill=black] (24,0) circle [radius = .1];
\draw [fill=black] (22,-2) circle [radius = .1];

\node at (25,-2) {.};
\end{tikzpicture}
\end{center}

\noindent
Here, $+$ legs are indicated by thick cyan arrows; the unique $-$ leg (in the first case) is indicated by the thick red arrow. 

\begin{rem}
\ 
\begin{enumerate}
\item 
As indicated, in the first two cases, the root node of the underlying directed tree may have $+$ legs attached. 
\item 
In the final case, the vertices that form part of the directed cycle are distinguished. Trees and/ or $+$ legs may be attached to these.
\end{enumerate}
\end{rem}

\section{Koszul complexes from operads}
\label{sect:opd_koszul}

By Theorem \ref{thm:ltfb}, for $\opd$ in $\nuopds$, we have the associated natural surjective morphism of  $\dwbtw$-modules 
\begin{eqnarray}
\label{eqn:opd_dwbtw_surjection}
\ltfb^*\opd \circledcirc \stfb^* |\shift{1}{1}\opd|
\twoheadrightarrow 
\ltfb^*\opd 
.
\end{eqnarray}

The purpose of this section is to apply the Koszul complex constructions of Section \ref{sect:uwb_quadratic_koszul} to this,  deriving (co)homological consequences. In particular, this allows us to introduce the hairy flow-graph complex associated to a non-unital operad, 
counterpart of the hairy graph complexes for cyclic operads (as in \cite{P_cyclic}).

\subsection{Associated Koszul complexes}
\label{subsect:opd_koszul_complexes}

 By the results of Section \ref{sect:uwb_quadratic_koszul}, we have the surjection between  Koszul complexes in $\kk \uwb$-modules:
\begin{eqnarray}
\label{eqn:complex1}
&&
\kzcx \otimes_{\dwbtw} \big( \ltfb^*\opd \circledcirc \stfb^* |\shift{1}{1}\opd|\big)
\twoheadrightarrow 
\kzcx \otimes_{\dwbtw} \ltfb^*\opd 
\\
\label{eqn:complex2}
&&\kk^\uwb \otimes_{\kk \uwb} \kzcx \otimes_{\dwbtw} \big( \ltfb^*\opd \circledcirc \stfb^* |\shift{1}{1}\opd|\big)
\twoheadrightarrow
\kk^\uwb \otimes_{\kk \uwb} \kzcx \otimes_{\dwbtw} \ltfb^*\opd
.
\end{eqnarray}
Moreover, (\ref{eqn:complex1}) is a morphism between complexes in $\uwbup$ and (\ref{eqn:complex2}) is a morphism between complexes in $\uwbdown$. These are related by the functors of Corollary \ref{cor:equivalence_uwbup_uwbdown}.

The underlying objects identify  respectively as:
\begin{eqnarray*}
&&
\kk \uwb \otimes_{\kk (\tfb)} \big( \ltfb^*\opd \circledcirc \stfb^* |\shift{1}{1}\opd|\big)
\twoheadrightarrow 
\kk \uwb \otimes_{\kk (\tfb)}  \ltfb^*\opd
\\
&&
\kk^\uwb \otimes_{\kk (\tfb)}  \big( \ltfb^*\opd \circledcirc \stfb^* |\shift{1}{1}\opd|\big)
\twoheadrightarrow 
\kk^\uwb \otimes_{\kk (\tfb)}  \ltfb^*\opd
,
\end{eqnarray*}
equipped with the usual $\kk\uwb$-module structures. The respective differentials are induced by the differential of $\kzcx$.

By Corollary \ref{cor:cohomology_koszul_complexes} and Proposition \ref{prop:homology_is_Tor}, we have the following interpretation of the (co)homologies of these complexes:

\begin{thm}
\label{thm:opd_ext_tor}
For $\opd$ an operad, 
\begin{enumerate}
\item 
the morphism of graded $\kk \uwb$-modules induced in cohomology by  (\ref{eqn:complex1}) is naturally isomorphic to 
$$
\ext^*_{\dwbtw} (\kk(\tfb), \ltfb^*\opd \circledcirc \stfb^* |\shift{1}{1}\opd|)
\rightarrow 
\ext^*_{\dwbtw} (\kk(\tfb), \ltfb^*\opd) 
$$ 
given by applying $\ext^* _{\dwbtw} (\kk (\tfb), -)$ to 
(\ref{eqn:opd_dwbtw_surjection}); 
\item 
the morphism of graded $\kk \uwb$-modules induced in homology by  (\ref{eqn:complex2}) is naturally isomorphic to 
$$
\tor_*^{\dwbtw} (\kk (\tfb), \ltfb^*\opd \circledcirc \stfb^* |\shift{1}{1}\opd|)
\rightarrow 
\tor_*^{\dwbtw} (\kk (\tfb), \ltfb^*\opd ) 
$$
given by  applying $\tor_* ^{\dwbtw} (\kk (\tfb), -)$ to 
(\ref{eqn:opd_dwbtw_surjection}). 
\end{enumerate}
\end{thm}

\begin{rem}
We could also consider the $\kk \dwb$-module $\stfb^* (\opd \oplus |\shift{1}{1}\opd|)$ and its associated Koszul complexes. This does not yield anything new, since the results are related by the equivalence of categories between $\kk \dwb$-modules and $\dwbtw$-modules in conjunction with Lemma \ref{lem:twist_stfb_operad}. However, when generalized to dioperads (as in \cite{2026arXiv260413750P}), the cases are distinct and must both be considered.
\end{rem}

\subsection{Hairy flow-graph complexes}
\label{subsect:hairy_flow-graph}
Recall that we have the morphism of complexes of $\kk \uwb$-modules
$$
\kk^\uwb \otimes_{\kk (\tfb)}  \big( \ltfb^*\opd \circledcirc \stfb^* |\shift{1}{1}\opd|\big)
\twoheadrightarrow 
\kk^\uwb \otimes_{\kk (\tfb)}  \ltfb^*\opd.
$$

We introduce the following definition, in which  `wheeled' refers to the theory of wheeled operads (see \cite[Section 3.2]{MR4945404} for a presentation of these) and the   {\em wheeled component} corresponds to $|\dbd\opd|$. The terminology {\em flow-graph} is  short for {\em graphs with a flow}, as in Section \ref{subsect:flow-graph}.

\begin{defn}
\label{defn:wheeled_hairy_flow_graph}
For an operad $\opd$,
\begin{enumerate}
\item 
the wheeled hairy flow-graph complex for $\opd$ 
 is the complex $\kk^\uwb \otimes_{\kk (\tfb)}  \big( \ltfb^*\opd \circledcirc \stfb^* |\shift{1}{1}\opd|\big)$;
 \item 
 the hairy flow-graph complex for $\opd$ is the complex $\kk^\uwb \otimes_{\kk (\tfb)}  \ltfb^*\opd$.
\end{enumerate}
\end{defn}

This definition can be justified by comparison with the construction of the (even) hairy graph complex associated to a cyclic operad, as considered in \cite{P_cyclic}, which explains the relationship with the definition of the hairy graph complex as considered by Conant, Kassabov, and Vogtmann \cite{MR3029423}, generalizing the (non-hairy) graph complex of Kontsevich. 

Below, we follow this blueprint for the wheeled case, introducing the modifications required to work in the operad framework as opposed to the cyclic operad framework, using the graphs with a flow introduced in Section \ref{subsect:flow-graph}. In particular, we work with $\kk (\tfb)$-modules and the Day convolution product $\circledcirc$; edges of graphs are encoded using the {\em walled} Brauer category.

To begin, let us fix a quintuple of finite sets $(V, X^+ ,X^- , L^+ , L^-)$  and consider the set of graphs $\Gamma$ such that the underlying quintuple $(V_\Gamma, X_\Gamma^+ ,X_\Gamma^- , L_\Gamma^+ , L_\Gamma^-)$ identifies with this. By Lemma \ref{lem:graph_finiteness}, the set of graphs $\Gamma$ on this quintuple is 
$$
\uwb ((L^+, L^-), (X^+, X^-)) \times \fs (X^+ \amalg X^- , V).
$$
The set of graphs with a flow (or flow-graphs) is given by the subset of pairs $(f, p)$ such that, for each $v \in V$, $|X^-  \cap p^{-1} (v)|\leq 1$ (this condition only depends upon $p$). We partition $V$ as $V = V_1 \amalg V_0$, where $V_1 = \{ v \mid |X^-  \cap p^{-1} (v)|=1 \}$ and $V_0 = \{ v \mid |X^-  \cap p^{-1} (v)|=0 \}$. We also decompose $p^{-1}(v) = p^{-1}_+ (v)\amalg p^{-1}_-(v)$, using the partition $X = X^+ \amalg X^-$ of the half edges. (Thus $p^{-1}_-(v)$ has cardinality at most $1$, by the flow-graph hypothesis.)

Now, given a flow-graph $(f,p)$ as above and an operad $\opd$, we may form 
$$
\bigotimes _{v \in V_1} \opd (p^{-1}_+(v), p^{-1}_-(v))
\otimes 
\bigotimes_{v \in V_0} |\dbd \opd | (p^{-1}_+ (v), \mathbf{0}).
$$
(This only depends on $p$.)

Summing over all $p \in  \fs^{\mathrm{flow}}  (X^+ \amalg X^- , V)$ (where the suffix indicates the subset of $p$ that satisfy the flow-graph condition), we have the identification 
$$
\bigoplus_{p \in \fs ^\mathrm{flow} (X^+ \amalg X^- , V)} \big( \bigotimes _{v \in V_1} \opd (p^{-1}_+(v), p^{-1}_-(v))
\otimes 
\bigotimes_{v \in V_0} |\dbd \opd | (p^{-1}_+ (v), \mathbf{0}) \Big) 
\cong 
(\opd \oplus |\dbd \opd |) ^{\circledcirc V}(X^+,X^-), 
$$
by construction of the convolution product $\circledcirc$. (Note that we are labelling the terms by $V$, so that we have retained an order.) The flow condition is now subsumed in the restrictive form of the $\kk (\tfb)$-module $
\opd \oplus |\dbd \opd |$, which is supported on arities of the form $(\n, \mathbf{1})$ or $(\n , \mathbf{0})$.

Then, summing over all possible flow graphs on $(V, X^+ ,X^- , L^+ , L^-)$, this yields: 
$$
\kk \uwb ((L^+, L^-), (X^+, X^-)) \otimes _\kk (\opd \oplus |\dbd \opd |) ^{\circledcirc V}(X^+, X^-).
$$
Using the isomorphism  $\kk \uwb ((L^+, L^-), (X^+, X^-))\cong \kk ^{\uwb ((L^+, L^-), (X^+, X^-))}$ given by the dual basis (of  the canonical basis given by $\uwb ((L^+, L^-), (X^+, X^-))$), we can rewrite this as:
$$
\kk ^{\uwb ((L^+, L^-), (X^+, X^-))} \otimes _\kk (\opd \oplus |\dbd \opd |) ^{\circledcirc V}(X^+, X^-).
$$
 
We next introduce the orientation signs: namely there is an orientation sign associated to the order of the $V^+$ vertices. This corresponds to twisting the $\kk (\tfb)$ module $(\opd \oplus |\dbd \opd |) ^{\circledcirc V}$ by $( \triv \boxtimes \sgn)$ (compare Lemma \ref{lem:twist_stfb_operad}).

We now  `forget' the labelling of the vertices (by passing to the quotient by the action of $\aut (V)$) and then take into account the isomorphisms between graphs. We shall also allow $V$ and $(X^+, X^-)$ to vary (these finite sets are chosen in the skeleton of $\fin$); we keep the hairs $(L^+,L^-)$ fixed.

When we pass to the quotient by the action of $\aut (V)$ and sum over all possible $V$ in the skeleton of $\fin$,  we obtain:
$$
\kk ^{\uwb ((L^+, L^-), (X^+, X^-))} \otimes _\kk \big( \ltfb^* (\opd) \circledcirc \stfb^* |\dbd \opd | \big)(X^ +, X^-).
$$
Here, summing over all possible $V$ is simply to allow us to write the term on the right of the $\otimes$ without condition on the $*$'s. (For the identification, we use Lemma \ref{lem:twist_stfb_operad}.) 

Finally, we sum over all possible $(X^+, X^-)$ and  take the quotient by the action of the groupoid of isomorphisms between flow-graphs. For given $(X^+, X^-)$, this corresponds to forming the coinvariants for the action of $\aut ((X^+, X^-))$. This gives: 
$$
\kk^{\uwb((L^+, L^-),-)}
\otimes_{\kk (\tfb)}  
\big( \ltfb^* (\opd) \circledcirc \stfb^* |\dbd \opd | \big)
.
$$

By {\em definition} (by analogy with the hairy graph complex in the cyclic operad case, as constructed in \cite{MR3029423}), this  gives the underlying object of the hairy flow-graph complex (with wheels corresponding to the term $|\dbd \opd|$) for the operad $\opd$, with legs $(L^+, L^-)$.

The differential of this wheeled hairy flow-graph complex is given by edge contraction, using the structure of the operad $\opd$, as encoded in the $\dwbtw$-module structure of $\ltfb^* (\opd) \circledcirc \stfb^* |\dbd \opd |$. One checks that this is precisely the differential on 
$$
\kk ^{\uwb} \otimes _{\kk \uwb} \kzcx \otimes_{\dwbtw} \big( \ltfb^* (\opd) \circledcirc \stfb^* |\dbd \opd | \big).
$$
 
This completes our justification of Definition \ref{defn:wheeled_hairy_flow_graph}.

\begin{rem}
There is an alternative justification, by comparing with the results of Dotsenko in \cite[Section 4]{MR4945404}, in particular his Theorem \cite[Theorem 4.15]{MR4945404} (corresponding to the complex defined using $ \ltfb^* (\opd) \circledcirc \stfb^* |\dbd \opd |$) and Theorem \cite[Theorem 4.11]{MR4945404} (corresponding to the complex defined using $ \ltfb^* (\opd) $).

Dotsenko expresses his results in terms of the (wheeled) bar construction for (wheeled) operads (and wheeled coPROP completion). One could consider that these constructions are the {\em natural} approach to defining graph complexes in this context, by comparison with the relationship between hairy graph complexes and the Feynman transform for modular operads (restricted to the case of cyclic operads) \cite{MR1601666}.

The comparison between our results and Dotsenko's results is given in Section \ref{subsect:compare_Dotsenko}.
\end{rem}

\section{From operads to complexes of $\spmon$-modules}
\label{sect:spmon_operads}

The purpose of this section is to use the mixed tensor functors to analyse the Koszul complex associated to the $\dwbtw$-module of Theorem \ref{thm:ltfb}, using the general procedure proposed in Section \ref{sect:schur_koszul}. Namely, for $\opd $ in $\nuopds$,  we have the morphism between Koszul complexes of $\kk\uwb$-modules
\begin{eqnarray}
\label{eqn:mor_kzcx_opd}
\kzcx \otimes_{\dwbtw} \big( \ltfb^*\opd \circledcirc \stfb^* |\shift{1}{1}\opd|\big)
\twoheadrightarrow 
\kzcx \otimes_{\dwbtw}  \ltfb^*\opd
.
\end{eqnarray}
On applying the functor $T^{\bullet , \bullet} \otimes _{\kk \uwb} -$ (as in Section \ref{sect:mixed_tensors}) we obtain the morphism of complexes  
\begin{eqnarray}
\label{eqn:T_kzcx_opd}
\xymatrix{
T^{\bullet , \bullet} \otimes _{\kk \uwb}\kzcx \otimes_{\dwbtw} \big( \ltfb^*\opd \circledcirc \stfb^* |\shift{1}{1}\opd|\big)
\ar[d]_\cong 
\ar[r]
&
T^{\bullet , \bullet} \otimes _{\kk \uwb}\kzcx \otimes_{\dwbtw}  \ltfb^*\opd 
\ar[d]^\cong 
\\
T^{\bullet , \bullet} \otimes _{\kk (\tfb)} \big( \ltfb^*\opd \circledcirc \stfb^* |\shift{1}{1}\opd|\big)
\ar[r]
&
T^{\bullet , \bullet} \otimes _{\kk (\tfb)} \ltfb^*\opd 
}
\end{eqnarray}
in $\f(\spmon)$, where the complexes in the bottom row are   equipped with the differential induced by that of $\kzcx$. We explain the relationship of this with the work of Dotsenko \cite{MR4945404}.

\begin{rem}
\ 
\begin{enumerate}
\item 
Dotsenko's work uses stable $\mathfrak{gl}(V)$-invariants (which is equivalent to the stable $\GL (V)$-invariants considered here). The analysis of the relationship between our methods and those of Dotsenko's uses the material of Section \ref{sect:gl_inv}; this is outlined in  Section \ref{subsect:compare_Dotsenko}.
\item 
The usage of  invariant theory (as above) to study  operadic and more general structures (such as wheeled PROPs) has been employed by other authors. See for example the work of Derksen and Makam \cite{MR4568126}.
\end{enumerate}
\end{rem}

We fix an operad $\opd$ in $\kk$-vector spaces over a field $\kk$ of characteristic zero; the structures considered are natural with respect to $\opd$. Recall that we do not require that the operad should have a unit. 

\subsection{Algebraic structures associated to $\opd$}

As in  Section \ref{sect:operads}, we work with the associated $\kk (\tfb)$-modules $\opd$ and $|\dbd \opd|$. We can thus apply the functor $T^{\bullet , \bullet } \otimes_{\kk (\tfb)} -$ (which is exact since we are working in characteristic zero). The associated functors on $\spmon$ are 
\begin{eqnarray*}
V & \mapsto & \der (\opd (V))
\\
V & \mapsto & |\partial \opd (V)|,
\end{eqnarray*}
where $\der (\opd(V))$ is the space of derivations of the free $\opd$-algebra on $V$, which identifies as $\hom_\kk (V, \opd (V)) \cong V^\sharp \otimes \opd (V)$, and 
the notation $\partial \opd$ is as in \cite[Section 2.1]{MR4945404}. 

The structure of the operad (as considered in Section \ref{sect:operads}) gives more structure (see \cite[Sections 2 and 3]{MR4945404}):
\begin{enumerate}
\item 
$\der (\opd (V))$ is naturally a Lie algebra; 
\item 
$|\partial \opd (V)|$ is naturally a $\der (\opd (V))$-module; 
\item 
the divergence $\mathrm{div} : \der (\opd (V)) \rightarrow |\partial \opd (V)|$, which   is a Lie $1$-cocycle.
 \end{enumerate}
 
 \begin{rem}
 \label{rem:derive_structure_morphisms}
 From the current viewpoint, the $\der(\opd (V))$-module structure on $|\partial \opd (V)|$ arises as follows. We have the structure map in $\kk (\tfb)$-modules
 $$
 \alpha : \shift{1}{0} (|\shift{1}{1} \opd |) \circledcirc \shift{0}{1} \opd \rightarrow |\shift{1}{1}\opd|.
 $$
 Precomposing with the projection $\shift{1}{1} (|\shift{1}{1} \opd |\circledcirc \opd) \twoheadrightarrow \shift{1}{0} (|\shift{1}{1} \opd |) \circledcirc \shift{0}{1} \opd $, this gives the structure map in $\kk (\tfb)$-modules
 $$
 \shift{1}{1} (|\shift{1}{1} \opd |\circledcirc \opd)\rightarrow |\shift{1}{1}\opd|.
 $$
 By Proposition \ref{prop:adjoints_to_shift}, this has mate in $\kk (\tfb)$-modules:
  $$
 |\shift{1}{1} \opd |\circledcirc \opd
 \rightarrow |\shift{1}{1}\opd|
 \circledcirc (\triv_1\boxtimes \triv_1).
 $$
Applying the functor $T^{\bullet , \bullet} (V) \otimes_{\kk (\tfb)} -$, this gives the natural (with respect to $V$ in $\spmon$) map
$$
|\delta \opd (V) | \otimes \der (\opd (V))
\rightarrow 
|\delta \opd (V) | \otimes (V \otimes V^\sharp).
$$
Composing with the morphism induced by the contraction $V \otimes V^\sharp \rightarrow \kk$, this yields the $\der (\opd (V))$-module structure on $|\delta \opd (V) | $. 

For the divergence, one starts from the projection $\pi : \shift{1}{1} \opd \rightarrow |\shift{1}{1}\opd|$, viewed as a morphism of $\kk(\tfb)$-modules. This has mate 
$\opd \rightarrow |\shift{1}{1}\opd| \circledcirc (\triv_1\boxtimes \triv_1)$; then, proceeding as above, one obtains the composite
$$
\der (\opd (V)) \rightarrow |\delta \opd (V) | \otimes (V \otimes V^\sharp)
\rightarrow 
|\delta \opd (V) |.
$$
 This is the divergence. (The latter was introduced in \cite{MR4881588} by {\em ad hoc} methods; the above presentation is adapted from that of \cite{MR4945404}.)
 \end{rem} 
 
As in \cite[Section 4.4]{MR4945404}, this structure yields a differential graded Lie algebra structure on 
 \begin{eqnarray}
 \label{eqn:DGLie}
 \der (\opd (V)) \oplus s^{-1} |\partial \opd (V)|,
 \end{eqnarray}
 where the underlying Lie algebra structure is given by the semi-direct product structure (treating the desuspension $s^{-1} |\partial \opd (V)|$ as a $\der (\opd(V))$-module), and the differential is induced by $\mathrm{div}$.
 
 \begin{rem}
 \label{rem:CE_dg_Lie_alg}
 \ 
 \begin{enumerate}
 \item 
By Theorem \ref{thm:ltfb_identify_with_CE} below, this differential graded Lie algebra  structure encodes the additional structure given by Theorem \ref{thm:ltfb}.
 \item 
 \label{item:ltimes}
 Dotsenko denotes the above DG Lie algebra by $\der (\opd (V)) \ltimes_{\mathrm{div}} |\partial \opd (V)|$. The above notation has been used so as to keep track  of the desuspension.
 \end{enumerate}
 \end{rem}

 \subsection{Analysing the structure of Theorem \ref{thm:ltfb}}
The DG Lie algebra (\ref{eqn:DGLie}) is equipped with the morphism of DG Lie algebras
  \begin{eqnarray}
 \label{eqn:mor_DGLie}
 \der (\opd (V)) \oplus s^{-1} |\partial \opd (V)|
 \rightarrow 
 \der (\opd (V))
 \end{eqnarray}
 sending $s^{-1} |\partial \opd (V)|$ to zero. (The differential on the codomain is zero.)
 
 One can form the Chevalley-Eilenberg complex of the DG Lie algebras appearing in (\ref{eqn:mor_DGLie});  the respective underlying graded objects are given by  the symmetric algebra (defined using Koszul signs) on the suspension of the  underlying object. This yields the morphism of complexes
\[
S^* ( s \der (\opd (V)) \oplus  |\partial \opd (V)|)
\twoheadrightarrow 
S^* ( s \der (\opd (V))),
\]
and this is functorial with respect to $V$ in $\spmon$. As usual, this can be rewritten  as
\begin{eqnarray}
\label{eqn:map_CE_complexes}
\Lambda^* ( \der (\opd (V))) \otimes S^* ( |\partial \opd (V)|)
\twoheadrightarrow 
\Lambda^* ( \der (\opd (V))),
\end{eqnarray}
where the homological degree is given by the degree in the exterior algebra $\Lambda^*$. The differential on these Chevalley-Eilenberg complexes can then be described explicitly in terms of 
\begin{enumerate}
\item 
the Lie algebra structure of $\der (\opd (V))$; 
\item 
the module structure of $|\partial \opd (V)|$;
\item 
the divergence $\mathrm{div}$.
\end{enumerate}
All the above is natural with respect to $V$ in $\spmon$.

The following result shows that the morphism of complexes (\ref{eqn:T_kzcx_opd}) derived from the $\dwbtw$-module structure on $\ltfb^*\opd \circledcirc \stfb^* |\shift{1}{1}\opd|$ is a model for the morphism of Chevalley-Eilenberg complexes.

\begin{thm}
\label{thm:ltfb_identify_with_CE}
For $\opd$ in $\nuopds$, the morphism of complexes
$$
T^{\bullet , \bullet} \otimes _{\kk (\tfb)} \big( \ltfb^*\opd \circledcirc \stfb^* |\shift{1}{1}\opd|\big)
\twoheadrightarrow
T^{\bullet , \bullet} \otimes _{\kk (\tfb)} \ltfb^*\opd 
$$
given in (\ref{eqn:T_kzcx_opd})  is naturally isomorphic to the morphism between Chevalley-Eilenberg complexes associated to  the morphism of DG Lie algebras (\ref{eqn:mor_DGLie}) 
\[
\Lambda^* ( \der (\opd (V))) \otimes S^* ( |\partial \opd (V)|)
\twoheadrightarrow 
\Lambda^* ( \der (\opd (V)))),
\]
considered as a functor of $V$ in $\spmon$.

Hence, on passing to homology, there is a commutative diagram of graded algebraic functors on $\spmon$ 
$$
\xymatrix{
H_* \Big(T^{\bullet , \bullet} \otimes _{\kk (\tfb)} \big( \ltfb^*\opd \circledcirc \stfb^* |\shift{1}{1}\opd|\big)\Big)
\ar[r]
\ar[d]_\cong 
&
H_* \big(T^{\bullet , \bullet} \otimes _{\kk (\tfb)}  \ltfb^*\opd \big)
\ar[d]^\cong 
\\
H^\mathsf{CE}_* ( \der (\opd (-)) \oplus s^{-1} |\partial \opd (-)|)
\ar[r]
&
H^\mathsf{CE}_* ( \der (\opd (-))).
}
$$
\end{thm}

\begin{proof}
We first check the isomorphism at the level of the underlying graded objects. This follows from the fact that $T^{\bullet, \bullet} \otimes_{\kk(\tfb)} -$ is symmetric monoidal with respect to the Day convolution product $\circledcirc$. When working with $\kk \fb$-modules and the convolution product $\odot$, this is a  standard property of the Schur functor construction $T^* \otimes_\fb -$. For the bimodule case, one can reduce to considering the case of  $(M \boxtimes N) \circledcirc (P \boxtimes Q)$ and thereby to the case of $\odot$.

It remains to check that the differentials agree across the above isomorphism. This is an exercice in unravelling the definitions. The first step is  identifying how the 
structure maps 
\begin{eqnarray*}
\pc &: &
\shift{1}{0} \opd \circledcirc \shift{0}{1} \opd 
\rightarrow 
\opd
\\
\alpha &:&\shift{1}{0}(|\dbd \opd|) \circledcirc \shift{0}{1} \opd
\rightarrow 
|\dbd \opd |
\\
\pi &:&\dbd \opd \twoheadrightarrow |\dbd \opd|
\end{eqnarray*}
induce the Lie algebra structure on $\der (\opd (V))$, the $\der (\opd (V))$-module structure on $|\partial \opd (V)|$, and the divergence $\mathrm{div}$. (See the discussion in Remark \ref{rem:derive_structure_morphisms}.)

Then the explicit description of the $\dwbtw$-module structure in Theorem \ref{thm:ltfb} makes it clear that this is obtained from the relevant data in exactly the same way in which the Chevalley-Eilenberg differential is obtained. 

The statement about the homology follows immediately.
\end{proof}

By Theorem \ref{thm:opd_ext_tor}, there are identifications (using  appropriate gradings) 
\begin{eqnarray*}
H_* \Big( \kzcx \otimes_{\dwbtw} \big( \ltfb^*\opd \circledcirc \stfb^* |\shift{1}{1}\opd| \big)\Big) 
& \cong &
\ext^* _{\dwbtw} (\kk (\tfb), \ltfb^*\opd \circledcirc \stfb^* |\shift{1}{1}\opd| ) 
\\
H_* \big( \kzcx \otimes_{\dwbtw} \ltfb^*\opd \circledcirc \stfb^* |\shift{1}{1}\opd| \big)
& \cong &
\ext^* _{\dwbtw} (\kk (\tfb), \ltfb^*\opd  ). 
\end{eqnarray*}
Moreover, these are isomorphisms of (graded) $\kk \uwb$-modules. These are compatible with the respective maps induced by the morphism of $\dwbtw$-modules:
\begin{eqnarray}
\label{eqn:mor_dwbtw_modules}
\ltfb^*\opd \circledcirc \stfb^* |\shift{1}{1}\opd|
\twoheadrightarrow 
 \ltfb^*\opd
\end{eqnarray}
of Theorem \ref{thm:ltfb}.

Now, by Proposition \ref{prop:Kunneth_T_Ext}, there is a morphism of universal coefficients spectral sequences relating the induced morphism between $\ext^*$ groups  to the morphism in homology given by Theorem \ref{thm:ltfb_identify_with_CE}. Putting these facts together yields:

\begin{thm}
\label{thm:Kunneth_ss_operadic_case}
There are natural  spectral sequences that are functorial with respect to $\spmon$:
\begin{eqnarray*}
\tor_*^{\kk \uwb} (T^{\bullet, \bullet} , \ext^* _{\dwbtw} (\kk (\tfb), \ltfb^*\opd \circledcirc \stfb^* |\shift{1}{1}\opd| )  ) 
&\Rightarrow &
H_*^\mathsf{CE} (\der (\opd (-)) \oplus s^{-1} |\partial \opd (-)| ) 
\\
\tor_*^{\kk \uwb} (T^{\bullet, \bullet} , \ext^* _{\dwbtw} (\kk (\tfb), \ltfb^*\opd   ) 
&\Rightarrow &
H_*^\mathsf{CE} (\der (\opd (-)) ) 
\end{eqnarray*}
together with a morphism of spectral sequences between these induced by (\ref{eqn:mor_dwbtw_modules}).

In particular, evaluating on $V$ (considered as an object of $\spmon$), there is a commutative diagram corresponding to the edge homomorphisms:
$$
\xymatrix{
T^{\bullet, \bullet} (V) \otimes_{\kk \uwb} \ext^* _{\dwbtw} (\kk (\tfb), \ltfb^*\opd \circledcirc \stfb^* |\shift{1}{1}\opd| )
\ar[r]
\ar[d]
&
H^\mathsf{CE}_* ( \der (\opd (V)) \oplus s^{-1} |\partial \opd (V)|)
\ar[d]
\\
T^{\bullet, \bullet} (V) \otimes_{\kk \uwb} \ext^* _{\dwbtw} (\kk (\tfb), \ltfb^*\opd )
\ar[r]
&
H^\mathsf{CE}_* ( \der (\opd (V))).
}
$$
\end{thm}

\begin{rem}
A subtlety has been glossed over in the above statement: one has to choose  appropriate compatible gradings.
\end{rem}

\subsection{Weak stabilization}

By Theorem \ref{thm:ltfb_identify_with_CE}, 
the morphism
\begin{eqnarray}
\label{eqn:mor_HCE}
H^\mathsf{CE}_* ( \der (\opd (-)) \oplus s^{-1} |\partial \opd (-)|)
\longrightarrow 
H^\mathsf{CE}_* ( \der (\opd (-)))
\end{eqnarray}
between algebraic functors on $\spmon$ identifies as:
$$
H_* \big(T^{\bullet , \bullet} \otimes _{\kk (\tfb)} \big( \ltfb^*\opd \circledcirc \stfb^* |\shift{1}{1}\opd|\big)\big)
\longrightarrow 
H_* \big(T^{\bullet , \bullet} \otimes _{\kk (\tfb)}  \ltfb^*\opd \big).
$$
Both of these are induced by the morphism $\ltfb^*\opd \circledcirc \stfb^* |\shift{1}{1}\opd| \rightarrow \ltfb^*\opd $ of $\dwbtw$-modules given by Theorem \ref{thm:ltfb}.

We may therefore apply Proposition \ref{prop:weak_stabilization_homology} to deduce the following result on the weak stabilization of the Chevalley-Eilenberg homology:

\begin{thm}
\label{thm:weak_stabilization_CE}
For $\opd$ in $\nuopds$, the weak stabilization of (\ref{eqn:mor_HCE}) in $\kk \dwb$-modules:
$$
\hom_{\spmon} (H^\mathsf{CE}_* ( \der (\opd (-))), T^{\ast, \ast}) 
\longrightarrow 
\hom_{\spmon }  (H^\mathsf{CE}_* ( \der (\opd (-)) \oplus s^{-1} |\partial \opd (-)|), T^{\ast, \ast})
$$
identifies with the $\kk$-linear dual of the morphism in homology 
\begin{eqnarray}
\label{eqn:mor_graph_homology}
H_* \big(\kk^\uwb \otimes_{\kk (\tfb)} ( \ltfb^*\opd \circledcirc \stfb^* |\shift{1}{1}\opd|)\big) 
\rightarrow 
H_* (\kk^\uwb \otimes_{\kk (\tfb)} \ltfb^*\opd )
\end{eqnarray}
given by the morphism of complexes (\ref{eqn:complex2}). 

The latter identifies with the morphism
$$
\tor_*^{\dwbtw} (\kk (\tfb), \ltfb^*\opd \circledcirc \stfb^* |\shift{1}{1}\opd|) 
\longrightarrow     
  \tor_*^{\dwbtw} (\kk (\tfb), \ltfb^*\opd )
$$
induced by (\ref{eqn:mor_dwbtw_modules}).
\end{thm}

\begin{proof}
The first statement follows from Proposition \ref{prop:weak_stabilization_homology}, as indicated by the discussion before the statement. The second statement is also included in 
Proposition \ref{prop:weak_stabilization_homology}; it is a consequence of Proposition \ref{prop:homology_is_Tor}.
\end{proof}

\begin{rem}
Using the terminology introduced in Definition \ref{defn:wheeled_hairy_flow_graph}, 
(\ref{eqn:mor_graph_homology}) should be understood as the comparison map 
from the wheeled hairy flow-graph homology of $\opd$ to the hairy flow-graph homology of $\opd$.
\end{rem}

\begin{rem}
\label{rem:ext_to_tor}
As observed in Section \ref{subsect:UCT_Koszul}, there is also a universal coefficients spectral sequence relating $\ext^*_{\dwbtw}$ and $\tor_*^{\dwbtw}$. This gives analogues of the spectral sequences of Theorem \ref{thm:Kunneth_ss_operadic_case}:
\begin{eqnarray*}
\tor_*^{\kk \uwb} (\kk ^\uwb , \ext^* _{\dwbtw} (\kk (\tfb), \ltfb^*\opd \circledcirc \stfb^* |\shift{1}{1}\opd| )  ) 
&\Rightarrow &
\tor_*^{\dwbtw} (\kk (\tfb), \ltfb^*\opd \circledcirc \stfb^* |\shift{1}{1}\opd|) 
\\
\tor_*^{\kk \uwb} (\kk^\uwb, \ext^* _{\dwbtw} (\kk (\tfb), \ltfb^*\opd   ) )
&\Rightarrow &
 \tor_*^{\dwbtw} (\kk (\tfb), \ltfb^*\opd ),
\end{eqnarray*}
which are spectral sequences in the category of $\kk \uwb$-modules. There is a morphism of spectral sequences between these induced by (\ref{eqn:mor_dwbtw_modules}).

More precisely, the $\kk$-linear duals of these spectral sequences identify with the spectral sequence obtained by applying the exact functor $\hom_{\spmon} (-, T^{\ast,\ast})$ to the spectral sequences of Theorem \ref{thm:Kunneth_ss_operadic_case}. Clearly it is preferable to avoid the vector space duality by considering the above spectral sequences directly. 
\end{rem}

\subsection{Comparison with Dotsenko's results}
\label{subsect:compare_Dotsenko}

Dotsenko's \cite[Theorem 4.15]{MR4945404} concerns the invariants of the Chevalley-Eilenberg complex: 
$$
C_*^{\mathsf{CE}}( \der^+ (\opd (V)) \ltimes_{\mathrm{div}} | \overline{\partial (\opd) (V)}| , T^{\ast, \ast}(V)), 
$$
written using our notation for the mixed tensors; the DG Lie algebra notation is derived from that of Remark \ref{rem:CE_dg_Lie_alg} (\ref{item:ltimes}). 

The {\em positive} derivations $\der^+ (\opd (V))$ and the reduced expression $| \overline{\partial (\opd) (V)}|$ can be understood in terms of the material of the previous section by restricting to the suboperad  $\opd_{>1}$ (without unit) of $\opd$ supported on arities $\n$ with $n>1$. In particular, the action of this DG Lie algebra on the coefficients $T^{\ast, \ast}(V)$ is trivial.

Theorem \ref{thm:ltfb_identify_with_CE} identifies the above Chevalley-Eilenberg complex with 
$$
\big(T^{\bullet, \bullet} \otimes_{\kk (\tfb)} (\ltfb^* (\opd_{>1}) \circledcirc \stfb^* |\dbd\opd_{>1}|)\big) \otimes T^{\ast, \ast}, 
$$
considered as a functor on $\spmon$. (Here $T^{\ast, \ast}$ corresponds to the coefficients.)

Dotsenko considers (stable) $\mathfrak{gl} (V)$-invariants, which is equivalent to considering (stable) $\GL (V)$-invariants. Here the formation of the stable $\GL (V)$-invariants is understood as  in Section \ref{subsect:apply_GL_invariants}.  Namely, the relationship with the weak stabilization as considered here is explained by Proposition \ref{prop:stable_inv_uwbup}. The latter is stated for  {\em traceless} mixed tensors $\traceless{\ast, \ast}$ in place of $T^{\ast, \ast}$. The mixed tensor case can be derived from this, as in Corollary \ref{cor:traceless_to_T}. 

So as to apply Proposition \ref{prop:stable_inv_uwbup} directly and give the cleanest statement, we replace $T^{\ast, \ast}$ by $\tracelessop{\ast,\ast}$ (introduced in Notation \ref{nota:tracelessop}) in the following:

\begin{thm}
\label{thm:stable_GL_inv_CE}
The stable $\GL$-invariants of 
$$
C_*^{\mathsf{CE}}( \der^+ (\opd (V)) \ltimes_{\mathrm{div}} | \overline{\partial (\opd) (V)}| , \tracelessop{\ast, \ast}(V)), 
$$
are naturally isomorphic to the complex 
$$
\kk^\uwb \otimes_{\kk \uwb} \kzcx \otimes_{\dwbtw}  (\ltfb^* (\opd_{>1}) \circledcirc \stfb^* |\dbd\opd_{>1}|)
\cong 
\kk^\uwb \otimes_{\kk \tfb} (\ltfb^* (\opd_{>1}) \circledcirc \stfb^* |\dbd\opd_{>1}|)
$$
considered as a complex of $\kk (\tfb)$-modules.
\end{thm}

\begin{proof}
Using the discussion preceding the statement, the Chevalley-Eilenberg complex is isomorphic to  
$$
\big( 
T^{\bullet, \bullet} \otimes_{\kk (\tfb)} (\ltfb^* (\opd_{>1}) \circledcirc \stfb^* |\dbd\opd_{>1}|)\big) \otimes \tracelessop{\ast, \ast}.
$$
The result follows by applying  Proposition \ref{prop:stable_inv_uwbup} to the latter. 
\end{proof}

Applying Corollary \ref{cor:traceless_to_T}, this yields:

\begin{cor}
\label{cor:stable_GL_inv_CE}
The stable $\GL$-invariants of 
$$
C_*^{\mathsf{CE}}( \der^+ (\opd (V)) \ltimes_{\mathrm{div}} | \overline{\partial (\opd) (V)}| , T^{b,a}(V)), 
$$
are naturally isomorphic to the complex of $\kk (\sym_a \times \sym_b)$-modules 
$$
\kk \uwb (-, (\mathbf{a}, \mathbf{b})) \otimes_{\kk (\tfb)} \kk^\uwb \otimes_{\kk \tfb} (\ltfb^* (\opd_{>1}) \circledcirc \stfb^* |\dbd\opd_{>1}|)
$$
\end{cor}

\begin{rem}
\label{rem:add_identities}
Recall that, for $m,n, a, b \in \nat$ such that $a-m = b-n = r \in \nat$, the $\kk (\sym_m \times \sym_n)\op$-module $\kk \uwb ((\mathbf{m}, \mathbf{n}), (\mathbf{a}, \mathbf{b}))$ has basis given by $\wpair_r (\mathbf{a}, \mathbf{b})$, by Proposition \ref{prop:uwb_generators}. 

Thus, the effect of passing from $T^{b,a}$ (as in Theorem \ref{thm:stable_GL_inv_CE}) to $T^{b,a}$ (as in Corollary \ref{cor:stable_GL_inv_CE}) is to add the possibility of having hairs corresponding to pairings $\wpair_r (\mathbf{a}, \mathbf{b})$, for varying $r$. These can be seen as `adding identity elements' to the hairy flow-graph complex. 
\end{rem}

By Remark \ref{rem:add_identities}, Corollary \ref{cor:stable_GL_inv_CE} together with  Dotsenko's \cite[Theorem 4.15]{MR4945404} show that the extension of the wheeled hairy flow-graph complex by `adding identity elements' is isomorphic to the coPROP completion of the wheeled bar construction $B^\circlearrowright (\opd ^\circlearrowright)$, where $\opd^{\circlearrowright}$ is the wheeled completion of $\opd$  (see {\em loc. cit.} for details on the latter). 

\appendix
\section{Stabilization using $\GL$-invariants}
\label{sect:gl_inv}

In this section, we explain how to extract `stable' information from algebraic functors on $\spmon$ by using invariant theory.  Throughout, $\kk$ is a field of characteristic zero.

 In Section \ref{sect:mixed_tensors}, we explained the approach using weak stabilization (see Remark \ref{rem:weak_stabilization}), based on the functors $\hom_{\spmon} (-, T^{\ast, \ast})$ restricted to $\falg (\spmon)$. An alternative approach is based upon the idea of evaluating an algebraic functor $F$ on `sufficiently large' $V$ and then calculating the invariants $F(V)^{\GL (V)}$. However, this essentially only sees the weak stabilization evaluated on $(\ast, \ast) = (\mathbf{0}, \mathbf{0})$ and thus, in general, loses most of the information. In the following, we assume that, not only $F$ is algebraic, but it is also of finite length in $\falg (\spmon)/ \falg_\tors (\spmon)$.

A solution is to tensor $F$ with $T^{b,a}$ (letting $a, b \in \nat$ vary)  - the reason for the choice of ordering of the indices in $T^{b,a}$  will become apparent later - and {\em then} calculate the invariants:
$$
\big( T^{b,a} (V) \otimes F \big) ^{\GL (V)}
$$
for sufficiently large $V$. (This approach was essentially taken by Dotsenko in \cite{MR4945404} by introducing `coefficients' into the Lie algebra homology that he considers.)

Now, by the `decomposition' of $T^{b,a}$ that is given by equation (\ref{eqn:decompose_Tab}), we can refine further by restricting to traceless mixed tensors, considering
$$
\big( \traceless{b,a} (V) \otimes F \big) ^{\GL (V)}
$$
for sufficiently large $V$. (For this reduction we use the hypothesis that $F$ is algebraic of finite length modulo torsion and the fact that we are stabilizing by taking $V$ sufficiently large.) 

The purpose of this section is to explain how this functor behaves, starting from the basic case $F = T^{k,l}$.

\subsection{Extracting information from mixed tensors}

We  extract information from the mixed tensors and functors constructed from them by using the first fundamental theorem of invariant theory. This implies that, for $\dim V$  sufficiently large, the $\GL(V)$-invariants of $T^{m,n} (V)$ identify as 
\begin{eqnarray}
\label{eqn:stable_inv_T}
(T^{m,n}) ^{\GL(V)} 
\cong 
\left\{
\begin{array}{ll}
\kk \fb (\n, \m) \cong \kk \sym_m & m=n \\
0 & m \neq n.
\end{array}
\right.
\end{eqnarray}

\begin{rem}
For $m=n$, the isomorphism  $(T^{m,n}) ^{\GL(V)} 
\cong \kk \sym_m$ is  one of $\kk (\sym_m \times \sym_m)$-modules, using the left and right regular actions on $\kk \sym_m$ (with appropriately adjusted variance).
\end{rem}

For the applications, we will require a refinement of this that will allow us to take into account the `legs' (aka. `hairs')  of graphs (see Section \ref{sect:graphs} for the sort of graphs that we consider). This is achieved by exploiting the traceless tensors, as follows.

\begin{lem}
\label{lem:GL_invariants}
For $a,b, k,l \in \nat$ and $V$ a finite-dimensional  $\kk$-vector space with $\dim V$ sufficiently large (compared to $a,b, k,l$), there is an isomorphism of $\kk$-vector spaces:
\[
(\traceless{b,a}(V)  \otimes T^{k,l} (V) )^{\GL (V)}
\cong 
\kk \uwb ((\mathbf{a}, \mathbf{b}) , (\mathbf{k},\mathbf{l})).
\]
In particular, this is zero unless $k-a = l -b \geq 0$.  
\end{lem}

\begin{proof}
Using  Definition \ref{defn:traceless} (with $a$ and $b$ transposed) and forming the tensor product with $T^{k,l}$, there is an exact sequence in $\f (\spmon)$:
$$
0
\rightarrow 
\traceless{b,a}  \otimes T^{k,l}
\rightarrow 
T^{b,a}  \otimes T^{k,l}
 \rightarrow
  \bigoplus_{(x,y)\in \wpair_1 (\mathbf{b}, \mathbf{a})} T^{b-1, a-1}  \otimes T^{k,l},
$$
in which the second map is defined as in Definition \ref{defn:traceless}. This can be rewritten as:
$$
0
\rightarrow 
\traceless{b,a}  \otimes T^{k,l}
\rightarrow 
T^{b+k,a+l}  
 \rightarrow
  \bigoplus_{(x,y)\in \wpair_1 (\mathbf{b}, \mathbf{a})} T^{b+k-1, a+l-1}
$$
using the isomorphisms $T^{b,a}  \otimes T^{k,l}\cong T^{b+k,a+l} $ and $T^{b-1, a-1}  \otimes T^{k,l}\cong T^{b+k-1, a+l-1}$.

Evaluating on $V$ (with $\dim V$ sufficiently large) and passing to $\GL(V)$-invariants, this gives the exact sequence 
$$
0
\rightarrow 
(\traceless{b,a} (V) \otimes T^{k,l}(V))^{\GL (V)} 
\rightarrow 
\kk \fb (\mathbf{a+l} , \mathbf{b+k})   
 \rightarrow
  \bigoplus_{(x,y)\in \wpair_1 (\mathbf{b}, \mathbf{a})} 
 \kk \fb (\mathbf{a+l-1} , \mathbf{b+k-1}) ,
$$
using (\ref{eqn:stable_inv_T}); the second map is induced by the contraction maps. Clearly all terms are zero unless  $a+l = b+k$.  

For fixed $(x,y)\in \wpair_1(\mathbf{b}, \mathbf{a})$,  by the first fundamental theorem of invariants the corresponding component 
$$
\kk \fb (\mathbf{a+l} , \mathbf{b+k})   
 \rightarrow
 \kk \fb (\mathbf{a+l-1} , \mathbf{b+k-1}) 
$$
identifies as follows. For a bijection $\phi$, the generator $[\phi]$ is sent to zero if $\phi (x) \neq y$, otherwise it is sent to the generator corresponding to the restriction of $\phi$ to $(\mathbf{a+k}) \backslash \{x\}$, using the order-preserving isomorphisms to identify this as an element of $\fb (\mathbf{a+k-1}, \mathbf{b+l-1})$.

From this one deduces  that the kernel of $\kk \fb (\mathbf{a+l} , \mathbf{b+k})   
 \rightarrow
  \bigoplus_{(x,y)\in \wpair_1 (\mathbf{b}, \mathbf{a})} 
 \kk \fb (\mathbf{a+l-1} , \mathbf{b+k-1}) $ is the span of the classes $[\phi]$ where $\phi : \mathbf{a+l} \stackrel{\cong}{\rightarrow} \mathbf{b+k}$ is such that $\mathbf{a} \subset \mathbf{a+l}$ maps to $\mathbf{k} \subset \mathbf{b+k}$.

One checks that the subset of $\phi$ satisfying the latter condition identifies with $\uwb ((\mathbf{a}, \mathbf{b}), (\mathbf{k}, \mathbf{l}))$. 
 Indeed, this corresponds to the inclusion 
\begin{eqnarray*}
\label{eqn:include_uwb}
\uwb ((\mathbf{a}, \mathbf{b}) , (\mathbf{k},\mathbf{l})) \hookrightarrow \fb (\mathbf{a+l}, \mathbf{b+k})
\end{eqnarray*}
 that sends a morphism represented by $i : \mathbf{a} \hookrightarrow \mathbf{k }$, $j: \mathbf{b} \hookrightarrow \mathbf{l}$, and $\alpha:  \mathbf{l} \backslash j(\mathbf{b}) \cong \mathbf{k} \backslash i(\mathbf{a}) $ to the bijection induced by the bijections $\mathbf{a}\cong i(\mathbf{a})\subset \mathbf{k}$, $\mathbf{l} \supset j(\mathbf{b})\cong \mathbf{b}$, and $\alpha$. 
\end{proof}

Lemma \ref{lem:GL_invariants} does not capture functoriality with respect to $T^{k,l}$, i.e., 
 that a morphism in $\dwb ((\mathbf{k}, \mathbf{l}) , (\m, \n))$ induces a morphism $
T^{k,l} \rightarrow T^{m,n}
$ in $\f (\spmon)$. As should become apparent in the proof of the following Proposition, for this it is more natural to replace $\kk \uwb ((\mathbf{a}, \mathbf{b}) , (\mathbf{k},\mathbf{l}))$ by the isomorphic (using the canonical basis and its dual basis) $ 
\kk^{\dwb ( (\mathbf{k},\mathbf{l}),(\mathbf{a}, \mathbf{b}) )}$, having replaced $\uwb$ by its opposite, $\dwb$ (likewise for $(m,n)$ in place of $(k,l)$).

\begin{prop}
\label{prop:refined_GL_invariants}
For $a,b, \in \nat$, $\psi \in \dwb ((\mathbf{k}, \mathbf{l}), (\m, \n))$, and $V$ a finite-dimensional $\kk$-vector space with $\dim V$ sufficiently large, there is a commutative diagram in $\kk (\sym_b \times\sym_a)$-modules:
\[
\xymatrix{
(\traceless{b,a}(V)  \otimes T^{k,l} (V) )^{\GL (V)}
\ar[r]^(.6)\cong 
\ar[d]_{(\id \otimes \psi)^{\GL (V)} }
&
\kk^{\dwb ( (\mathbf{k},\mathbf{l}),(\mathbf{a}, \mathbf{b}) )}
\ar[d]^{\kk^{\dwb ( \psi,(\mathbf{a}, \mathbf{b}) )}}
\\
(\traceless{b,a}(V)  \otimes T^{m,n} (V) )^{\GL (V)}
\ar[r]_(.6)\cong 
&
\kk^{\dwb ( (\mathbf{m},\mathbf{n}),(\mathbf{a}, \mathbf{b}) )},
}
\]
where the horizontal isomorphisms are induced by Lemma \ref{lem:GL_invariants} and the vertical morphisms by the respective $\kk \dwb$-module structures.
\end{prop}

\begin{proof}
Using the dual basis and the fact that $\dwb$ is the opposite of $\uwb$, there is an isomorphism of vector spaces 
\[
\kk \uwb ((\mathbf{a}, \mathbf{b}) , (\mathbf{k},\mathbf{l}))
\cong 
\kk^{\dwb ( (\mathbf{k},\mathbf{l}),(\mathbf{a}, \mathbf{b}) )}
\]
(respectively for $(m,n)$ in place of $(k,l)$)  and the horizontal isomorphisms of the statement are given by composing the isomorphism  of Lemma \ref{lem:GL_invariants} with this. These isomorphisms determine how the symmetric groups act; from this, it is straightforward to establish the $\kk (\sym_b \times \sym_a)$-equivariance.

To understand the variance with respect to the $\kk \dwb$-module structure of  $T^{\bullet, \bullet}$, without significant loss of generality, we can reduce to the case $a=b=0$ (so that we take $k=l$). The key point is to understand the effect of the trace map 
\[
T^{k,l}
\rightarrow 
T^{k-1, l-1}
\]
(contracting using $\iota\op_{k,l}$) on passage to stable $\GL(V)$-invariants. As in the proof of Lemma \ref{lem:GL_invariants} 
 this gives the surjection 
\begin{eqnarray}
\label{eqn:iota_stable_GL}
\kk \fb (\mathbf{k}, \mathbf{l}) 
\rightarrow 
\kk \fb (\mathbf{k-1}, \mathbf{l-1})
\end{eqnarray}
that sends bijections $\phi$ such that $\phi (l) = k$ to the restriction $\phi|_{\mathbf{l-1}}$ and to zero otherwise. 

Now, there is a canonical bijection $\fb (\mathbf{k}, \mathbf{l}) \cong \dwb ((\mathbf{k}, \mathbf{l}), (\mathbf{0}, \mathbf{0}))$ and likewise for the pair $(k-1, l-1)$. The morphism $\iota_{k,l}\op$ induces 
$
\dwb ((\mathbf{k-1}, \mathbf{l-1}), (\mathbf{0}, \mathbf{0}))
\hookrightarrow 
\dwb ((\mathbf{k}, \mathbf{l}), (\mathbf{0}, \mathbf{0})).
$
Upon $\kk$-linearizing and dualizing, this identifies with (\ref{eqn:iota_stable_GL}).

Putting these points together, the result follows.
\end{proof}

\begin{rem}
\label{rem:GL_invariants_substitute}
Proposition \ref{prop:refined_GL_invariants} gives a counterpart of the natural isomorphism
$$
\hom_{\f (\spmon)} (T^{k,l}, T^{a,b}) \cong \kk \dwb ((\mathbf{k}, \mathbf{l}), (\mathbf{a}, \mathbf{b})).
$$
However, we only consider the full naturality with respect to $(\mathbf{a}, \mathbf{b})$ in the Proposition.
\end{rem}

Using the observation made in the introduction to this section, Proposition \ref{prop:refined_GL_invariants} implies the following result.

\begin{cor}
\label{cor:traceless_to_T}
For $a,b, k,l\in \nat$ and $V$ sufficiently large, there is an isomorphism of $\kk$-vector spaces
$$
(T^{b,a} (V) \otimes T^{k,l}(V)  )^{\GL (V)} 
\cong 
\bigoplus
_{m,n}
\kk \uwb ((\mathbf{m},\mathbf{n}) , (\mathbf{a},\mathbf{b})) \otimes_{\kk (\sym_m \times \sym_n) } \kk ^{\dwb ((\mathbf{k}, \mathbf{l}),(\mathbf{m}, \mathbf{n}))}
$$
and this is natural with respect to $(\mathbf{k}, \mathbf{l})$ in $\dwb$, as interpreted in Proposition \ref{prop:refined_GL_invariants}.
\end{cor}

\begin{rem}
The significance of this result may need some explanation, since at the level of $\kk$-vector spaces, we can use the identification $T^{b,a} \otimes T^{k,l} \cong T^{b+k, a+l}$ and the first fundamental theorem of invariants to identify the left hand side with $\kk \fb (\mathbf{a+l} , \mathbf{b+k})$.  The point is that we have distinguished subsets 
$ \mathbf{l} \subset \mathbf{a+l}$ and $\mathbf{k} \subset \mathbf{b+k}$ and we wish to be able to interpret the functoriality with respect to  $(\mathbf{k}, \mathbf{l})$ in $\dwb$.

Using $T^{b,a}$ instead of $\traceless{b,a}$ means that we allow an element of $\mathbf{b}$ to be paired with one of $\mathbf{a}$. In the expression of the Corollary, this is provided for by the contribution from   $\uwb ((\mathbf{m},\mathbf{n}) , (\mathbf{a},\mathbf{b})) $. If $a - m = b-n = r \in \nat$, this allows for $r$ such pairs. 
\end{rem}

\subsection{Applying $\GL$-invariants}
\label{subsect:apply_GL_invariants}

Using the above,  we have a second relationship with the functor $\kk ^\uwb \otimes_{\kk \uwb} -$. For simplicity, we restrict to functors in $\uwbup$ (introduced in Definition \ref{defn:uwbup_uwbdown}).
 
We consider the functor $\uwbup \rightarrow \f (\tfb)$ given for $F$ in $\uwbup$ by 
\[
(\mathbf{a}, \mathbf{b}) 
\mapsto 
(\traceless{b,a}(V)  \otimes (T^{\bullet , \bullet}(V) \otimes_{\kk \uwb} F) )^{\GL(V)} 
\]
 `taking $V$ sufficiently large'.  To make this precise, we use that an object of $\uwbup$ is a direct sum of objects of the form 
$$\kk \uwb ((\m,\n) , -) \otimes_{\kk(\sym_m \times \sym_n)} M (\m, \n).
$$
 Restricted to such a summand, the above functor is given by 
\[
(\mathbf{a}, \mathbf{b}) 
\mapsto 
\big(\traceless{b,a}(V)  \otimes (T^{m, n}(V) \otimes_{\kk (\sym_m \times \sym_n)}  M (\m, \n)) \big)^{\GL(V)},
\]
for which the condition $\dim V$ sufficiently large has a sense. We stabilize on each direct summand. 

\begin{nota}
\label{nota:tracelessop}
Denote by
\begin{enumerate}
\item 
 $\tracelessop{*,*} : \kk (\tfb) \rightarrow \f (\spmon)$ the functor $ (\mathbf{a}, \mathbf{b}) \mapsto \traceless{b,a}$;
\item 
$
(\tracelessop{*,*}  \otimes (T^{\bullet , \bullet} \otimes_{\kk \uwb} -) )^{\GL} 
: 
\uwbup 
\rightarrow 
\f (\tfb)
$ 
the functor defined by passage to stable invariants (as above).
\end{enumerate}
\end{nota}

\begin{prop}
\label{prop:stable_inv_uwbup}
The stable invariant functor 
\[
(\tracelessop{*,*}  \otimes (T^{\bullet , \bullet} \otimes_{\kk \uwb} -) )^{\GL} 
: 
\uwbup 
\rightarrow 
\f (\tfb)
\]
 is naturally isomorphic to the  composite of $\kk^\uwb \otimes_{\kk\uwb} -: \uwbup \rightarrow \uwbdown \subset \f (\uwb)$ with the restriction functor $\f (\uwb) \rightarrow \f (\tfb)$.
\end{prop}

\begin{proof}
First we check the natural isomorphism on $\f (\tfb)$ by restriction along the induction functor $\kk \uwb \otimes_{\kk (\tfb)} - : \f (\tfb) \rightarrow \uwbup$. We may reduce to considering objects supported on $(\m,\n)$ for some $m,n \in \nat$, say $\kk \uwb ((\m,\n) , -) \otimes _{\kk (\sym_m \times \sym_n)} M(\m, \n)$. Then, applying the functor $
(\tracelessop{*,*}  \otimes (T^{\bullet , \bullet} \otimes_\uwb -) )^{\GL} $ gives 
$$
(\tracelessop{*,*}(V)  \otimes (T^{m, n}(V) \otimes_{\kk (\sym_m \times \sym_n)}  M (\m, \n)) )^{\GL(V)}
$$
for $\dim V$ sufficiently large.

Proposition \ref{prop:refined_GL_invariants} implies that the latter is naturally isomorphic to 
\[
\kk^{\uwb ((*,*) , (\m, \n))} \otimes_{\kk(\sym_m \times \sym_n)}  M (\m, \n).
\]
This is the image of $\kk \uwb ((\m,\n) , -) \otimes _{\kk (\sym_m \times \sym_n)} M(\m, \n)$ under the functor $\kk^\uwb \otimes_{\kk \uwb} -$. This is clearly natural with respect to $M$ considered as a $\kk (\tfb)$-module, so establishes the natural isomorphism restricted to $\f (\tfb)$.

It remains to check the naturality with respect to $\uwbup$; this is where the full naturality of Proposition \ref{prop:refined_GL_invariants} is required. We may restrict without significant loss of generality to considering the behaviour of a morphism in $\uwbup$ of the form 
\begin{eqnarray}
\label{eqn:morphism_uwbup}
\kk \uwb ((\m,\n) , -) \otimes _{\kk (\sym_m \times \sym_n)} M(\m, \n)
\rightarrow 
\kk \uwb ((\mathbf{s},\mathbf{t}) , -) \otimes _{\kk (\sym_s\times \sym_t)} N(\mathbf{s}, \mathbf{t}).
\end{eqnarray}
By Yoneda's lemma, this is equivalent to a morphism of $\kk (\sym_m \times \sym_n)$-modules
\begin{eqnarray}
\label{eqn:morphism_uwbup_equivalent}
M (\m, \n) 
\rightarrow 
\kk \uwb ((\mathbf{s},\mathbf{t}) , (\m, \n)) \otimes _{\kk (\sym_s\times \sym_t)} N(\mathbf{s}, \mathbf{t}).
\end{eqnarray}
(In particular, we may assume that $m-s = n-t \geq 0$.)

Applying the functor $
(\tracelessop{*,*}  \otimes (T^{\bullet , \bullet} \otimes_{\kk\uwb} -) )^{\GL} $ to (\ref{eqn:morphism_uwbup}) gives the composite morphism 
\begin{eqnarray*}
&&(\tracelessop{*,*}(V)  \otimes (T^{m, n}(V) \otimes_{\kk (\sym_m \times \sym_n)}   M (\m, \n)) )^{\GL(V)}
\longrightarrow 
\\
&&(\tracelessop{*,*}(V)  \otimes (T^{m, n}(V) \otimes_{\kk (\sym_m \times \sym_n)}  \kk \uwb ((\mathbf{s},\mathbf{t}) , (\m, \n)) \otimes _{\kk (\sym_s\times \sym_t)} N(\mathbf{s}, \mathbf{t}))^{\GL(V)}
\longrightarrow 
\\
&&(\tracelessop{*,*}(V)  \otimes (T^{s, t}(V)  \otimes _{\kk (\sym_s\times \sym_t)} N(\mathbf{s}, \mathbf{t}))^{\GL(V)}
\end{eqnarray*}
(for $\dim V$ sufficiently large), where the first morphism is induced by (\ref{eqn:morphism_uwbup_equivalent}) and the second is induced by the $\kk \dwb$-module structure of $T^{\bullet, \bullet}$.

To conclude, we use the naturality statement of Proposition \ref{prop:refined_GL_invariants}  with respect to the $\kk \dwb$-module structure. 
\end{proof}

\ 


\newcommand{\etalchar}[1]{$^{#1}$}

\end{document}